\documentclass[reqno,11pt]{amsart}
\usepackage{amsfonts}
\usepackage{a4wide}
\usepackage{color}
\usepackage{mathrsfs}
\usepackage{mathtools}
\usepackage{amsmath}
\usepackage{amssymb}
\usepackage{bbm}
\usepackage{bm}
\usepackage{esint}
\usepackage{nicefrac}
\usepackage{comment}
\numberwithin{equation}{section}
\usepackage[colorlinks,citecolor=green,linkcolor=red]{hyperref}

\usepackage[latin1]{inputenc}
\input xy
\xyoption{all}

\newtheorem{theorem}{Theorem}[section]
\newtheorem{Ca}[theorem]{Corollary}
\newtheorem{Th}[theorem]{Theorem}
\newtheorem{Lm}[theorem]{Lemma}
\newtheorem{Prop}[theorem]{Proposition}
\newtheorem{Def}[theorem]{Definition}
\newtheorem{Example}[theorem]{Example}
\newtheorem{Remark}[theorem]{Remark}
\newtheorem{Problem}[theorem]{Problem}

\title{Traces of weighted Besov spaces to Ahlfors--David regular sets: the limiting case}
\author{Aleksei Y. Chikalov}
\address{Steklov Mathematical Institute of Russian Academy of Sciences}
\email{chikalov.a@phystech.edu}
\begin{document}
\allowdisplaybreaks
\begin{abstract}
 Given $n\in \mathbb{N}$, $p\in [1,\infty)$, and a weight $\gamma$ satisfying the local Muckenhoupt $A_p$ condition, we introduce a weakened version of the Ahlfors--David codimension-$\theta$ regularity condition for Ahlfors--David $d$-regular sets $E\subset\mathbb{R}^n$, where $d\in(0,n)$ and $\theta\in(0,p)$. Under this assumption, we provide a complete intrinsic description of the trace-space of the weighted Besov space $B^{\frac{\theta}{p}}_{p,1}(\mathbb{R}^n,\gamma)$ to $E$. In particular, our results cover the case of power-type weights $\gamma(x)=|x|^\alpha$ with $-n<\alpha<n(p-1)$, $\alpha\neq -(n-1)$, when $E=\mathbb{R}^{n-1}$. This extends earlier results obtained by Haroske and Schmeisser in \cite{Har_Sch}.

\end{abstract}
\maketitle
\tableofcontents
\section{Introduction}
\noindent

The theory of Besov spaces is an important and rapidly growing area of modern analysis. We refer the interested reader to the monographs \cite{Jon, Sawano, Trieb} for a comprehensive treatment of the theory of classical Besov spaces. One of the most fundamental questions concerning Besov spaces, which has attracted considerable attention over the years, is the trace problem, i.e., the problem of finding sharp intrinsic descriptions of trace-spaces of Besov spaces to various closed sets \(E \subset \mathbb{R}^n\). In his pioneering work \cite{Bes_or}, Besov proved that, given $d\in [1, n-1]$, \(p,q \in [1,\infty]\) and \(s \in (\frac{n-d}{p},\infty)\), the trace-space of the Besov space \(B^{s}_{p,q}(\mathbb{R}^n)\) to the plane \(\mathbb{R}^{d}\subset\mathbb{R}^n\) is linearly and continuously isomorphic to \(B^{s-\frac{n-d}{p}}_{p,q}(\mathbb{R}^{d})\). More precisely, every element of \(B^{s}_{p,q}(\mathbb{R}^n)\) has a well-defined trace to \(\mathbb{R}^{d}\) lying in \(B^{s-\frac{n-d}{p}}_{p,q}(\mathbb{R}^{d})\), and conversely, every element of \(B^{s-\frac{n-d}{p}}_{p,q}(\mathbb{R}^{d})\) arises as the trace of some element of \(B^{s}_{p,q}(\mathbb{R}^n)\). Moreover, there exists a bounded linear extension operator \(\operatorname{Ext}:B^{s-\frac{n-d}{p}}_{p,q}(\mathbb{R}^{d})\rightarrow B^{s}_{p,q}(\mathbb{R}^n)\).
\par
The limiting case \(s=\frac{n-d}{p}\) is substantially more delicate. At this critical smoothness, the case \(q=1\) is exceptional, for \(q>1\), the usual trace theorem fails: functions in the corresponding Besov space need not have traces in the Lebesgue-point sense. In the endpoint case \(q=1\), Burenkov and Gol'dman \cite{Bur_art} proved that, for \(p \in [1,\infty)\), the trace-space of \(B^{\frac1p}_{p,1}(\mathbb{R}^{n})\) to \(\mathbb{R}^{n-1}\) can be identified in an appropriate sense with \(L_p(\mathbb{R}^{n-1})\). Another remarkable feature of this limiting case is that the corresponding extension operator has to be nonlinear (see \cite{Bur_art}). A generalization of this result to an arbitrary plane \(\mathbb{R}^d \subset \mathbb{R}^n\) was proved by Gol'dman in \cite{Gol}. In particular, for \(p\in [1,\infty)\), the trace-space of \(B^{\frac{n-d}{p}}_{p,1}(\mathbb{R}^{n})\) to \(\mathbb{R}^{d}\) can be identified with \(L_p(\mathbb{R}^{d})\).
\par
It should be underlined that the limiting phenomenon was first discovered by Gagliardo in the context of Sobolev spaces. In his pioneering work \cite{Gagl}, he proved that, for \(p>1\), the trace-space of the Sobolev space \(W^1_p(\mathbb{R}^n)\) to \(\mathbb{R}^{n-1}\) is linearly and continuously isomorphic to \(B^{1-\frac{1}{p}}_{p,p}(\mathbb{R}^{n-1})\), while in the endpoint case \(p=1\) it can be identified with \(L_1(\mathbb{R}^{n-1})\). Moreover, the necessity of a nonlinear extension operator also appears in the limiting case \(p=1\) (see \cite{koch_sn, Peetre}).
\par
Recently, weighted function spaces have attracted considerable attention due to their applications in the analysis of degenerate and singular elliptic equations and in elliptic and parabolic boundary value problems with inhomogeneous boundary conditions (see, e.g., \cite{Fabes,HumLin} and references therein). It is therefore natural to look at weighted analogues of the trace results discussed above. Traces of weighted Sobolev spaces have been extensively investigated (see, e.g., \cite{shanm, Tyul_s_n, Tyul_s, tyulenev2026}), with both non-limiting and limiting cases receiving substantial attention.
For weighted Besov spaces, the non-limiting trace problem is fairly well understood in the case of traces to hyperplanes and, more generally, to ``flat'' lower-dimensional subsets (see, e.g., \cite{Har_tr, Fraz, Tyul_b}). In contrast, in the context of Ahlfors--David \(d\)-regular sets, the corresponding trace theory is far less understood. 
\par
As far as we know, the results closest to the limiting case of the trace problem considered here are available only for special examples of weights and sets.
For example, using atomic decomposition techniques, Haroske and Schmeisser \cite{Har_Sch} proved that, for \(1<p<\infty\), \(-(n-1)<\alpha<n(p-1)\), and \(\gamma(x)=|x|^\alpha\), \(x\in\mathbb{R}^n\), the trace-space of \(B^{\frac1p}_{p,1}(\mathbb{R}^n,\gamma)\) to \(\mathbb{R}^{n-1}\) can be identified with \(L_p(\mathbb{R}^{n-1},{\bar{\gamma}})\), where \(\bar{\gamma}(x')=\gamma(x',0)\), \(x'\in\mathbb{R}^{n-1}\). Based on a similar technique, Piotrowska \cite{Iwon} obtained a description of traces of weighted Besov spaces to Ahlfors--David \(d\)-regular subsets \(E\subset\mathbb{R}^n\) for distance weights \(\gamma(x)=\operatorname{dist}(x,E)^\alpha\), \(x\in\mathbb{R}^n\), \(\alpha>-(n-d)\). In particular, given \(p\in[1,\infty)\), the trace-space of
\(B^{\frac{\alpha}{p}+\frac{n-d}{p}}_{p,1}(\mathbb{R}^n,\gamma)\)
to \(E\) can be identified with \(L_p(E)\).
\par
The aim of this paper is to extend those limiting trace results to a broader class of weights. Let us emphasize that the range \(-n<\alpha<-(n-1)\) for the power weights \(\gamma(x)=|x|^\alpha\), \(x\in\mathbb{R}^n\), is not covered by the result in \cite{Har_Sch}. This singular range cannot be reached by a straightforward adaptation of the method used in \cite{Har_Sch}. Indeed, if \(-n<\alpha<-(n-1)\), then the pointwise restriction of the weight \(\gamma\) to the hyperplane, \(\bar\gamma(x')=|x'|^\alpha\), \(x'\in\mathbb{R}^{n-1}\), is not integrable in any neighborhood of the origin, and hence \(L_p(\mathbb{R}^{n-1},{\bar{\gamma}})\) is not a suitable candidate for the role of the trace-space. To overcome this difficulty, we develop a different approach based on techniques from the theory of function spaces on metric measure spaces (see, for example, \cite{GKZ, Luk_e, Sak_Sot}). Our method is also inspired by the ideas in \cite{Tyul_s, tyulenev2026}.
\par
Following the monograph of Jonsson and Wallin \cite{Jon}, we consider traces to Ahlfors--David \(d\)-regular sets $E\subset\mathbb{R}^n$. Furthermore, we impose a weakened version of the classical Ahlfors--David codimension-\(\theta\) regularity condition on \(E\) and describe the trace-space of \(B^{\frac{\theta}{p}}_{p,1}(\mathbb{R}^n,\gamma)\) to \(E\), for each \(p\in[1,\infty)\) and each \(\theta\in(0,p)\). In particular, our results include the case \(E=\mathbb{R}^{n-1}\) and \(\gamma(x)=|x|^\alpha\), \(x\in\mathbb{R}^n\), \(-n<\alpha<n(p-1)\), with \(\alpha\neq -(n-1)\). Thus, we obtain a natural extension of the results obtained in \cite{Har_Sch}. The borderline case \(\alpha=-(n-1)\) is not covered either by the methods of \cite{Har_Sch} or by the present paper. This case is critical for our approach: the corresponding weight \(\bar{\gamma}(x')=|x'|^{-(n-1)}\), \(x'\in\mathbb{R}^{n-1}\), on the hyperplane is not integrable in any neighborhood of the origin, while the corresponding singularity is not ``admissible'' in the sense introduced below. The borderline case calls for conceptually new methods and will be considered in our future investigations.
\subsection{Main results}
In order to pose the problem and formulate the main results, we briefly sketch the necessary background (see Section~2 for the detailed exposition).
\par
First we recall the concepts of the Hausdorff measure and Ahlfors--David regular sets. Given \(d\in[0,n]\) and \(E\subset\mathbb{R}^n\), we set
\begin{equation}
    \mathcal{H}^d(E)
    :=\lim_{\delta\to 0}
    \inf\left\{\sum_i r_i^d:
    E\subset \bigcup_i Q_{r_i}(x_i),\ 0<r_i<\delta\right\},
\end{equation}
where, for each $x\in\mathbb{R}^n$ and $r>0$, we put $Q_r(x) := x+rI^n$.
Given a set \(E\subset\mathbb R^n\), we denote by
\(\mathcal H^d\lfloor_E\) the restriction of \(\mathcal H^d\) to \(E\), that is,
\begin{equation}
    (\mathcal H^d\lfloor_E)(G):=\mathcal H^d(G\cap E)
\end{equation}
for each measurable set \(G\subset\mathbb R^n\).
Given $d\in (0, n]$, a closed set \(E\subset\mathbb{R}^n\) is said to be \emph{Ahlfors--David \(d\)-regular} if there exist constants \(C_1,C_2>0\) such that
\begin{equation}
    C_1r^d
    \le \mathcal{H}^d\lfloor_E(Q_r(x))
    \le C_2r^d,
    \qquad \text{for all } (x,r)\in E\times(0,1].
\end{equation}
\par
Given \(d\in(0,n]\) and an Ahlfors--David \(d\)-regular set \(E\subset\mathbb R^n\), including the case \(E=\mathbb R^n\) with \(d=n\), a measurable function \(\gamma:E\to\mathbb R\) is called a \emph{weight} on $E$ if
\(\gamma(x)>0\) for \(\mathcal H^d\lfloor_E\)-almost every \(x\in E\). For such  a weight \(\gamma\), we denote by \(\bm{\gamma}\) the corresponding weighted measure, i.e.,
\begin{equation}
\label{eq.weight_meas_def}
    \bm{\gamma}(G):=\int\limits_G \gamma(x)\,d\mathcal{H}^d\lfloor_E(x)
\end{equation}
for every measurable set \(G\subset\mathbb{R}^n\).
 Given a measurable set \(G\) and \(p\in[1,\infty]\), by \(L_p(G,\gamma)\) we denote the space of all equivalence classes of real-valued measurable functions \(f\) such that
\begin{equation}
    \|f\|_{L_p(G,\gamma)}
    :=\left(\int\limits_G |f(x)|^p\,d\bm{\gamma}(x)\right)^{\frac{1}{p}}<\infty
\end{equation}
with the usual modification when \(p=\infty\). Given a measurable set \(G\) and \(p\in[1,\infty]\), we say that \(f\in L_p^{\operatorname{loc}}(G,\gamma)\) if \(f\in L_p(K,\gamma)\) for each compact set \(K\subset G\). When \(\gamma(x)\equiv 1\), we omit the corresponding symbol, i.e.,
\begin{equation}
L_p(G):=L_p(G,1),
\qquad
L_p^{\operatorname{loc}}(G):=L_p^{\operatorname{loc}}(G,1).
\end{equation}
Finally, if \(G\subset\mathbb{R}^n\) is measurable, \(\bm{\gamma}(G)\in(0,\infty)\), and \(f\in L_1(G,\gamma)\), we set
\begin{equation}
    \fint\limits_G f(x)\,d\bm{\gamma}(x)
    :=\frac{1}{\bm{\gamma}(G)}\int\limits_G f(x)\,d\bm{\gamma}(x).
\end{equation}
\par
Throughout the paper, the symbol \(\overline{\Delta}_t f\), \(t>0\), stands for the mean oscillation of \(f\in L_1^{\operatorname{loc}}(\mathbb{R}^n)\) over a cube with side length $2t$, i.e.,
\begin{equation}
    \overline{\Delta}_t f(x)
    :=\frac{1}{t^n}\int\limits_{tI^n}|f(x+h)-f(x)|\,dh,
\end{equation}
where \(tI^n=\prod_{i=1}^n[-t,t]\), \(t>0\).
Given \(p,q\in[1,\infty]\), \(s\in(0,1)\), and a weight \(\gamma\) on \(\mathbb{R}^n\), by \(B^s_{p,q}(\mathbb{R}^n,\gamma)\) we denote the Besov space consisting of all \(f\in L_1^{\operatorname{loc}}(\mathbb{R}^n)\cap L_p(\mathbb{R}^n,\gamma)\) such that
\begin{equation}
    \|f\|_{b^s_{p,q}(\mathbb{R}^n,\gamma)}
    :=\left(\int\limits_0^1
    \left(t^{-s}\|\overline{\Delta}_t f\|_{L_p(\mathbb{R}^n,\gamma)}\right)^q
    \frac{dt}{t}\right)^{\frac{1}{q}}<+\infty
\end{equation}
with the standard modification when \(q=\infty\). We equip this space with the norm
\begin{equation}
    \|f\|_{B^s_{p,q}(\mathbb{R}^n,\gamma)}
    :=\|f\|_{L_p(\mathbb{R}^n,\gamma)}
    +\|f\|_{b^s_{p,q}(\mathbb{R}^n,\gamma)}, \qquad f\in B^s_{p, q}(\mathbb{R}^n, \gamma).
\end{equation}
\par
Without additional assumptions on \(\gamma\), the study of \(B^s_{p,q}(\mathbb{R}^n,\gamma)\) is very difficult. In this paper, we assume that the weights satisfy the local Muckenhoupt condition. Given \(p\in(1,\infty)\), the symbol \(A_p^{\operatorname{loc}}(\mathbb{R}^n)\) denotes the collection of all weights \(\gamma\) on \(\mathbb{R}^n\) such that
\begin{equation}
    \sup_{Q\subset\mathbb{R}^n:\,\ell(Q)\le 1}
    \left(\fint\limits_Q \gamma(x)\,dx\right)
    \left(\fint\limits_Q \gamma(x)^{-\frac{p'}{p}}\,dx\right)^{\frac{p}{p'}}
    <+\infty,
\end{equation}
where \(p'\) is the conjugate exponent to \(p\), \(Q\) ranges over closed cubes with sides parallel to the coordinate axes, and \(\ell(Q)\) denotes the side length of \(Q\). By \(A_1^{\operatorname{loc}}(\mathbb{R}^n)\) we denote the collection of all weights \(\gamma\) on \(\mathbb{R}^n\) satisfying
\begin{equation}
    \sup_{Q\subset\mathbb{R}^n:\,\ell(Q)\le 1}
    \left(\fint\limits_Q \gamma(x)\,dx\right)
    \left(\operatorname*{ess\,inf}_{x\in Q}\gamma(x)\right)^{-1}
    <+\infty.
\end{equation}
The classes \(A_p^{\operatorname{loc}}(\mathbb{R}^n)\), $p\in [1, \infty)$, were introduced by Rychkov in \cite{Rych} for the study of weighted Besov and Lizorkin--Triebel spaces. These classes are natural generalizations of the famous Muckenhoupt classes \(A_p(\mathbb{R}^n)\), $p\in [1, \infty)$, (see, e.g., \cite[Chapter 5]{Stein}). In contrast to the global classes $A_p(\mathbb{R}^n)$, $p\in[1, \infty)$, local Muckenhoupt classes $A_p^{\operatorname{loc}}(\mathbb{R}^n)$, $p\in[1, \infty)$, allow certain growth at infinity, including exponential growth.
\par
Given \(d\in(0,n)\), an Ahlfors--David \(d\)-regular set \(E\), and \(f\in L_1^{\operatorname{loc}}(\mathbb{R}^n)\), we say that a measurable function \(\phi:E\to\mathbb{R}\) is a \emph{trace} of \(f\) if
\begin{equation}
\label{eq.tr_def}
     \lim_{r\to 0}
     \fint\limits_{Q_r(x)} |f(y)-\phi(x)|\,dy=0,
     \qquad
     \text{for \(\mathcal{H}^d\lfloor_E\)-a.e. } x\in E.
\end{equation}
In this case, the equivalence class of \(\phi\) modulo \(\mathcal H^d\lfloor_E\)-negligible sets is denoted by \(\operatorname{Tr} f\). If \(X\subset L_1^{\operatorname{loc}}(\mathbb R^n)\) is a normed linear space of functions, we say that the trace operator is well defined on \(X\) if every \(f\in X\) has a trace to \(E\) in the above sense. In this case, the mapping 
\begin{equation} \operatorname{Tr}:X\to L_0(E), \qquad f\mapsto \operatorname{Tr}f, 
\end{equation} 
is called the trace operator, where \(L_0(E)\) denotes the space of
all equivalence classes of measurable functions on \(E\) modulo
\(\mathcal H^d\!\lfloor_E\)-negligible sets.
\par
The trace operator need not be well defined on
\(B^{s}_{p,1}(\mathbb R^n,\gamma)\) for an arbitrary
\(A_p^{\operatorname{loc}}\)-weight \(\gamma\) (see Section~3 for details).
Our \textit{first main} result presents a natural sufficient condition for that.
\begin{Th}
    \label{Th.tr_exist_stat}
    Let \(d\in(0,n)\), \(p\in [1, \infty)\), \(\theta\in(0, p)\), and \(\gamma \in A_p^{\operatorname{loc}}(\mathbb{R}^n)\). Assume that \(E\subset\mathbb R^n\) is an Ahlfors--David \(d\)-regular set and that, for
    \(\mathcal H^d\lfloor_E\)-almost every \(x\in E\), the weight \(\gamma\) is
    \(\theta\)-nondegenerate at \(x\), i.e.,
\begin{equation}
\label{eq.Tr_reg}
    \inf_{r\in(0,1)}\frac{\bm{\gamma}(Q_r(x))}{r^{d+\theta}}>0.
\end{equation}
Then, for each \(f\in B^{\frac{\theta}{p}}_{p,1}(\mathbb R^n,\gamma)\), there exists the trace of \(f\) to \(E\) in the sense of \eqref{eq.tr_def}.
\end{Th}
\par
Related questions concerning fine representatives, Lebesgue points, and exceptional sets for Besov spaces are often studied by capacitary methods (see, for example, \cite{ Costea, Li, Netrusov, Nuutinen}). However, those results are not directly applicable to the present weighted Besov spaces, and therefore we give a self-contained proof of Theorem~\ref{Th.tr_exist_stat} in Section~3.
\par
Whenever the trace operator is well defined on
\(B^{\frac{\theta}{p}}_{p,1}(\mathbb R^n,\gamma)\), we define the corresponding
\emph{trace-space} by
\begin{equation}
B^{\frac{\theta}{p}}_{p,1}(\mathbb R^n,\gamma)\big|_E
:=
\left\{
\operatorname{Tr}f:
f\in B^{\frac{\theta}{p}}_{p,1}(\mathbb R^n,\gamma)
\right\}.
\end{equation}
As usual, this space is equipped with the quotient-space norm, i.e.,
\begin{equation}
    \|\phi\|_{B^{\frac{\theta}{p}}_{p,1}(\mathbb R^n,\gamma)\big|_E}
:=
\inf\Bigl\{
\|f\|_{B^{\frac{\theta}{p}}_{p,1}(\mathbb R^n,\gamma)}:
f\in B^{\frac{\theta}{p}}_{p,1}(\mathbb R^n, \gamma),\ \operatorname{Tr} f=\phi
\Bigr\}.
\end{equation}
Now, the trace problem for the Besov space
\(B^{\frac{\theta}{p}}_{p,1}(\mathbb R^n,\gamma)\) can be stated as follows.

\begin{Problem}
Let $E \subset\mathbb{R}^n$ be Ahlfors--David $d$-regular, $d\in (0, n)$. Given $p\in [1, \infty)$, a weight $\gamma\in A_p^{\operatorname{loc}}(\mathbb{R}^n)$, and $\theta\in (0, p)$, assume that the hypotheses of Theorem~\ref{Th.main_st} are satisfied, so that the trace operator is
well defined on \(B^{\frac{\theta}{p}}_{p,1}(\mathbb R^n,{\gamma})\). 
    \begin{enumerate}
        \item Given a function \(\phi:E\to\mathbb R\), find necessary and sufficient conditions for the existence of an extension of $\phi$, i.e., \(f\in B^{\frac{\theta}{p}}_{p,1}(\mathbb R^n,\gamma)\) such that \(\operatorname{Tr}f = \phi\).
        \item Find an intrinsic norm on the trace-space
        \(B^{\frac{\theta}{p}}_{p,1}(\mathbb R^n,\gamma)\big|_E\)
        which is equivalent to the quotient-space norm.
        \item Does there exist a bounded operator
        \begin{equation}
        \operatorname{Ext}:
        B^{\frac{\theta}{p}}_{p,1}(\mathbb R^n,\gamma)\big|_E
        \to
        B^{\frac{\theta}{p}}_{p,1}(\mathbb R^n,\gamma),
        \end{equation}
        called an extension operator, such that
        \begin{equation}
        \operatorname{Tr}\circ\operatorname{Ext}=\operatorname{Id}
        \quad
        \text{on }
        B^{\frac{\theta}{p}}_{p,1}(\mathbb R^n,\gamma)\big|_E?
        \end{equation}
    \end{enumerate}
\end{Problem}
In the present paper, we solve this problem and construct a \emph{bounded} nonlinear extension operator for the class of sets \(E\) that are almost regular with respect to the weight \(\gamma\), in the sense specified below. The nonlinear character of this construction is natural in the limiting case: already in the classical endpoint case considered by Burenkov and Gol'dman \cite{Bur_art}, the corresponding extension operator has to be nonlinear. The same obstruction is present even for the case of power-type weights covered by \cite{Har_Sch}. Indeed, if \(\gamma(x)=|x|^\alpha\), \(x\in\mathbb R^n\), with \(-(n-1)<\alpha<n(p-1)\) and \(p\in(1,\infty)\), then on every cube separated from the origin the weight is comparable to a positive constant. Thus, locally away from the origin, the weighted endpoint trace problem reduces to the classical unweighted one. This indicates that bounded linear extension operators should not be expected in the present limiting setting. However, in the present paper we do not address a general nonexistence theorem for bounded linear right inverses.
\par
In the unweighted case, the critical smoothness in the trace problem is governed by the geometric codimension \(n-d\). In the weighted setting, however, the ambient weight may change the effective codimension of the set. For instance, if $E$ is an Ahlfors--David $d$-regular set, $d \in (0, n)$, and
\(\gamma(x)=\operatorname{dist}(x,E)^\alpha\), \(x\in\mathbb{R}^n\), with $\alpha>-(n-d)$, then, for \(x\in E\), $\bm{\gamma}(Q_r(x))\approx r^{n+\alpha}$, and the relevant codimension is \(\theta=n-d+\alpha\), rather than \(n-d\). Thus, besides the Ahlfors--David regularity of \(E\), one needs a compatibility condition between the ambient weight \(\gamma\) and a boundary measure on \(E\). In the special cases treated in \cite{Har_Sch,Iwon}, this compatibility is encoded in the explicit form of the weights. Here we formulate it abstractly as follows. Given \(\theta\in(0,\infty)\), an Ahlfors--David \(d\)-regular set \(E\), \(d\in(0,n)\), and a weight \(\gamma\) on $\mathbb{R}^n$, we say that \(E\) is \emph{Ahlfors--David codimension-\(\theta\) regular with respect to \(\gamma\)} if there exist constants \(C_1,C_2>0\) and a weight \(\bar{\gamma}\) on \(E\) such that
\begin{equation}
        C_1\,\frac{\bm{\gamma}(Q_r(x))}{r^{\theta}}
        \le
        \bm{\bar{\gamma}}(Q_r(x)\cap E)
        \le
        C_2\,\frac{\bm{\gamma}(Q_r(x))}{r^{\theta}},
        \qquad \text{for all } (x,r)\in E\times(0,1],
\end{equation}
where $\bm{\gamma}$ and $\bm{\bar{\gamma}}$ are weighted measures defined in \eqref{eq.weight_meas_def}.
Related compatibility conditions appear in trace theorems on metric measure spaces (see, e.g., \cite{Luk_e}). In the usual formulation of Ahlfors--David coregularity, one may prescribe a measure on \(E\). In this setting, traces of Besov spaces to regular subsets of metric measure spaces have also been extensively studied (see, e.g., \cite{Gul,Ihnat,Jon,mig,Sak_Sot}). In the present paper we restrict ourselves to the weighted measures $\bm{\bar\gamma}$. This choice is natural for the trace problem considered here, since traces to \(E\) are defined up to \(\mathcal H^d\lfloor_E\)-null sets, and the above measures have the same null sets as \(\mathcal H^d\lfloor_E\).
\par
We aim to relax this condition. We note that a necessary condition for Ahlfors--David codimension-\(\theta\) regularity is
\begin{equation}
    \lim_{r\to0}\frac{\bm{\gamma}(Q_r(x))}{r^\theta}=0
\end{equation}
for all \(x\in E\). Since our main motivating example is $E = \mathbb{R}^{n-1}$, \(\gamma(x)=|x|^\alpha\), where
\(-n<\alpha<-(n-1)\), we need a weakened form of regularity that permits controlled divergences at isolated points. For technical reasons, we impose a specific rate of divergence. More precisely, given \(p\in[1,\infty)\) and \(\theta\in(0,\infty)\), we say that \(x\in\mathbb R^n\) is a point of \emph{\(p\)-rapid singularity of degree \(\theta\)} of the weight \(\gamma\) if
\begin{equation}
        \lim_{r\to0}\frac{\bm{\gamma}(Q_r(x))}{r^\theta}=\infty
\end{equation}
and, for \(p>1\),
\begin{equation}
        \sup_{\rho\in(0,1)}
        \left(\int\limits_{\rho}^1
        \frac{\bm{\gamma}(Q_r(x))}{r^\theta}\frac{dr}{r}\right)^{\frac{1}{p}}
        \left(\int\limits_0^\rho
        \left(\frac{\bm{\gamma}(Q_r(x))}{r^\theta}\right)^{-\frac{1}{p-1}}
        \frac{dr}{r}\right)^{\frac{p-1}{p}}
        <\infty,
\end{equation}
while, for \(p=1\),
\begin{equation}
      \sup_{\rho\in(0,1)}
      \left(\int\limits_\rho^1
      \frac{\bm{\gamma}(Q_r(x))}{r^\theta}\frac{dr}{r}\right)
      \left(\frac{\bm{\gamma}(Q_\rho(x))}{\rho^\theta}\right)^{-1}
      <\infty.
\end{equation}
We denote the set of \(p\)-rapid singular points of degree \(\theta\) of \(\gamma\) by
\(RS_{p,\theta}(\gamma)\).
\par
Let us indicate two model cases. Let \(E=\mathbb R^d\subset\mathbb R^n\), $d\in (0, n)$, and let \(\gamma(x)=|x|^\alpha\). If \(\alpha>-d\), then the restriction
\(\bar\gamma(x)=|x|^\alpha\), \(x\in\mathbb R^d\), is locally integrable on \(E\), and
\(E\) is Ahlfors--David codimension-\((n-d)\) regular with respect to \(\gamma\). In this case $RS_{p, \theta}(\gamma) = \emptyset$. This is
the situation covered by the results of \cite{Har_Sch}.

If, on the other hand, \(-n<\alpha<-d\), then the restriction
\(\bar\gamma(x)=|x|^\alpha\) is not locally integrable near the origin on \(E\). In this
case \(RS_{p,n-d}(\gamma)=\{0\}\), and the usual Ahlfors--David codimension-\((n-d)\) regularity condition
fails at the origin. This is the basic model for the almost regular situation considered
below.
\par
Now we introduce the following relaxation of the Ahlfors--David coregularity condition. We say that an Ahlfors--David \(d\)-regular set \(E\subset\mathbb R^n\) is \emph{Ahlfors--David codimension-\(\theta\) almost regular with respect to \(\gamma\)} if the following conditions hold:
\begin{enumerate}
    \item \(S:=RS_{p,\theta}(\gamma)\cap E\) is finite, possibly empty;
    \item there exists a weight \(\bar{\gamma}\) on \(E\), called the \emph{boundary weight}, and constants \(C_1,C_2>0\) such that
    \begin{equation}
        C_1\,\frac{\bm{\gamma}(Q_r(x))}{r^\theta}
        \le \bm{\bar{\gamma}}(Q_r(x)\cap E)
        \le C_2\,\frac{\bm{\gamma}(Q_r(x))}{r^\theta}, \qquad \begin{gathered} x\in E\setminus S,\\ 0<r\le \min\left\{\frac12\operatorname{dist}(x,S),1\right\}. \end{gathered}
    \end{equation}
\end{enumerate}
In essence, this condition means that the pair \((E,\bar{\gamma})\) satisfies the usual Ahlfors--David coregularity condition away from a finite set of points, while near those points the measure is allowed to have a prescribed type of divergence.
\par
If \(\gamma(x)=|x|^\alpha\), where \(-n<\alpha<-(n-1)\), then a simple example of a function
\(f\in B^{\frac{1}{p}}_{p,1}(\mathbb R^n,\gamma)\) that is equal to \(1\) in a neighborhood of the origin shows that
\(\operatorname{Tr}f\notin L_p(\mathbb R^{n-1},{\bar{\gamma}})\), because
\(\bar{\gamma}\notin L_1^{\operatorname{loc}}(\mathbb R^{n-1})\). Therefore, near rapid singular points the usual local integrability condition has to be replaced by a different requirement. To this end, we introduce a weighted version of the Lebesgue point condition.
\par
First, we define the separation scale of the finite set \(S=RS_{p,\theta}(\gamma)\cap E\) by
\begin{equation}
\rho_S:=
\begin{cases}
\min\left\{1,\frac12\min\limits_{\substack{x,y\in S\\ x\ne y}}|x-y|\right\},
& \text{if } \#S\ge 2,\\[0.2cm]
1,
& \text{if } \#S\le 1.
\end{cases}
\end{equation}
For each subset \(G\subset E\) and each \(\rho>0\), we write $Q^E_\rho(G):=\{x\in E:\operatorname{dist}(x,G)\le \rho\}$. Given an equivalence class of measurable functions \(\phi:E\to\mathbb R\), we say that
\(x_0\in S\) is a \emph{generalized weighted Lebesgue point} of \(\phi\) if there exists \(a\in\mathbb R\) such that $\phi-a\in L_p(Q^E_{\rho_S}(x_0),{\bar{\gamma}})$. Since \(\bm{\bar{\gamma}}(Q_r^E(x_0))=\infty\) for all \(r>0\) (see Section~2 for details), this condition uniquely determines the value \(a\), which will be denoted by \(\phi(x_0)\).
Our \textit{second main} result reads as follows.
\begin{Th}
    \label{Th.main_st}
     Let \(p\in[1,\infty)\) and let \(\gamma\in A_p^{\operatorname{loc}}(\mathbb R^n)\). Given \(d\in(0,n)\) and \(\theta\in(0,p)\), assume that an Ahlfors--David \(d\)-regular set \(E\) is Ahlfors--David codimension-\(\theta\) almost regular with respect to \(\gamma\). Let $S:=RS_{p, \theta}(\gamma)\cap E$.
     A function \(\phi\) belongs to the trace-space
     \(B^{\frac{\theta}{p}}_{p,1}(\mathbb R^n,\gamma)\big|_E\) if and only if
     \begin{enumerate}
         \item \(\phi\in L_p(E\setminus Q^E_{\rho_S}(S),{\bar{\gamma}})\);
         \item \(\phi\) has generalized weighted Lebesgue points at each
         \(x_0\in S\).
     \end{enumerate}
     Moreover, the following equivalence holds:
     \begin{equation}
     \label{eq.main_th_st}
     \begin{split}
         \|\phi\|_{B^{\frac{\theta}{p}}_{p,1}(\mathbb R^n,\gamma)\big|_E}
         \approx
         &\ \|\phi\|_{L_p(E\setminus Q^E_{\rho_S}(S),{\bar{\gamma}})} \\
         &\ +\sum_{x_0\in S}
         \left(
         |\phi(x_0)|
         +
         \|\phi-\phi(x_0)\|_{L_p(Q^E_{\rho_S}(x_0),\bar{\gamma})}
         \right),
     \end{split}
     \end{equation}
     with constants depending only on the structural constants of the assumptions.
\end{Th}

\begin{Remark}
    Assume that \(S=RS_{p, \theta}(\gamma)\cap E=\emptyset\). Then
    \(Q^E_{\rho_S}(S)=\emptyset\), and the previous theorem reduces to the usual weighted \(L_p\)-description of the trace-space. In particular, for \(E=\mathbb R^{n-1}\) and
    \(\gamma(x)=|x|^\alpha\) with \(-(n-1)<\alpha<n(p-1)\), we have the equality $B^{\frac{1}{p}}_{p,1}(\mathbb R^n,\gamma)\big|_{\mathbb R^{n-1}}
    =
    L_p(\mathbb R^{n-1},{\bar{\gamma}})$ as linear spaces, the corresponding norms being equivalent, i.e.,
    \begin{equation}
          \|\phi \|_{B^{\frac{1}{p}}_{p,1}(\mathbb R^n,\gamma)\big|_{\mathbb R^{n-1}}}
          \approx
          \|\phi \|_{L_p(\mathbb R^{n-1},{\bar{\gamma}})}, \qquad \phi \in B^{\frac{1}{p}}_{p,1}(\mathbb R^n,\gamma)\big|_{\mathbb R^{n-1}}. 
    \end{equation}
    Thus, we recover the result of Haroske and Schmeisser \cite{Har_Sch}.
\end{Remark}

\subsection{Plan of the paper}
This paper is organized as follows.
\begin{itemize}
    \item In Section~2, we fix notation, recall standard definitions and known facts, and collect a few auxiliary lemmas.
    \item In Section~3, we prove Theorem~\ref{Th.tr_exist_stat}, showing that, under the nondegeneracy hypothesis, the trace operator is well defined on \(B^{\frac{\theta}{p}}_{p,1}(\mathbb{R}^n,\gamma)\).
    \item In Section~4, we prove the necessity part of Theorem~\ref{Th.main_st} and establish the upper bound in \eqref{eq.main_th_st}.
    \item In Section~5, we prove the sufficiency part of Theorem~\ref{Th.main_st} and obtain the lower bound in \eqref{eq.main_th_st} by constructing a bounded nonlinear extension operator.
    \item In Section~6, we provide several examples illustrating the theory developed in this paper. In particular, we discuss the weights considered in \cite{Har_Sch, Iwon}.
\end{itemize}
\section{Preliminaries}
The aim of this section is to fix notation, recall definitions, and prove some auxiliary results needed later.
\par
Throughout the paper, \(C\) will denote a generic positive constant that may vary from line to line. If the constant \(C\) depends on certain parameters, say \(a,b,c,\ldots\), we indicate this by writing \(C(a,b,c,\ldots)\). The notation \(A\lesssim B\) or \(B\gtrsim A\) means that \(A\le CB\). If \(A\lesssim B\) and \(B\lesssim A\), we write \(A\approx B\).
\subsection{Geometric notation}
Throughout the rest of the paper, we fix \(n\in\mathbb{N}\) with \(n\ge 2\). Points of \(\mathbb{R}^n\) will be denoted by Latin letters such as \(x,y,w\). We denote by \(Q_r(x)\) the closed cube centered at \(x\) of edge length \(2r>0\). For a constant \(c>0\) and a cube \(Q=Q_r(x)\), we write \(cQ:=Q_{cr}(x)\) for the corresponding dilation about the center. Given a nonempty set \(E\subset\mathbb{R}^n\) and \(x\in E\), we write \(Q^E_r(x):=Q_r(x)\cap E\) for the relative cube in \(E\). For convenience in the proofs, we primarily work with cubes rather than balls. Accordingly, unless explicitly stated otherwise, \(|x-y|\) will denote the \(\ell_\infty\)-distance \(|x-y|_\infty\).
\par
Given a nonempty set \(G\subset\mathbb{R}^n\) and \(x\in\mathbb{R}^n\), we set $\operatorname{dist}(x,G):=\inf\{|x-y|:y\in G\}$.
Furthermore, given two nonempty sets \(G_1,G_2\subset\mathbb{R}^n\), we define $\operatorname{dist}(G_1,G_2)
:=
\inf\{|x-y|:x\in G_1,\ y\in G_2\}$.
We use the convention \(\operatorname{dist}(x,\emptyset)=+\infty\). For \(r>0\), we define the closed \(r\)-neighborhood of \(G\) by
$Q_r(G):=\{x\in\mathbb{R}^n:\operatorname{dist}(x,G)\le r\}$. If \(G=\emptyset\), we set \(Q_r(G):=\emptyset\) for each \(r>0\). If \(G\subset E\), then we write
$Q^E_r(G):=Q_r(G)\cap E
=
\{x\in E:\operatorname{dist}(x,G)\le r\}$
for the \(r\)-neighborhood of \(G\) in \(E\).
\par
We shall use Hausdorff measures to control exceptional sets in the proof of the existence of traces. For this reason, we recall the definition and some basic properties of the Hausdorff measure. For \(d\in[0,n]\) and a set \(E\subset\mathbb{R}^n\), define, for \(\delta>0\),
\begin{equation}
    \mathcal{H}^d_{\delta}(E)
    =
    \inf\left\{
    \sum_i r_i^d:
    E\subset \bigcup_i Q_{r_i}(x_i),\ r_i\in(0,\delta)
    \right\},
\end{equation}
where the infimum is taken over all at most countable coverings of \(E\) by cubes \(\{Q_{r_i}(x_i)\}\). The \(s\)-dimensional Hausdorff measure of \(E\) is then
\begin{equation}
    \mathcal{H}^d(E)
    =
    \lim_{\delta\to0}\mathcal{H}^d_{\delta}(E).
\end{equation}
It is well known that, up to a dimensional constant, the \(n\)-dimensional Hausdorff measure coincides with the \(n\)-dimensional Lebesgue measure (see, e.g., \cite[Chapter~2]{law}).
\par
We shall also use Ahlfors--David regular subsets of \(\mathbb{R}^n\).
\begin{Def}
    Given \(d\in(0,n]\), a closed set \(E\subset\mathbb{R}^n\) is called Ahlfors--David \(d\)-regular if there are constants \(C_1,C_2>0\) such that
    \begin{equation}
        C_1r^d
        \le
        \mathcal{H}^d(Q_r(x)\cap E)
        \le
        C_2r^d
    \end{equation}
    for every \(x\in E\) and every \(r\in(0,1]\). We shall use the notation \(\mathcal{H}^d\lfloor_E\) for the restriction of the Hausdorff measure to \(E\).
\end{Def}
\par
The following standard consequence will be used in Section~3, in order to control exceptional sets on an Ahlfors--David \(d\)-regular set \(E\).
\par
\begin{Th}(\cite[Section~2.4.3, Theorem~3]{law})
    \label{Th.leb_points_ref}
    Let \(f\in L_1^{\operatorname{loc}}(\mathbb{R}^n)\), let \(d\in[0,n)\), and define
    \begin{equation}
        \Lambda_d
        =
        \left\{
        x\in\mathbb{R}^n:
        \limsup_{r\to0}
        \frac{1}{r^d}
        \int\limits_{Q_r(x)} |f(y)|\,dy
        >0
        \right\}.
    \end{equation}
    Then \(\mathcal{H}^d(\Lambda_d)=0\).
\end{Th}
\par
We next introduce the weighted measures and the basic properties of \(A_p^{\operatorname{loc}}\)-weights used throughout the paper.
\subsection{Weights}
Given \(d\in(0,n]\) and an Ahlfors--David \(d\)-regular set \(E\subset\mathbb R^n\), including the case \(E=\mathbb R^n\) with \(d=n\), a measurable function \(\gamma:E\to\mathbb R\) is called a weight, or a weight function, if
\(\gamma(x)>0\) for \(\mathcal H^d\lfloor_E\)-almost every \(x\in E\).
Given such a weight \(\gamma\), we denote by \(\bm{\gamma}\) the corresponding weighted measure, i.e.,
\begin{equation}
    \bm{\gamma}(G)
    :=
    \int\limits_G \gamma(x)\,d\mathcal{H}^d\lfloor_E(x)
\end{equation}
for every measurable set \(G\subset E\).
\par
Let \(p\in[1,\infty]\), and let \(\gamma\) be a weight. For a measurable set \(G\subset E\), we denote by \(L_p(G,\gamma)\) the space of all equivalence classes of measurable functions \(f\) such that
\begin{equation}
    \|f\|_{L_p(G,\gamma)}
    :=
    \left(
    \int\limits_G |f(x)|^p\,d\bm{\gamma}(x)
    \right)^{\frac{1}{p}}
    <+\infty
\end{equation}
with the standard modification when \(p=\infty\). If \(\gamma\equiv 1\), we abuse notation and write \(L_p(G)\) instead of \(L_p(G,1)\). As usual, \(L_p^{\operatorname{loc}}(G,\gamma)\) stands for the space of all equivalence classes of measurable functions satisfying \(f\in L_p(K,\gamma)\) for every compact set \(K\subset G\).
\par
Let \(G\subset E\) be a measurable set with \(0<\bm{\gamma}(G)<\infty\). For each \(f\in L_1(G,\gamma)\), we set
\begin{equation}
    \fint\limits_G f(x)\,d\bm{\gamma}(x)
    :=
    \frac{1}{\bm{\gamma}(G)}
    \int\limits_G f(x)\,d\bm{\gamma}(x).
\end{equation}
In the unweighted case \(\gamma\equiv 1\), we occasionally write
\begin{equation}
    f_G
    :=
    \fint\limits_G f(x)\,dx.
\end{equation}
Throughout this paper, we assume that any given weight \(\gamma\) on \(\mathbb R^n\) is locally integrable unless otherwise stated. Note, however, that we will naturally obtain weights on lower-dimensional sets \(E\) which are not locally integrable with respect to \(\mathcal H^d\lfloor_E\).
\par
The main class of weights considered in this paper is the local Muckenhoupt class \(A_p^{\operatorname{loc}}\).
\begin{Def}[\cite{Rych}]
    \label{Def.muck}
     Given \(p\in(1,\infty)\), a weight \(\gamma:\mathbb R^n\to\mathbb R\) is said to satisfy the local \(p\)-Muckenhoupt condition if
    \begin{equation}
    \label{eq.muck_cond}
        \sup_{Q:\,\ell(Q)\le 1}
        \left(
        \fint\limits_Q \gamma(x)\,dx
        \right)
        \left(
        \fint\limits_Q \gamma(x)^{-\frac{p'}{p}}\,dx
        \right)^{\frac{p}{p'}}
        <+\infty,
    \end{equation}
    where \(p'\) is the conjugate exponent to \(p\), i.e.,
    \(\frac{1}{p}+\frac{1}{p'}=1\). The set of all such weights is denoted by
    \(A_p^{\operatorname{loc}}(\mathbb R^n)\).
    If \(p=1\), then \(A_1^{\operatorname{loc}}(\mathbb R^n)\) consists of weights satisfying
    \begin{equation}
    \label{eq.muck_cond1}
        \sup_{Q:\,\ell(Q)\le 1}
        \left(
        \fint\limits_Q \gamma(x)\,dx
        \right)
        \left(
        \operatorname*{ess\,inf}_{x\in Q}\gamma(x)
        \right)^{-1}
        <+\infty.
    \end{equation}
    In what follows,
    \(A_\infty^{\operatorname{loc}}(\mathbb R^n):=
    \bigcup_{p\ge 1}A_p^{\operatorname{loc}}(\mathbb R^n)\).
\end{Def}
\begin{Remark}
\label{Rm.muck_weight}
The particular choice of the scale \(1\) in Definition~\ref{Def.muck} is immaterial. More precisely, replacing the condition \(\ell(Q)\le 1\) in \eqref{eq.muck_cond} or \eqref{eq.muck_cond1} by \(\ell(Q)\le R\), where \(R>0\) is fixed, leads to the same class \(A_p^{\operatorname{loc}}(\mathbb R^n)\), with possibly different constants.
\par
    We shall use the following standard consequences of the \(A_p^{\operatorname{loc}}\)-condition (see, e.g., \cite{Stein} for global $A_p$ weights, the case of local $A_p$ weights follows by the same argument). First, every \(A_p^{\operatorname{loc}}\)-weight is locally doubling: for each \(R>0\), there exists a constant \(C=C(R)>0\) such that
\begin{equation}
\label{eq.double_cond}
    \bm{\gamma}(Q_{2r}(x))
    \le
    C\,\bm{\gamma}(Q_r(x))
\end{equation}
for all \((x,r)\in\mathbb R^n\times(0,R]\).
\par
Second, the \(A_p^{\operatorname{loc}}\)-condition is equivalent to the weighted averaging inequality
\begin{equation}
\label{eq.equiv_muck}
    \left(
    \fint\limits_Q f(x)\,dx
    \right)^p
    \le
    \frac{C}{\bm{\gamma}(Q)}
    \int\limits_Q f(x)^p\,d\bm{\gamma}(x)
\end{equation}
for all nonnegative functions \(f\) and all cubes \(Q\subset\mathbb R^n\) with \(\ell(Q)\le 1\).
\end{Remark}
Although our results apply to a broader class of weights, power weights provide the simplest motivating examples for the regularity assumptions introduced below.
\begin{Example}
    Let \(\gamma(x)=|x|^\alpha\). If \(p\in(1,\infty)\), then
    \(\gamma\in A_p^{\operatorname{loc}}(\mathbb R^n)\) if and only if
    \(-n<\alpha<n(p-1)\). If \(p=1\), then
    \(\gamma\in A_1^{\operatorname{loc}}(\mathbb R^n)\) if and only if
    \(-n<\alpha\le 0\).
\end{Example}
\par
The \(A_p^{\operatorname{loc}}\)-condition controls the ambient weighted measure \(\bm{\gamma}\) in \(\mathbb R^n\). For trace problems, however, one also needs a compatibility condition between the ambient measure \(\bm{\gamma}\) and a measure on the set to which the trace is taken. We introduce the following weighted version of the Ahlfors--David coregularity condition.
\begin{Def}
\label{Def.reg}
    Given \(d\in(0,n)\), \(\theta\in(0,\infty)\), and a weight \(\gamma\) on \(\mathbb R^n\), we say that an Ahlfors--David \(d\)-regular set \(E\) is Ahlfors--David codimension-\(\theta\) regular with respect to \(\gamma\) if there exists a weight
    \(\bar{\gamma}\in L_1^{\operatorname{loc}}(E)\) such that, for some constants \(C_1,C_2>0\),
    \begin{equation}
    \label{eq.reg_cond}
        C_1\frac{\bm{\gamma}(Q_r(x))}{r^\theta}
        \le
        \bm{\bar{\gamma}}(Q^E_r(x))
        \le
        C_2\frac{\bm{\gamma}(Q_r(x))}{r^\theta},
        \qquad
        \text{for all } (x,r)\in E\times(0,1].
    \end{equation}
    The weight \(\bar{\gamma}\) will be called the boundary weight associated with \(\gamma\).
\end{Def}
\begin{Remark}
In the usual measure-theoretic formulation of Ahlfors--David coregularity one prescribes a measure on \(E\) (see, e.g., \cite{Luk_e}). In principle, Definition~\ref{Def.reg} could be reformulated in this way by replacing the weighted measure \(\bm{\bar\gamma}\) with a prescribed measure \(\mu\) on \(E\). However, in that formulation traces and trace spaces would have to be understood modulo \(\mu\)-null sets. Since in this paper traces to \(E\) are defined \(\mathcal H^d\lfloor_E\)-almost everywhere, we restrict ourselves to measures generated by weights \(\bar\gamma\) with respect to \(\mathcal H^d\lfloor_E\). This keeps the boundary measure compatible with the underlying notion of trace.
\end{Remark}
\begin{Example}
    Let \(\gamma(x)=|x|^\alpha\), where \(\alpha>-d\), \(d<n\). Then \(\mathbb R^{d}\subset\mathbb R^n\) is Ahlfors--David codimension-\((n-d)\) regular with respect to \(\gamma\). The corresponding boundary weight \(\bar{\gamma}\) is given by the pointwise restriction of \(\gamma\) to the plane \(\mathbb R^d\).
\end{Example}
The coregularity condition excludes many singular weights. Indeed, if
\begin{equation}
    \lim_{r\to0}
    \frac{\bm{\gamma}(Q_r(x))}{r^\theta}
    =
    \infty,
\end{equation}
then \eqref{eq.reg_cond} cannot hold. For instance, if \(\gamma(x)=|x|^\alpha\), where \(\alpha\in(-n,-d)\), then the plane \(\mathbb R^d \subset \mathbb{R}^n\) is not Ahlfors--David codimension \((n-d)\)-regular with respect to $\gamma$. This motivates the following weakened notion, which allows controlled divergences at isolated points. For technical reasons, we impose a specific rate of divergence.
\begin{Def}
    Given \(p\in[1,\infty)\), \(\theta\in(0,\infty)\), and a weight \(\gamma\) on \(\mathbb R^n\), we say that a point \(x\in\mathbb R^n\) is a point of \(p\)-rapid singularity of degree \(\theta\) of \(\gamma\) if
    \begin{equation}
        \lim_{r\to0}
        \frac{\bm{\gamma}(Q_r(x))}{r^\theta}
        =
        \infty
    \end{equation}
    and, for \(p>1\),
    \begin{equation}
    \label{eq.hard_in_cond}
        \sup_{\rho\in(0,1)}
        \left(
        \int\limits_{\rho}^1
        \frac{\bm{\gamma}(Q_r(x))}{r^\theta}
        \frac{dr}{r}
        \right)^{\frac{1}{p}}
        \left(
        \int\limits_0^\rho
        \left(
        \frac{\bm{\gamma}(Q_r(x))}{r^\theta}
        \right)^{-\frac{1}{p-1}}
        \frac{dr}{r}
        \right)^{\frac{p-1}{p}}
        <\infty,
    \end{equation}
    while, for \(p=1\),
    \begin{equation}
    \label{eq.hard_in_condp1}
      \sup_{\rho\in(0,1)}
      \left(
      \int\limits_\rho^1
      \frac{\bm{\gamma}(Q_r(x))}{r^\theta}
      \frac{dr}{r}
      \right)
      \left(
      \frac{\bm{\gamma}(Q_\rho(x))}{\rho^\theta}
      \right)^{-1}
      <+\infty.
    \end{equation}
    The set of points of \(p\)-rapid singularity of degree \(\theta\) of the weight \(\gamma\) is denoted by \(RS_{p,\theta}({\gamma})\).
\end{Def}
\begin{Remark}
    Condition \eqref{eq.hard_in_cond} guarantees the applicability of a suitable weighted Hardy inequality in the direct trace theorem; see Section~4. In practice, one may replace \eqref{eq.hard_in_cond} with any other sufficient condition that ensures the same Hardy estimate (see, e.g., \cite[Chapter~6]{kuf} and the references therein).
    A simple sufficient condition implying \eqref{eq.hard_in_cond} or \eqref{eq.hard_in_condp1} is
    \begin{equation}
        \frac{\bm{\gamma}(Q_r(x))}{r^\theta}
        \approx
        r^\lambda,
        \qquad
        \lambda<0.
    \end{equation}
   For example, for all \(p\in[1,\infty)\), the origin is a point of \(p\)-rapid singularity of degree \(\theta\) for the power weights \(\gamma(x)=|x|^\alpha\) whenever \(\alpha<\theta-n\).
\end{Remark}
\begin{Def}
    Given \(d\in(0,n)\), \(p\in[1,\infty)\), \(\theta\in(0,\infty)\), and a weight \(\gamma\) on \(\mathbb R^n\), we say that an Ahlfors--David \(d\)-regular set \(E\) is Ahlfors--David codimension-\(\theta\) almost regular with respect to \(\gamma\) if the following conditions hold:
    \begin{enumerate}
        \item \(S:=RS_{p,\theta}(\gamma)\cap E\) is finite, possibly empty;
        \item there exists a weight \(\bar{\gamma}\) on \(E\), called the boundary weight, and positive constants \(C_1,C_2>0\) such that
    \begin{equation}
    \label{eq.alm_reg}
        C_1\frac{\bm{\gamma}(Q_r(x))}{r^\theta}
        \le
        \bm{\bar{\gamma}}(Q^E_r(x))
        \le
        C_2\frac{\bm{\gamma}(Q_r(x))}{r^\theta}
    \end{equation}
    whenever \(x\in E\setminus S\) and
    $r\le
    \min\left\{
    \frac{1}{2}\operatorname{dist}(x,S),1
    \right\}$.
    \end{enumerate}
\end{Def}
In what follows, whenever the parameters \(p,\theta\), the weight \(\gamma\), and the set \(E\) are fixed, we write $S:=RS_{p,\theta}(\gamma)\cap E$ for the corresponding set of rapid singularities.
\par
First, we collect several basic consequences of almost regularity that will be used in Sections~4 and~5.
\begin{Lm}
\label{Lm.alm_reg_bas_pr}
Let \(d\in(0,n)\), \(p\in[1,\infty)\), \(\theta\in(0,\infty)\), and assume that an Ahlfors--David \(d\)-regular set \(E\) is Ahlfors--David codimension-\(\theta\) almost regular with respect to \(\gamma\). Let \(\bar{\gamma}\) be the corresponding boundary weight. Then:
\begin{enumerate}
    \item \(\bar{\gamma}\in L_1^{\operatorname{loc}}(E\setminus S)\);
    \item if \(\bm{\gamma}\) is locally doubling, then \(\bm{\bar{\gamma}}(Q_r^E(x_0))=\infty\) for every
    \(x_0\in S\) and every \(r>0\);
    \item if \(\bm{\gamma}\) is locally doubling, then the measure \(\bm{\bar{\gamma}}\) is locally doubling in
    \(E\setminus S\), i.e., there is a constant \(C>0\) such that, for each \(x\in E\setminus S\),
    \begin{equation}
        \bm{\bar{\gamma}}(Q_r^E(x))
        \le
        C\bm{\bar{\gamma}}(Q^{E}_{\frac{r}{2}}(x))
    \end{equation}
    whenever
    $0<r\le
    \min\left\{
    \frac12\operatorname{dist}(x,S),1
    \right\}$.
\end{enumerate}
In particular, the last two assertions hold whenever
\(\gamma\in A_{\infty}^{\operatorname{loc}}(\mathbb R^n)\).
\end{Lm}

\begin{proof}
The first assertion follows directly from the upper estimate in the definition of almost regularity. Indeed, every compact set \(K\subset E\setminus S\) has positive distance from \(S\), and hence can be covered by finitely many cubes on which the upper estimate applies. Since \(\bm{\gamma}\) is finite on compact subsets of \(\mathbb R^n\), this gives \(\bm{\bar{\gamma}}(K)<\infty\).
\par
Assume now that \(\bm{\gamma}\) is locally doubling and let \(x_0\in S\). Fix \(r>0\). By the Ahlfors--David regularity of \(E\), the point \(x_0\) is not isolated in \(E\). Since \(S\) is finite, for all sufficiently large \(k\in\mathbb N\) one can find a point \(x_k\in E\setminus S\) such that \(x_k\to x_0\) and
$\operatorname{dist}(x_k,S)=|x_k-x_0|$.
Let $\rho_k:=\frac12|x_k-x_0|$. Then, for all sufficiently large \(k\), we have $Q^E_{\rho_k}(x_k)\subset Q^E_r(x_0)$. Applying the lower estimate in the definition of almost regularity to the cubes
\(Q_{\rho_k}(x_k)\) and using the local doubling property of \(\bm{\gamma}\), we obtain
\begin{equation}
    \bm{\bar{\gamma}}(Q_r^E(x_0))
    \gtrsim
    \frac{\bm{\gamma}(Q_{\rho_k}(x_k))}{\rho_k^\theta}
    \gtrsim
    \frac{\bm{\gamma}(Q_{\rho_k}(x_0))}{\rho_k^\theta}
\end{equation}
for all sufficiently large \(k\). Since \(x_0\in S\), the right-hand side tends to infinity as \(k\to\infty\). Therefore $\bm{\bar{\gamma}}(Q_r^E(x_0))=\infty$.
\par
It remains to prove the local doubling property of \(\bm{\bar{\gamma}}\). Let \(x\in E\setminus S\) and let $0<r\le
\min\left\{
\frac12\operatorname{dist}(x,S),1
\right\}$.
Then the almost regularity estimates apply both to \(Q_{\frac{r}{2}}^E(x)\) and to \(Q_r^E(x)\). Therefore, using also the doubling property of \(\bm{\gamma}\), we get
\begin{equation}
    \bm{\bar{\gamma}}(Q_r^E(x))
    \lesssim
    \frac{\bm{\gamma}(Q_r(x))}{r^\theta}
    \lesssim
    \frac{\bm{\gamma}(Q_{\frac{r}{2}}(x))}{(\frac{r}{2})^\theta}
    \lesssim
    \bm{\bar{\gamma}}(Q_{\frac{r}{2}}^E(x)).
\end{equation}
This proves that \(\bm{\bar{\gamma}}\) is locally doubling in \(E\setminus S\).
\end{proof}
\begin{Example}
    Let \(d<n\) and let \(\gamma(x)=|x|^\alpha\), where \(\alpha>-n\), \(\alpha\ne -d\). Then \(\mathbb R^d\subset\mathbb R^n\) is Ahlfors--David codimension-\((n-d)\) almost regular.
The corresponding boundary weight \(\bar{\gamma}\) is given by the pointwise restriction of \(\gamma\) to the plane.
Moreover, for all \(p\in[1,\infty)\),
\begin{equation}
RS_{p,n-d}(\gamma)
=
\begin{cases}
\emptyset, & \text{if } \alpha>-d,\\[2mm]
\{0\}, & \text{if } \alpha\in(-n,-d).
\end{cases}
\end{equation}
\end{Example}
\subsection{Weighted Besov spaces}
We now recall the definition of the weighted Besov spaces used throughout the paper. Since in the sequel we work with the smoothness parameter \(s=\frac{\theta}{p}\), we assume throughout the trace results that \(0<\theta<p\). Thus \(s\in(0,1)\), and first-order differences are sufficient. Given \(f\in L_1^{\operatorname{loc}}(\mathbb{R}^n)\) and \(t>0\), we set
\begin{equation}
    \overline{\Delta}_t f(x)
    :=
    \frac{1}{t^n}
    \int\limits_{tI^n}|f(x+h)-f(x)|\,dh,
\end{equation}
where \(I^n=[-1,1]^n\).
\begin{Def}
    Let \(p,q\in[1,\infty]\), \(s\in(0,1)\), and let \(\gamma\) be a weight on \(\mathbb{R}^n\). The weighted Besov space \(B^s_{p,q}(\mathbb{R}^n,\gamma)\) is the collection of all equivalence classes of functions
    \(f\in L_1^{\operatorname{loc}}(\mathbb{R}^n)\cap L_p(\mathbb{R}^n,\gamma)\)
    such that
    \begin{equation}
        \|f\|_{b^s_{p,q}(\mathbb{R}^n,\gamma)}
        :=
        \left(
        \int_0^1
        \left(
        t^{-s}\|\overline{\Delta}_t f\|_{L_p(\mathbb{R}^n,\gamma)}
        \right)^q
        \frac{dt}{t}
        \right)^{\frac{1}{q}}
        <+\infty
    \end{equation}
    with the usual modification when \(q=\infty\). We equip this space with the norm
    \begin{equation}
        \|f\|_{B^s_{p,q}(\mathbb{R}^n,\gamma)}
        :=
        \|f\|_{L_p(\mathbb{R}^n,\gamma)}
        +
        \|f\|_{b^s_{p,q}(\mathbb{R}^n,\gamma)}, \qquad f\in B^s_{p, q}(\mathbb{R}^n, \gamma).
    \end{equation}
\end{Def}
\begin{Remark}
    If $\bm{\gamma}$ is locally doubling, then the choice of the upper limit \(1\) in the definition of
    \(\|f\|_{b^s_{p,q}(\mathbb R^n,\gamma)}\) is not essential on finite
    scales. More precisely, one may replace the integral over \((0,1)\) by
    the integral over \((0,R)\), where \(R>0\) is fixed, and obtain an
    equivalent norm, with constants depending on \(R\). 
\end{Remark}
\begin{Remark}
\label{Rm.equiv_seminorm}
   For the estimates below, it is convenient to replace the integral in \(t\) by a dyadic sum; see, e.g., \cite[Lemma~2.1]{GKZ}. Namely,
    \begin{equation}
        \label{eq.equiv_besov_norm}
        \|f\|_{b^s_{p,q}(\mathbb{R}^n,\gamma)}
        \approx
        \left(
        \sum_{k\in\mathbb{N}_0}
        \left(
        2^{ks}\|\overline{\Delta}_{2^{-k}}f\|_{L_p(\mathbb{R}^n,\gamma)}
        \right)^q
        \right)^{\frac{1}{q}},
    \end{equation}
    with the standard modification when \(q=\infty\). This equivalence follows from the pointwise inequality on each dyadic interval:
    \begin{equation}
        2^{(k+1)s}\overline{\Delta}_{2^{-k-1}}f(x)
        \le
        2^n t^{-s}\overline{\Delta}_t f(x)
        \le
        2^{2n+s}2^{ks}\overline{\Delta}_{2^{-k}}f(x),
    \end{equation}
    for all \(x\in\mathbb{R}^n\) and all \(t\in[2^{-k-1},2^{-k}]\).
\end{Remark}
The next elementary estimate will be used in the construction of the extension operator, where the extension is defined by means of a smooth partition of unity.
\begin{Lm}
\label{Lm.use_ineq}
Let \(p\ge 1\) and \(\gamma\in A_p^{\operatorname{loc}}(\mathbb{R}^n)\). Then:
    \begin{enumerate}
        \item for all \(f\in L_p(\mathbb{R}^n,\gamma)\) and all \(t\in(0,1]\), the following inequality holds:
        \begin{equation}
        \label{eq.av_lp_est_st}
            \|\overline{\Delta}_t f\|_{L_p(\mathbb{R}^n,\gamma)}
            \lesssim
            \|f\|_{L_p(\mathbb{R}^n,\gamma)}.
        \end{equation}
        \item for all \(f\in C^1(\mathbb{R}^n)\) such that \(|\nabla f|\in L_p(\mathbb{R}^n,\gamma)\), and all \(t\in(0,1]\), the following inequality holds:
        \begin{equation}
            \|\overline{\Delta}_t f\|_{L_p(\mathbb{R}^n,\gamma)}
            \lesssim
            t\||\nabla f|\|_{L_p(\mathbb{R}^n,\gamma)}.
        \end{equation}
    \end{enumerate}
\end{Lm}

\begin{proof}
The first inequality follows from the estimate
\begin{equation}
    \overline{\Delta}_t f(x)
    \lesssim
    \fint\limits_{Q_t(x)}|f(y)|\,dy+|f(x)|
\end{equation}
and Fubini's theorem. Indeed, by \eqref{eq.equiv_muck},
\begin{equation}
\label{eq.av_lp_est}
\begin{split}
\left\|
\fint\limits_{Q_t(\cdot)}|f(y)|\,dy
\right\|_{L_p(\mathbb{R}^n,\gamma)}^p
&\lesssim
\int\limits_{\mathbb{R}^n}
\frac{1}{\bm{\gamma}(Q_t(x))}
\int\limits_{Q_t(x)}
|f(y)|^p\,d\bm{\gamma}(y)\,d\bm{\gamma}(x)
\\
&=
\int\limits_{\mathbb{R}^n}
|f(y)|^p
\int\limits_{Q_t(y)}
\frac{d\bm{\gamma}(x)}{\bm{\gamma}(Q_t(x))}
\,d\bm{\gamma}(y).
\end{split}
\end{equation}
Since \(\gamma\in A_p^{\operatorname{loc}}(\mathbb{R}^n)\), the measure \(\bm{\gamma}\) is locally doubling. Hence, for all \(y\in\mathbb{R}^n\) and all \(x\in Q_t(y)\),
\begin{equation}
    \bm{\gamma}(Q_t(x))
    \approx
    \bm{\gamma}(Q_t(y)).
\end{equation}
Together with \eqref{eq.av_lp_est}, this gives \eqref{eq.av_lp_est_st}.

To prove the second inequality, we use the following estimate; see \cite[p.~215, Corollary~6]{Bur_book}:
\begin{equation}
    |f(x+h)-f(x)|
    \le
    |h|
    \int\limits_0^1 |\nabla f(x+\tau h)|\,d\tau.
\end{equation}
Then Fubini's theorem and elementary integration yield
\begin{equation}
\begin{split}
    \overline{\Delta}_t f(x)
    &\le
    \frac{1}{t^n}
    \int\limits_{tI^n}
    |h|
    \int\limits_0^1|\nabla f(x+\tau h)|\,d\tau\,dh\lesssim
    t
    \int\limits_0^1
    \fint\limits_{Q_{\tau t}(x)}
    |\nabla f(y)|\,dy\,d\tau.
\end{split}
\end{equation}
Furthermore, by Jensen's inequality, and by the weighted averaging inequality \eqref{eq.equiv_muck}, we have
\begin{equation}
\begin{split}
   \left(
   \int\limits_0^1
   \fint\limits_{Q_{\tau t}(x)}
   |\nabla f(y)|\,dy\,d\tau
   \right)^p
   &\lesssim
   \int\limits_0^1
   \left(
   \fint\limits_{Q_{\tau t}(x)}
   |\nabla f(y)|\,dy
   \right)^p
   d\tau
   \\
   &\lesssim
   \int\limits_0^1
   \frac{1}{\bm{\gamma}(Q_{\tau t}(x))}
   \int\limits_{Q_{\tau t}(x)}
   |\nabla f(y)|^p\,d\bm{\gamma}(y)\,d\tau.
\end{split}
\end{equation}
Consequently, applying Fubini's theorem and using the local doubling property of \(\bm{\gamma}\) as above, we obtain
\begin{equation}
    \|\overline{\Delta}_t f\|_{L_p(\mathbb{R}^n,\gamma)}^p
    \lesssim
    t^p
    \int\limits_0^1
    \||\nabla f|\|_{L_p(\mathbb{R}^n,\gamma)}^p
    d\tau
    =
    t^p
    \||\nabla f|\|_{L_p(\mathbb{R}^n,\gamma)}^p.
\end{equation}
The proof is complete.
\end{proof}
\begin{Remark}
\label{Rm.dif_bes}
Let \(p,q\in[1,\infty)\) and \(s\in(0,1)\). Assume that
\(\gamma\in A_p^{\operatorname{loc}}(\mathbb{R}^n)\). Then, for every
\(f\in C^1(\mathbb{R}^n)\) such that \(f,|\nabla f|\in L_p(\mathbb{R}^n,\gamma)\) and every \(\delta\in(0,1]\), Lemma~\ref{Lm.use_ineq} yields
    \begin{equation}
    \begin{split}
        \|f\|_{b^s_{p,q}(\mathbb{R}^n,\gamma)}
        &\lesssim
        \left(
        \int_0^\delta
        \left(
        t^{1-s}\||\nabla f|\|_{L_p(\mathbb{R}^n,\gamma)}
        \right)^q
        \frac{dt}{t}
        \right)^{\frac{1}{q}}
        +
        \left(
        \int_\delta^1
        \left(
        t^{-s}\|f\|_{L_p(\mathbb{R}^n,\gamma)}
        \right)^q
        \frac{dt}{t}
        \right)^{\frac{1}{q}}
        \\
        &\approx
        \delta^{1-s}\||\nabla f|\|_{L_p(\mathbb{R}^n,\gamma)}
        +
        \delta^{-s}\|f\|_{L_p(\mathbb{R}^n,\gamma)}.
    \end{split}
    \end{equation}
\end{Remark}
\par
Finally, we record a pointwise estimate which connects pointwise difference with the Besov oscillation. This estimate will be the main tool in the proof of the generalized Lebesgue point property at rapid singularities.
\begin{Lm}(\cite[Lemma~2.3]{GKZ})
\label{Lm.poinw_dif}
    There exists a positive constant \(C\) such that, for every
    \(f\in L_1^{\operatorname{loc}}(\mathbb{R}^n)\), one can find a set \(N_f\) with
    \(\mathcal{L}^n(N_f)=0\) so that, for every pair of points
    \(x,y\in\mathbb{R}^n\setminus N_f\) with
    \(|x-y|\in[2^{-k-1},2^{-k})\), one has
    \begin{equation}
    \label{eq.point_dif}
        |f(x)-f(y)|
        \le
        C
        \sum_{j=k-2}^{\infty}
        \left(
        \inf_{c\in\mathbb{R}}
        \fint\limits_{Q_{2^{-j}}(x)}
        |f(w)-c|\,dw
        +
        \inf_{c\in\mathbb{R}}
        \fint\limits_{Q_{2^{-j}}(y)}
        |f(w)-c|\,dw
        \right).
    \end{equation}
\end{Lm}
\begin{Remark}
\label{Rm.point_dif_leb}
    In particular, \eqref{eq.point_dif} holds whenever \(x\) and \(y\) are Lebesgue points of \(f\), which follows directly from the proof.
\end{Remark}
\subsection{Special function space}
As mentioned in the introduction, the usual weighted Lebesgue space is no longer suitable as the trace-space in the presence of rapid singular points. The difficulty is localized near these points, where the boundary measure has infinite mass. We therefore introduce a special space of functions which allows finitely many singular points and is adapted to the trace-space.
\begin{Def}
    Let \(d\in(0,n)\), let \(E\) be an Ahlfors--David \(d\)-regular set, and let \(S\subset E\) be a finite set, possibly empty. We put
    \begin{equation}
        \rho_S
        :=
        \begin{cases}
        \displaystyle
        \min\left\{1,\frac12\min\limits_{\substack{x,y\in S\\ x\ne y}}|x-y|\right\},
        & \text{if } \#S\ge 2,\\[2mm]
        1,
        & \text{if } \#S\le 1.
        \end{cases}
    \end{equation}
    Let \(\gamma\) be a weight on \(E\) such that
    \(\gamma\in L_1^{\operatorname{loc}}(E\setminus S)\) and $\bm{\gamma}(Q^E_r(x_0))=\infty$ for every \(r>0\) and every \(x_0\in S\). Given \(p\in[1,\infty)\), we define
    \(\mathfrak{L}_p(E,\gamma,S)\) as the collection of all equivalence classes of measurable functions \(f:E\to\mathbb R\) such that
    \begin{enumerate}
        \item \(f\in L_p(E\setminus Q^E_{\rho_S}(S),\gamma)\);
        \item \(f\) has a generalized weighted Lebesgue point at every \(x_0\in S\), i.e., there exists \(a\in\mathbb R\) such that $f-a\in L_p(Q^E_{\rho_S}(x_0),\gamma)$. In this case the number \(a\) is uniquely determined, and we denote it by \(f(x_0)\).
    \end{enumerate}
    We equip this space with the norm
    \begin{equation}
        \|f\|_{\mathfrak{L}_p(E,\gamma,S)}
        :=
        \|f\|_{L_p(E\setminus Q^E_{\rho_S}(S),\gamma)}
        +
        \sum_{x_0\in S}
        \left(
        |f(x_0)|
        +
        \|f-f(x_0)\|_{L_p(Q^E_{\rho_S}(x_0),\gamma)}
        \right).
    \end{equation}
\end{Def}
\begin{Remark}
    If \(S=\emptyset\), then
    \(\mathfrak{L}_p(E,\gamma,\emptyset)=L_p(E,\gamma)\) with equality of norms.
\end{Remark}
\begin{Remark}
\label{Rm.ren_sp}
    The definition implies that
    \(\mathfrak{L}_p(E,\gamma,S)\subset L_p^{\operatorname{loc}}(E\setminus S,\gamma)\).
    Indeed, let \(K\subset E\setminus S\) be compact. On
    \(K\setminus Q^E_{\rho_S}(S)\), the claim follows from the global condition
    \(f\in L_p(E\setminus Q^E_{\rho_S}(S),\gamma)\). On
    \(K\cap Q^E_{\rho_S}(x_0)\), where \(x_0\in S\), we write
    \(f=(f-f(x_0))+f(x_0)\). The first term belongs to
    \(L_p(Q^E_{\rho_S}(x_0),\gamma)\), while the second one belongs to
    \(L_p(K\cap Q^E_{\rho_S}(x_0),\gamma)\), since
    \(\gamma\in L_1^{\operatorname{loc}}(E\setminus S)\).
\end{Remark}
We shall use the following elementary approximation lemma. Although it is standard, we include the proof for completeness.
\begin{Lm}
    \label{Lm.av_conv}
Let \(G\subset E\) be a measurable subset of an Ahlfors--David \(d\)-regular set \(E\), let \(r_0>0\), and let \(\gamma\in L_1^{\operatorname{loc}}(Q^E_{r_0}(G))\) be a weight which is locally doubling on \(G\) at scales not exceeding \(\frac{r_0}{2}\), i.e., for some constant $C>0$,
\begin{equation}
    \bm{\gamma}(Q_{2t}^E(x)) \le C\bm{\gamma}(Q^E_{t}(x)), \qquad \text{for all } (x, t) \in G \times \left(0, \frac{r_0}{2}\right]
\end{equation}
Given $t\le t_0:=\min\left\{\frac{r_0}{4},1\right\}$
and \(f\in L_p(Q^E_{r_0}(G),\gamma)\), set
\begin{equation}
    \delta_t^{E,p} f(x)
    :=
    \fint\limits_{Q^E_t(x)}
    |f(x)-f(y)|^p\,d\bm{\gamma}(y),
    \qquad x\in G.
\end{equation}
Then, for every \(f\in L_p(Q^E_{r_0}(G),\gamma)\),
\begin{equation}
    \|\delta_t^{E,p} f\|_{L_1(G,\gamma)}\to0
    \quad\text{as } t\to0.
\end{equation}
\end{Lm}
\begin{proof}
    \emph{Step 1.}
    Let \(g\in C_0(Q^E_{r_0}(G))\). By uniform continuity,
    \(\delta_t^{E,p}g\to0\) uniformly on \(G\) as \(t\to0\). Since \(g\) has compact support and \(\gamma\) is locally integrable, this implies
    \begin{equation}
    \label{eq.av_conv1}
         \|\delta_t^{E,p}g\|_{L_1(G,\gamma)}\to0
         \quad\text{as }t\to0.
    \end{equation}
    \emph{Step 2.}
    For \(f\in L_p(Q^E_{r_0}(G),\gamma)\), set
    \begin{equation}
        A_t^{E,p}f(x)
        :=
        \fint\limits_{Q^E_t(x)}
        |f(y)|^p\,d\bm{\gamma}(y).
    \end{equation}
    We prove an \(L_1\)-estimate for \(A_t^{E,p}\). By Fubini's theorem,
    \begin{equation}
        \label{eq.av_conv2}
        \begin{split}
        \|A_t^{E,p} f\|_{L_1(G,\gamma)} =
        \int\limits_{Q^E_t(G)}
        |f(y)|^p
        \int\limits_{Q^E_t(y)\cap G}
        \frac{d\bm{\gamma}(x)}
        {\bm{\gamma}(Q^E_t(x))}
        \,d\bm{\gamma}(y).
        \end{split}
    \end{equation}
    If \(x\in Q^E_t(y)\cap G\), then \(Q^E_t(y)\subset Q^E_{2t}(x)\). Therefore, the local doubling condition gives
    \begin{equation}
        \label{eq.av_conv3}
        \bm{\gamma}(Q^E_t(y))
        \le
        \bm{\gamma}(Q^E_{2t}(x))
        \le
        C\bm{\gamma}(Q^E_t(x)).
    \end{equation}
    Consequently,
    \begin{equation}
        \label{eq.av_conv4}
        \|A_t^{E,p}f\|_{L_1(G,\gamma)}
        \le
        C\|f\|_{L_p(G',\gamma)}^p,
    \end{equation}
    where \(G':=Q^E_{t_0}(G)\).

    \emph{Step 3.}
    Let \(f\in L_p(Q^E_{r_0}(G),\gamma)\) and let \(\varepsilon>0\). Since \(C_0(G')\) is dense in \(L_p(G',\gamma)\), we can choose
    \(g\in C_0(\mathbb R^n)\) such that
    \begin{equation}
    \label{eq.av_conv5}
        \|f-g\|_{L_p(G',\gamma)}\le \varepsilon.
    \end{equation}
    Then, for \(0<t\le t_0\), \(x\in G\), and \(y\in Q^E_t(x)\), we have
    \begin{equation}
        |f(x)-f(y)|^p
        \lesssim
        |f(x)-g(x)|^p
        +
        |g(x)-g(y)|^p
        +
        |f(y)-g(y)|^p.
    \end{equation}
    Consequently,
    \begin{equation}
        \label{eq.av_conv6}
        \|\delta_t^{E,p}f\|_{L_1(G,\gamma)}
        \lesssim
        \|g-f\|_{L_p(G,\gamma)}^p
        +
        \|A_t^{E,p}(g-f)\|_{L_1(G,\gamma)}
        +
        \|\delta_t^{E,p}g\|_{L_1(G,\gamma)}.
    \end{equation}
    Applying \eqref{eq.av_conv4} and \eqref{eq.av_conv5}, we obtain
    \begin{equation}
    \label{eq.av_conv7}
        \|\delta_t^{E,p}f\|_{L_1(G,\gamma)}
        \lesssim
        \varepsilon^p
        +
        \|\delta_t^{E,p}g\|_{L_1(G,\gamma)}.
    \end{equation}
    Taking the limit superior as \(t\to0\) and using Step~1, we get
    \begin{equation}
        \limsup_{t\to0}
        \|\delta_t^{E,p}f\|_{L_1(G,\gamma)}
        \lesssim
        \varepsilon^p.
    \end{equation}
    Since \(\varepsilon>0\) is arbitrary, the claim follows.
\end{proof}
The next proposition records the approximation property of the spaces
\(\mathfrak{L}_p(E,\gamma,S)\), which will be crucial in the construction of the extension operator.
\begin{Prop}
    \label{Prop.cruc_prop}
    Let \(p\in[1,\infty)\), let \(\alpha\in(0,\frac14]\), and let \(S\subset E\) be finite. Assume that \(\gamma\) satisfies the assumptions in the definition of \(\mathfrak{L}_p(E,\gamma,S)\). Assume also that \(\gamma\) is locally doubling in \(E\setminus S\), in the sense that there exists a constant \(C>0\) such that
    \begin{equation}
        \bm{\gamma}(Q^E_{2r}(x))
        \le
        C\bm{\gamma}(Q^E_r(x))
    \end{equation}
    whenever \(x\in E\setminus S\) and $0<r\le \min\left\{\frac12\operatorname{dist}(x,S),1\right\}$.
    Then, for each \(f\in\mathfrak{L}_p(E,\bm\gamma,S)\),
    \begin{equation}
        \lim_{r\to0}
        \left\|\delta_{\alpha r}^{E,p}f\right\|_{L_1(E\setminus Q^E_r(S),\gamma)}
        =
        0.
    \end{equation}
\end{Prop}
\begin{proof}
Take an arbitrary \(f\in\mathfrak{L}_p(E,\gamma,S)\). It is enough to consider $ 0<r<\min\left\{\frac{\rho_S}{4},1\right\}$.
\par
\emph{Step 1.}
We first prove that
\begin{equation}
\label{eq.cr_pr1}
     \lim_{r\to0}
     \left\|\delta_{\alpha r}^{E,p}f\right\|_{L_1(E\setminus Q^E_{\frac{\rho_S}{2}}(S),\gamma)}
     =
     0.
\end{equation}
Indeed, by Remark~\ref{Rm.ren_sp}, we have $f\in L_p(E\setminus Q^E_{\frac{\rho_S}{2}}(S),\gamma)$. Moreover, the weight \(\gamma\) is locally doubling in a neighborhood of
\(E\setminus Q^E_{\frac{\rho_S}{2}}(S)\). Applying Lemma~\ref{Lm.av_conv}, we obtain \eqref{eq.cr_pr1}.
\par
\emph{Step 2.}
Fix \(x_0\in S\) and set $C_{\rho_S,r}(x_0):=Q^E_{\frac{\rho_S}{2}}(x_0)\setminus Q^E_r(x_0)$.
We claim that, for every \(h\in L_p(Q^E_{\rho_S}(x_0),\gamma)\),
\begin{equation}
    \label{eq.cr_prn1}
    \|A_{\alpha r}^{E,p}h\|_{L_1(C_{\rho_S,r}(x_0),\gamma)}
    \le
    C\|h\|_{L_p(Q^E_{\rho_S}(x_0),\gamma)}^p,
\end{equation}
where \(C\) is independent of \(r\) and \(h\). Indeed, by Fubini's theorem,
\begin{equation}
\label{eq.cr_pr2}
\begin{split}
    \|A_{\alpha r}^{E,p}h\|_{L_1(C_{\rho_S,r}(x_0),\gamma)}
    \le
    \int\limits_{Q^E_{\rho_S}(x_0)\setminus Q^E_{(1-\alpha)r}(x_0)}
    |h(y)|^p
    \int\limits_{Q^E_{\alpha r}(y)\cap C_{\rho_S,r}(x_0)}
    \frac{d\bm{\gamma}(x)}
    {\bm{\gamma}(Q^E_{\alpha r}(x))}
    \,d\bm{\gamma}(y).
\end{split}
\end{equation}
Let $y\in Q^E_{\rho_S}(x_0)\setminus Q^E_{(1-\alpha)r}(x_0)$
and $x\in Q^E_{\alpha r}(y)\cap C_{\rho_S,r}(x_0)$. Since \(\operatorname{dist}(x,S)\ge r\) and \(\alpha\le\frac14\), we have $2\alpha r\le \frac12\operatorname{dist}(x,S)$.
Thus the doubling condition in \(E\setminus S\) gives
\begin{equation}
\label{eq.cr_pr3}
    \bm{\gamma}(Q^E_{\alpha r}(y))
    \le
    \bm{\gamma}(Q^E_{2\alpha r}(x))
    \le
    C\bm{\gamma}(Q^E_{\alpha r}(x)).
\end{equation}
Combining \eqref{eq.cr_pr2} and \eqref{eq.cr_pr3}, we obtain \eqref{eq.cr_prn1}.
\par
\emph{Step 3.}
Fix \(x_0\in S\), and put $g:=f-f(x_0)$. Then \(g\in L_p(Q^E_{\rho_S}(x_0),\gamma)\), and $\delta_{\alpha r}^{E,p}f = \delta_{\alpha r}^{E,p}g$.
For \(\rho\in(0,\frac{\rho_S}{2})\), set
\begin{equation}
\label{eq.cr_pr6}
    g_{\rho,x_0}
    :=
    g\chi_{Q^E_{\rho_S}(x_0)\setminus Q^E_\rho(x_0)}.
\end{equation}
By the triangle inequality,
\begin{equation}
\label{eq.cr_pr7}
\begin{split}
    \|\delta_{\alpha r}^{E,p}g\|_{L_1(C_{\rho_S,r}(x_0),\gamma)}
    &\lesssim
    \|g-g_{\rho,x_0}\|_{L_p(C_{\rho_S,r}(x_0),\gamma)}^p+
    \|\delta_{\alpha r}^{E,p}g_{\rho,x_0}\|_{L_1(C_{\rho_S,r}(x_0),\gamma)}
    \\
    &\quad+
    \|A_{\alpha r}^{E,p}(g-g_{\rho,x_0})\|_{L_1(C_{\rho_S,r}(x_0),\gamma)}.
\end{split}
\end{equation}
By \eqref{eq.cr_prn1}, we have
\begin{equation}
\label{eq.cr_pr8}
    \|A_{\alpha r}^{E,p}(g-g_{\rho,x_0})\|_{L_1(C_{\rho_S,r}(x_0),\gamma)}
    \lesssim
    \|g-g_{\rho,x_0}\|_{L_p(Q^E_{\rho_S}(x_0),\gamma)}^p
    \le
    \|g\|_{L_p(Q^E_\rho(x_0),\gamma)}^p.
\end{equation}
Similarly,
\begin{equation}
    \|g-g_{\rho,x_0}\|_{L_p(C_{\rho_S,r}(x_0),\gamma)}^p
    \le
    \|g\|_{L_p(Q^E_\rho(x_0),\gamma)}^p.
\end{equation}
Since \(g\in L_p(Q^E_{\rho_S}(x_0),\gamma)\), the right-hand side tends to \(0\) as \(\rho\to0\).
\par
It remains to estimate $\|\delta_{\alpha r}^{E,p}g_{\rho,x_0}\|_{L_1(C_{\rho_S,r}(x_0),\gamma)}$. For fixed \(\rho>0\), the function \(g_{\rho,x_0}(x)\) vanishes in \(Q^E_\rho(x_0)\). Hence, for all sufficiently small \(r\), the integrand in
\(\delta_{\alpha r}^{E,p}g_{\rho,x_0}\) vanishes whenever
\(x\in Q^E_{\rho/2}(x_0)\). Therefore the norm over
\(C_{\rho_S,r}(x_0)\) can be restricted to the fixed set
\begin{equation}
    Q^E_{\rho_S}(x_0)\setminus Q^E_{\frac{\rho}{2}}(x_0)
    \subset E\setminus S.
\end{equation}
Applying Lemma~\ref{Lm.av_conv} on this fixed set gives
\begin{equation}
    \label{eq.cr_pr10}
    \lim_{r\to0}
    \|\delta_{\alpha r}^{E,p}g_{\rho,x_0}\|_{L_1(C_{\rho_S,r}(x_0),\gamma)}
    =
    0.
\end{equation}
Taking first \(\limsup_{r\to0}\) and then letting \(\rho\to0\) in \eqref{eq.cr_pr7}, we get
\begin{equation}
    \label{eq.cr_prn2}
    \lim_{r\to0}
    \|\delta_{\alpha r}^{E,p}f\|_{L_1(C_{\rho_S,r}(x_0),\gamma)}
    =
    0.
\end{equation}
\par
\emph{Step 4.}
For \(r<\frac{\rho_S}{4}\), the set \(E\setminus Q^E_r(S)\) is covered by $E\setminus Q^E_{\frac{\rho_S}{2}}(S)$ and the finitely many annuli $\{C_{\rho_S,r}(x_0)\}_{x_0\in S}$.
Hence \eqref{eq.cr_pr1} and \eqref{eq.cr_prn2}, together with the finiteness of \(S\), imply
\begin{equation}
    \lim_{r\to0}
    \left\|\delta_{\alpha r}^{E,p}f\right\|_{L_1(E\setminus Q^E_r(S),\gamma)}
    =
    0.
\end{equation}
The proof is complete.
\end{proof}

\section{Existence of traces}
The aim of this section is to establish sufficient conditions for the existence of traces to an Ahlfors--David regular set. Throughout this section, we fix the following data:
\begin{itemize}
    \item a number \(d\in(0,n)\) and an Ahlfors--David \(d\)-regular set \(E\subset\mathbb R^n\);
    \item an integrability parameter \(p\in[1,\infty)\);
    \item a codimension parameter \(\theta\in(0,p)\).
\end{itemize}
We first recall the definition of the trace used in this paper.
\begin{Def}
    Given \(f\in L_1^{\operatorname{loc}}(\mathbb R^n)\), we say that a measurable function \(\phi:E\to\mathbb R\) is a trace of \(f\) to \(E\) if
    \begin{equation}
        \label{eq.tr_defn}
        \lim_{r\to0}
        \fint\limits_{Q_r(x)}
        |f(y)-\phi(x)|\,dy
        =
        0
    \end{equation}
    for \(\mathcal H^d\lfloor_E\)-almost every \(x\in E\). In this case, the equivalence class of \(\phi\) is denoted by \(\operatorname{Tr}f\).
\end{Def}
In particular, if the trace of \(f\) exists, then
\begin{equation}
\label{eq.trace_op}
    \operatorname{Tr}f(x)
    =
    \lim_{r\to0}
    \fint\limits_{Q_r(x)}f(y)\,dy
\end{equation}
for \(\mathcal H^d\lfloor_E\)-almost every \(x\in E\).
\par
In the classical unweighted case, the existence of traces is controlled by the relation between the smoothness of the space and the codimension of the set. For instance, in the endpoint case \(q=1\), the trace of \(B^{\frac{n-d}{p}}_{p,1}(\mathbb R^n)\) to \(\mathbb R^d\) is well defined. Thus the quantity \(n-d\) plays the role of the critical codimension. In the weighted setting, the relevant codimension is encoded by the behaviour of the weighted measure \(\bm\gamma\) near \(E\). This leads to the following nondegeneracy condition.
\begin{Def}
Given a weight \(\gamma\) on \(\mathbb R^n\), for each \(x\in E\), we set
\begin{equation}
    \hat{\gamma}_{\theta}(x)
    :=
    \inf_{r\in(0,1)}
    \frac{\bm{\gamma}(Q_r(x))}{r^{d+\theta}}.
\end{equation}
We say that \(\gamma\) is \(\theta\)-nondegenerate on \(E\) if $\hat{\gamma}_{\theta}(x)>0$ for \(\mathcal H^d\lfloor_E\)-almost every \(x\in E\).
\end{Def}
Our goal is to establish Theorem~\ref{Th.tr_exist_stat}. To this end, we shall prove several auxiliary statements.
\begin{Prop}
\label{Prop.tr_finit_def}
    Let \(\gamma\in A_p^{\operatorname{loc}}(\mathbb{R}^n)\) be \(\theta\)-nondegenerate on \(E\). Then, for each
    \(f\in B^{\frac{\theta}{p}}_{p,1}(\mathbb{R}^n,\gamma)\), there exists a finite dyadic limit
    \begin{equation}
    \label{eq.prop_tr_finit_st}
        \phi(x)
        :=
        \lim_{k\to\infty}
        \fint\limits_{Q_{2^{-k}}(x)} f(y)\,dy
        \in\mathbb{R}
    \end{equation}
    for \(\mathcal{H}^d\lfloor_E\)-almost every \(x\in E\).
\end{Prop}
\begin{proof}
    Take an arbitrary \(f\in B^{\frac{\theta}{p}}_{p,1}(\mathbb{R}^n,\gamma)\) and define, for \(x\in E\),
    \begin{equation}
        \label{eq.tr_ex1}
        \Phi(x)
        :=
        \sum_{k=0}^{\infty}
        \left|
        f_{Q_{2^{-k-1}}(x)}
        -
        f_{Q_{2^{-k}}(x)}
        \right|.
    \end{equation}
    We first prove that \(\Phi(x)<\infty\) for
    \(\mathcal{H}^d\lfloor_E\)-almost every \(x\in E\). This will imply that the sequence
    \(\{f_{Q_{2^{-k}}(x)}\}_{k=0}^{\infty}\) is Cauchy, and hence has a finite limit, for almost every \(x\in E\).
    \par
    For all \(k\in\mathbb{N}_0\), we have $Q_{2^{-k-1}}(x)\subset Q_{2^{-k+1}}(y)$
    whenever \(y\in Q_{2^{-k}}(x)\). Therefore,
\begin{equation}
\label{eq.tr_ex2}
\begin{split}
    \left|
    f_{Q_{2^{-k-1}}(x)}
    -
    f_{Q_{2^{-k}}(x)}
    \right|
    \le
    \fint\limits_{Q_{2^{-k}}(x)}
    \fint\limits_{Q_{2^{-k-1}}(x)}
    |f(w)-f(y)|\,dw\,dy
    \\\lesssim
    \fint\limits_{Q_{2^{-k}}(x)}
    \fint\limits_{Q_{2^{-k+1}}(y)}
    |f(w)-f(y)|\,dw\,dy\approx
    \fint\limits_{Q_{2^{-k}}(x)}
    \overline{\Delta}_{2^{-k+1}}f(y)\,dy.
\end{split}
\end{equation}
Consequently, by the weighted averaging inequality \eqref{eq.equiv_muck},
\begin{equation}
\label{eq.tr_ex3}
    \Phi(x)
    \lesssim
    \sum_{k=0}^{\infty}
    \left(
    \fint\limits_{Q_{2^{-k}}(x)}
    \left(\overline{\Delta}_{2^{-k+1}}f(y)\right)^p
    \,d\bm{\gamma}(y)
    \right)^{\frac{1}{p}}.
\end{equation}
\par
Let $Z_0
    :=
    \{x\in E:\hat{\gamma}_{\theta}(x)=0\}$.
By the \(\theta\)-nondegeneracy of \(\gamma\) on \(E\), we have $\mathcal{H}^d\lfloor_E(Z_0)=0$. For \(R,m\in\mathbb{N}\), set
\begin{equation}
    K_{R,m}
    :=
    Q_R^E(0)
    \cap
    \left\{
    x\in E:\widehat\gamma_\theta(x)\ge \frac1m
    \right\}.
\end{equation}
Then, for each \(x\in K_{R,m}\) and each \(k\in\mathbb{N}_0\),
\begin{equation}
    \label{eq.tr_ex4}
    \frac{\bm{\gamma}(Q_{2^{-k}}(x))}
    {2^{-k(d+\theta)}}
    \ge
    \frac1m.
\end{equation}
Integrating \eqref{eq.tr_ex3} with respect to
\(\mathcal{H}^d\lfloor_E\) over \(K_{R,m}\) and using \eqref{eq.tr_ex4}, we obtain
\begin{equation}
    \label{eq.tr_ex5}
    \begin{split}
    \|\Phi\|_{L_p(K_{R,m},\mathcal{H}^d\lfloor_E)}
    \lesssim
    \sum_{k=0}^{\infty}
    \left(
    m2^{k(d+\theta)}
    \int\limits_{K_{R,m}}
    \int\limits_{Q_{2^{-k}}(x)}
    \left(\overline{\Delta}_{2^{-k+1}}f(y)\right)^p
    d\bm{\gamma}(y)\,
    d\mathcal{H}^d\lfloor_E(x)
    \right)^{\frac1p}.
    \end{split}
\end{equation}
\par
For all \(k\in\mathbb{N}_0\) and all \(g\in L_p(\mathbb{R}^n,\gamma)\), Fubini's theorem and the Ahlfors--David \(d\)-regularity of \(E\) give
\begin{equation}
\label{eq.tr_ex6}
\begin{split}
    &\int\limits_{K_{R,m}}
    2^{k(d+\theta)}
    \int\limits_{Q_{2^{-k}}(x)}
    |g(y)|^p\,d\bm{\gamma}(y)\,
    d\mathcal{H}^d\lfloor_E(x)
    =\\
    \int\limits_{Q_{2^{-k}}(K_{R,m})}
    |g(y)|^p&
    \int\limits_{Q_{2^{-k}}(y)\cap K_{R,m}}
    2^{k(d+\theta)}
    d\mathcal{H}^d\lfloor_E(x)\,
    d\bm{\gamma}(y)
    \lesssim
    2^{k\theta}
    \|g\|_{L_p(\mathbb{R}^n,\bm\gamma)}^p.
\end{split}
\end{equation}
Combining \eqref{eq.tr_ex5} and \eqref{eq.tr_ex6}, we get
\begin{equation}
    \label{eq.tr_ex7}
    \begin{split}
    \|\Phi\|_{L_p(K_{R,m},\mathcal{H}^d\lfloor_E)}
    \lesssim
    m^{\frac1p}
    \sum_{k=0}^{\infty}
    2^{\frac{k\theta}{p}}
    \|\overline{\Delta}_{2^{-k+1}}f\|_{L_p(\mathbb{R}^n,\gamma)}
    \lesssim
    m^{\frac1p}
    \|f\|_{B^{\frac{\theta}{p}}_{p,1}(\mathbb{R}^n,\gamma)}.
    \end{split}
\end{equation}
\par
Let $E_{\Phi} := \{x\in E:\Phi(x)=\infty\}$. By \eqref{eq.tr_ex7}, $\mathcal{H}^d\lfloor_E(E_{\Phi}\cap K_{R,m})=0$ for all \(R,m\in\mathbb N\). Since
\begin{equation}
    E\setminus Z_0
    =
    \bigcup_{R=1}^{\infty}
    \bigcup_{m=1}^{\infty}
    K_{R,m},
\end{equation}
and \(\mathcal{H}^d\lfloor_E(Z_0)=0\), we conclude that $\mathcal{H}^d\lfloor_E(E_{\Phi})=0$.
Thus \(\Phi(x)<\infty\) for \(\mathcal{H}^d\lfloor_E\)-almost every \(x\in E\), and the dyadic averages $f_{Q_{2^{-k}}(x)}$ converge to a finite limit for almost every \(x\in E\). This proves the proposition.
\end{proof}

\begin{Prop}
    \label{Prop.tr_ext_step_2}
    Let \(\gamma\in A_p^{\operatorname{loc}}(\mathbb{R}^n)\) be \(\theta\)-nondegenerate on \(E\), and let
    \(f\in B^{\frac{\theta}{p}}_{p,1}(\mathbb{R}^n,\gamma)\). Let
    \begin{equation}
        \phi(x)
        :=
        \lim_{k\to\infty}
        \fint\limits_{Q_{2^{-k}}(x)} f(y)\,dy
    \end{equation}
    be the finite limit given by Proposition~\ref{Prop.tr_finit_def}. Then \(\phi\) defines the trace of \(f\) to \(E\) in the sense of \eqref{eq.tr_defn}.
\end{Prop}
\begin{proof}
    We need to prove that, for \(\mathcal{H}^d\lfloor_E\)-almost every \(x\in E\),
\begin{equation}
    \lim_{r\to0}
    \fint\limits_{Q_r(x)}
    |f(y)-\phi(x)|\,dy
    =
    0.
\end{equation}
Given \(r\in(0,\frac12]\), let \(k(r)\in\mathbb{N}\) be the unique integer such that $r\in[2^{-k(r)-1},2^{-k(r)})$. Then, for each \(x\in E\),
\begin{equation}
\label{eq.tr_ex8}
    \fint\limits_{Q_r(x)}
    |f(y)-\phi(x)|\,dy
    \le
    \fint\limits_{Q_r(x)}
    |f(y)-f_{Q_{2^{-k(r)}}(x)}|\,dy
    +
    |f_{Q_{2^{-k(r)}}(x)}-\phi(x)|.
\end{equation}
\par
By the definition of \(\phi\),
\begin{equation}
\label{eq.tr_ex9}
     |f_{Q_{2^{-k(r)}}(x)}-\phi(x)|
     \to0
     \quad\text{as } r\to0
\end{equation}
for \(\mathcal{H}^d\lfloor_E\)-almost every \(x\in E\). 
Set
\begin{equation}
    g(x)
    :=
    \sum_{k=0}^{\infty}
    2^{\frac{k\theta}{p}}
    \overline{\Delta}_{2^{-k}}f(x).
\end{equation}
By Remark~\ref{Rm.equiv_seminorm}, we have $g\in L_p(\mathbb{R}^n,\gamma)$.
Moreover,
\begin{equation}
\label{eq.tr_ex10}
\begin{split}
    \fint\limits_{Q_r(x)}
    |f(y)-f_{Q_{2^{-k(r)}}(x)}|\,dy
    \lesssim
    \fint\limits_{Q_r(x)}
    \overline{\Delta}_{2^{-k(r)+1}}f(y)\,dy
    \lesssim
    2^{-\frac{k(r)\theta}{p}}
    \fint\limits_{Q_r(x)}
    g(y)\,dy.
\end{split}
\end{equation}
Using the weighted averaging inequality \eqref{eq.equiv_muck}, we obtain
\begin{equation}
    \label{eq.tr_ex11}
    \begin{split}
    \fint\limits_{Q_r(x)}
    |f(y)-f_{Q_{2^{-k(r)}}(x)}|\,dy
    &\lesssim
    \left(
    \frac{2^{-k(r)\theta}}{\bm{\gamma}(Q_r(x))}
    \int\limits_{Q_r(x)}
    |g(y)|^p\,d\bm{\gamma}(y)
    \right)^{\frac1p}
    \\
    &\lesssim
    \left(
    \frac{1}{\hat{\gamma}_{\theta}(x)}
    \right)^{\frac1p}
    \left(
    \frac{1}{r^d}
    \int\limits_{Q_r(x)}
    |g(y)|^p\,d\bm{\gamma}(y)
    \right)^{\frac1p}.
    \end{split}
\end{equation}
Here we used that \(r\approx 2^{-k(r)}\) and the definition of
\(\hat\gamma_\theta(x)\). Since \(|g|^p\gamma\in L_1(\mathbb{R}^n)\), Theorem~\ref{Th.leb_points_ref} applied to the function \(|g|^p\gamma\) gives
\begin{equation}
    \frac{1}{r^d}
    \int\limits_{Q_r(x)}
    |g(y)|^p\,d\bm{\gamma}(y)
    \to0
    \quad\text{as }r\to0
\end{equation}
for \(\mathcal{H}^d\lfloor_E\)-almost every \(x\in E\). Since
\(\hat\gamma_\theta(x)>0\) for \(\mathcal{H}^d\lfloor_E\)-almost every \(x\in E\), the right-hand side of \eqref{eq.tr_ex11} tends to \(0\) for almost every \(x\in E\). Combining this with \eqref{eq.tr_ex8} and \eqref{eq.tr_ex9}, we obtain
\begin{equation}
    \lim_{r\to0}
    \fint\limits_{Q_r(x)}
    |f(y)-\phi(x)|\,dy
    =
    0
\end{equation}
for \(\mathcal{H}^d\lfloor_E\)-almost every \(x\in E\). Hence \(\phi\) defines the trace of \(f\) to \(E\).
\end{proof}
Now Theorem~\ref{Th.tr_exist_stat} follows immediately from
Propositions~\ref{Prop.tr_finit_def} and~\ref{Prop.tr_ext_step_2}.
\begin{Ca}
\label{Ca.coreg_tr}
    Let \(\gamma\in A^{\operatorname{loc}}_p(\mathbb R^n)\), and assume that \(E\) is Ahlfors--David codimension-\(\theta\) almost regular with respect to \(\gamma\). Then, for every
    \(f\in B^{\frac{\theta}{p}}_{p,1}(\mathbb R^n,\gamma)\), the trace of \(f\) to \(E\) exists in the sense of \eqref{eq.tr_defn}.
\end{Ca}
\begin{proof}
Since \(S:=RS_{p,\theta}(\gamma)\cap E\) is finite, it has zero
\(\mathcal H^d\lfloor_E\)-measure. By the almost regularity condition
\eqref{eq.alm_reg}, for every \(x\in E\setminus S\)
and all sufficiently small \(r>0\), we have
\begin{equation}
    \frac{\bm{\gamma}(Q_r(x))}{r^{d+\theta}}
    \gtrsim
    \frac{\bm{\bar{\gamma}}(Q^E_r(x))}{r^d}.
\end{equation}
By Lemma~\ref{Lm.alm_reg_bas_pr}, we have $\bar{\gamma}\in
    L_1^{\operatorname{loc}}(E\setminus S)$. Since \(\bar{\gamma}(x)>0\) for \(\mathcal H^d\lfloor_E\)-almost every \(x\in E\), the Lebesgue differentiation theorem on the Ahlfors--David regular measure space \(E\) gives
\begin{equation}
    \liminf_{r\to0}\frac{\bm{\bar{\gamma}}(Q^E_r(x))}{r^d}>0
\end{equation}
for \(\mathcal H^d\lfloor_E\)-almost every
\(x\in E\setminus S\). Hence
\begin{equation}
    \liminf_{r\to0}
    \frac{\bm{\gamma}(Q_r(x))}{r^{d+\theta}}
    >0
    \qquad
    \text{for }\mathcal H^d\lfloor_E\text{-a.e. }
    x\in E.
\end{equation}
Therefore, \(\hat{\gamma}_{\theta}(x)>0\) for
\(\mathcal H^d\lfloor_E\)-almost every \(x\in E\). Thus \(\gamma\) is
\(\theta\)-nondegenerate on \(E\), and the result follows from
Theorem~\ref{Th.tr_exist_stat}.
\end{proof}
\par
Under the assumptions of Theorem~\ref{Th.tr_exist_stat}, each
\(f\in B^{\frac{\theta}{p}}_{p,1}(\mathbb R^n,\gamma)\) admits a precise representative \(\widetilde f\) of \(f\) such that
\begin{equation}
    \operatorname{Tr}f(x)
    =
    \widetilde f(x),
    \qquad
    \text{for }\mathcal H^d\lfloor_E\text{-a.e. }x\in E.
\end{equation}
In particular, \(\mathcal H^d\lfloor_E\)-almost every \(x\in E\) is a Lebesgue point of \(\widetilde f\) in the sense of \eqref{eq.tr_defn}. Until the end of the paper, we identify every
\(f\in B^{\frac{\theta}{p}}_{p,1}(\mathbb R^n,\gamma)\) with its precise representative \(\widetilde f\).
\section{The direct trace theorem}
The aim of this section is to prove the direct part of the trace theorem. More precisely, we prove that $\operatorname{Tr}f \in \mathfrak{L}_p\left(E,{\bar{\gamma}},RS_{p,\theta}(\gamma)\cap E)\right)$
for each \(f\in B^{\frac{\theta}{p}}_{p,1}(\mathbb R^n,\gamma)\), and establish the boundedness of the trace operator.
\par
Throughout this section, we fix the following data:
\begin{itemize}
    \item a number \(d\in(0,n)\) and an Ahlfors--David \(d\)-regular set \(E\subset\mathbb R^n\);
    \item an integrability parameter \(p\in[1,\infty)\);
    \item a codimension parameter \(\theta\in(0,p)\);
    \item a weight \(\gamma\in A_p^{\operatorname{loc}}(\mathbb R^n)\) such that \(E\) is Ahlfors--David codimension-\(\theta\) almost regular with respect to \(\gamma\).
\end{itemize}
We set $S:=RS_{p,\theta}(\gamma)\cap E$ and $\rho_0:=\rho_S$.
We divide the proof of the direct trace theorem into several parts, stated below as lemmas.
\begin{Lm}
\label{Lm.0dir}
    For each \(f\in B^{\frac{\theta}{p}}_{p,1}(\mathbb R^n,\gamma)\), one has $\operatorname{Tr}f \in L_p(E\setminus Q_{\rho_0}^E(S),{\bar{\gamma}})$.
    Moreover, there is a constant \(C>0\), depending only on the fixed data, such that
    \begin{equation}
    \label{eq.lm1_dir_st}
        \|\operatorname{Tr}f\|_{L_p(E\setminus Q_{\rho_0}^E(S),{\bar{\gamma}})}
        \le
        C
        \|f\|_{B^{\frac{\theta}{p}}_{p,1}(\mathbb R^n,\gamma)}.
    \end{equation}
\end{Lm}
\begin{proof}
    Choose \(k\in\mathbb N_0\) such that $\rho_0\in(2^{-k},2^{-k+1}]$,
    and put $G_k:=E\setminus Q^E_{2^{-k}}(S)$.
    Since $E\setminus Q_{\rho_0}^E(S)\subset G_k$, it is enough to prove that \(\operatorname{Tr}f\in L_p(G_k,{\bar{\gamma}})\) and to obtain the corresponding estimate on \(G_k\).
    \par
    By the definition of the trace operator and the precise representative chosen in the previous section, for \(\mathcal H^d\lfloor_E\)-almost every \(x\in E\), and hence for \(\bm{\bar{\gamma}}\)-almost every \(x\in G_k\), we have
    \begin{equation}
        \operatorname{Tr}f(x)
        =
        f_{Q_{2^{-k-2}}(x)}
        +
        \sum_{j=k+2}^{\infty}
        \left(
        f_{Q_{2^{-j-1}}(x)}
        -
        f_{Q_{2^{-j}}(x)}
        \right).
    \end{equation}
    By Minkowski's inequality and the same estimate as in Proposition~\ref{Prop.tr_finit_def}, we obtain
    \begin{equation}
        \begin{split}
            \|\operatorname{Tr}f\|_{L_p(G_k,{\bar{\gamma}})}
            &\lesssim
            \left(
            \int\limits_{G_k}
            \fint\limits_{Q_{2^{-k-2}}(x)}
            |f(y)|^p\,d\bm{\gamma}(y)\,
            d\bm{\bar{\gamma}}(x)
            \right)^{\frac1p}
            \\
            &\quad+
            \sum_{j=k+2}^{\infty}
            \left(
            \int\limits_{G_k}
            \fint\limits_{Q_{2^{-j}}(x)}
            \bigl(\overline{\Delta}_{2^{-j+1}}f(y)\bigr)^p
            \,d\bm{\gamma}(y)\,
            d\bm{\bar{\gamma}}(x)
            \right)^{\frac1p}.
        \end{split}
    \end{equation}
    \par
    We now prove the estimate
    \begin{equation}
    \label{eq.tr_loc_aux}
        \int\limits_{G_k}
        \fint\limits_{Q_{2^{-j}}(x)}
        |h(y)|^p\,d\bm{\gamma}(y)\,
        d\bm{\bar{\gamma}}(x)
        \lesssim
        2^{j\theta}
        \|h\|_{L_p(\mathbb R^n,\gamma)}^p
    \end{equation}
    for every \(j\ge k+2\) and every \(h\in L_p(\mathbb R^n,\gamma)\). Indeed, by Fubini's theorem,
    \begin{equation}
        \label{eq.tr_intn1}
        \begin{split}
        \int\limits_{G_k}
        \fint\limits_{Q_{2^{-j}}(x)}
        |h(y)|^p\,d\bm{\gamma}(y)\,
        d\bm{\bar{\gamma}}(x)=
        \int\limits_{Q_{2^{-j}}(G_k)}
        |h(y)|^p
        \int\limits_{Q^E_{2^{-j}}(y)\cap G_k}
        \frac{d\bm{\bar{\gamma}}(x)}
        {\bm{\gamma}(Q_{2^{-j}}(x))}
        \,d\bm{\gamma}(y).
        \end{split}
    \end{equation}
    Let \(x^*\in Q^E_{2^{-j}}(y)\cap G_k\). Then $Q^E_{2^{-j}}(y)\cap G_k
        \subset Q^E_{2^{-j+1}}(x^*)$. Moreover, since \(x^*\in G_k\), we have $\operatorname{dist}(x^*,S)\ge 2^{-k}$.
    Since \(j\ge k+2\), it follows that $2^{-j+1}
        \le \frac12\operatorname{dist}(x^*,S)$.
    Hence the almost regularity condition \eqref{eq.alm_reg} applies to
    \(Q^E_{2^{-j+1}}(x^*)\). Together with the local doubling property of
    \(\bm{\gamma}\), this gives
    \begin{equation}
        \begin{split}
        \int\limits_{Q^E_{2^{-j}}(y)\cap G_k}
        \frac{d\bm{\bar{\gamma}}(x)}
        {\bm{\gamma}(Q_{2^{-j}}(x))}
        \lesssim
        \frac{\bm{\bar{\gamma}}(Q^E_{2^{-j+1}}(x^*))}
        {\bm{\gamma}(Q_{2^{-j+1}}(x^*))}
        \lesssim
        2^{j\theta}.
        \end{split}
    \end{equation}
    Substituting this estimate into \eqref{eq.tr_intn1}, we obtain
    \eqref{eq.tr_loc_aux}.
    \par
    Applying \eqref{eq.tr_loc_aux} with \(j=k+2\) and \(h=f\), and then with
    \(h=\overline{\Delta}_{2^{-j+1}}f\), \(j\ge k+2\), we get
    \begin{equation}
    \label{eq.tr_loc_int}
    \begin{split}
         \|\operatorname{Tr}f\|_{L_p(G_k,{\bar{\gamma}})}
         \lesssim
         2^{\frac{k\theta}{p}}
         \|f\|_{L_p(\mathbb R^n,\gamma)}
         +
         \sum_{j=k+2}^{\infty}
         2^{\frac{j\theta}{p}}
         \|\overline{\Delta}_{2^{-j+1}}f\|_{L_p(\mathbb R^n,\gamma)}
         \lesssim
         2^{\frac{k\theta}{p}}
         \|f\|_{B^{\frac{\theta}{p}}_{p,1}(\mathbb R^n,\gamma)}.
    \end{split}
    \end{equation}
    Since \(k\) is fixed by \(\rho_0\), this proves \eqref{eq.lm1_dir_st}.
\end{proof}
\begin{Lm}
\label{Lm.2dir}
     For every \(f\in B^{\frac{\theta}{p}}_{p,1}(\mathbb{R}^{n},\gamma)\) and each
     \(x_0\in S\), the point \(x_0\) is a Lebesgue point of the precise representative of \(f\). More precisely, the limit
    \begin{equation}
        \operatorname{Tr}f(x_0)
        :=
        \lim_{r\to0}
        \fint\limits_{Q_r(x_0)}f(y)\,dy
    \end{equation}
    exists and is finite, and
    \begin{equation}
        \lim_{r\to0}
        \fint\limits_{Q_r(x_0)}
        |f(y)-\operatorname{Tr}f(x_0)|\,dy
        =
        0.
    \end{equation}
    Moreover, there exists a positive constant \(C>0\), depending only on the fixed data, such that
    \begin{equation}
    \label{eq.lm2_dir_st}
       \sum_{x_0\in S}|\operatorname{Tr}f(x_0)|
       \le
       C
       \|f\|_{B^{\frac{\theta}{p}}_{p,1}(\mathbb{R}^n,\gamma)}.
    \end{equation}
\end{Lm}
\begin{proof}
If \(S=\emptyset\), there is nothing to prove. Thus assume that \(S\neq\emptyset\).
\par
Since \(S\) is finite and each point \(x_0\in S\) is a point of \(p\)-rapid singularity of degree \(\theta\), there exists a constant \(C_S>0\) such that
\begin{equation}
\label{eq.tr_dir7}
    \frac{r^\theta}{\bm{\gamma}(Q_r(x_0))}
    \le
    C_S
\end{equation}
for all \(x_0\in S\) and all \(r\in(0,1]\).
\par
Fix \(x_0\in S\). We first prove that the sequence $\left\{f_{Q_{2^{-k}}(x_0)}\right\}_{k=1}^{\infty}$ is Cauchy. Indeed, for every \(l>k\), we have
\begin{equation}
\label{eq.tr_dir5}
\begin{split}
    \left|
    f_{Q_{2^{-k}}(x_0)}
    -
    f_{Q_{2^{-l}}(x_0)}
    \right|
    \le
    \sum_{j=k}^{l-1}
    \left|
    f_{Q_{2^{-j}}(x_0)}
    -
    f_{Q_{2^{-j-1}}(x_0)}
    \right|
    \lesssim
    \sum_{j=k}^{l-1}
    \fint\limits_{Q_{2^{-j}}(x_0)}
    \overline{\Delta}_{2^{-j+1}}f(x)\,dx.
\end{split}
\end{equation}
Using the weighted averaging inequality \eqref{eq.equiv_muck}, we obtain
\begin{equation}
\label{eq.tr_dir6}
\begin{split}
    \left|
    f_{Q_{2^{-k}}(x_0)}
    -
    f_{Q_{2^{-l}}(x_0)}
    \right|
    &\lesssim
    \sum_{j=k}^{l-1}
    2^{\frac{j\theta}{p}}
    \left({2^{-j\theta}}
    \fint\limits_{Q_{2^{-j}}(x_0)}
    \bigl(\overline{\Delta}_{2^{-j+1}}f(x)\bigr)^p
    \,d\bm{\gamma}(x)
    \right)^{\frac1p}
    \\
    &\lesssim
    \sum_{j=k}^{l-1}
    2^{\frac{j\theta}{p}}
    \|\overline{\Delta}_{2^{-j+1}}f\|_{L_p(\mathbb R^n,\gamma)}.
\end{split}
\end{equation}
The last series converges by the definition of
\(B^{\frac{\theta}{p}}_{p,1}(\mathbb R^n,\gamma)\). Hence
\(\{f_{Q_{2^{-k}}(x_0)}\}_{k=1}^{\infty}\) is Cauchy, and therefore
\begin{equation}
    \operatorname{Tr}f(x_0)
    :=
    \lim_{k\to\infty}
    f_{Q_{2^{-k}}(x_0)}
\end{equation}
exists and is finite.
\par
We next prove that \(x_0\) is a Lebesgue point of the precise representative. Let
\(r\in(0,\frac12]\), and let \(k(r)\in\mathbb N\) be the unique integer such that $r\in[2^{-k(r)-1},2^{-k(r)})$.
Then
\begin{equation}
\begin{split}
    \fint\limits_{Q_r(x_0)}
    |f(x)-\operatorname{Tr}f(x_0)|\,dx
    \le
    \fint\limits_{Q_r(x_0)}
    |f(x)-f_{Q_{2^{-k(r)}}(x_0)}|\,dx
    +
    |f_{Q_{2^{-k(r)}}(x_0)}-\operatorname{Tr}f(x_0)|.
\end{split}
\end{equation}
The second term tends to \(0\) by the definition of
\(\operatorname{Tr}f(x_0)\). For the first term, using the inclusion $Q_{2^{-k(r)}}(x_0)\subset Q_{2^{-k(r)+1}}(x)$ for $x\in Q_r(x_0)$,
we obtain
\begin{equation}
\begin{split}
    \fint\limits_{Q_r(x_0)}
    |f(x)-f_{Q_{2^{-k(r)}}(x_0)}|\,dx
    \lesssim
    \fint\limits_{Q_r(x_0)}
    \overline{\Delta}_{2^{-k(r)+1}}f(x)\,dx
    \lesssim\\
    \left(
    \fint\limits_{Q_r(x_0)}
    \bigl(\overline{\Delta}_{2^{-k(r)+1}}f(x)\bigr)^p
    \,d\bm{\gamma}(x)
    \right)^{\frac1p}
    \lesssim
    2^{k(r)\frac{\theta}{p}}\|\overline{\Delta}_{2^{-k(r)+1}}f\|_{L_p(\mathbb R^n,\gamma)}.
\end{split}
\end{equation}
Since $2^{\frac{k\theta}{p}} \|\overline{\Delta}_{2^{-k}}f\|_{L_p(\mathbb R^n,\gamma)}$ is summable in \(k\), we have $2^{k\frac{\theta}{p}}\|\overline{\Delta}_{2^{-k}}f\|_{L_p(\mathbb R^n,\gamma)}
    \to0
    \qquad\text{as }k\to\infty$.
Thus
\begin{equation}
    \fint\limits_{Q_r(x_0)}
    |f(x)-\operatorname{Tr}f(x_0)|\,dx
    \to0
    \qquad\text{as }r\to0.
\end{equation}
Therefore \(x_0\) is a Lebesgue point of \(f\).
\par
It remains to prove \eqref{eq.lm2_dir_st}. For each \(x_0\in S\), we have
\begin{equation}
\begin{split}
    |\operatorname{Tr}f(x_0)|
    &\le
    |f_{Q_1(x_0)}|
    +
    \sum_{j=0}^{\infty}
    \left|
    f_{Q_{2^{-j}}(x_0)}
    -
    f_{Q_{2^{-j-1}}(x_0)}
    \right|
    \\
    &\lesssim
    \|f\|_{L_p(\mathbb R^n,\gamma)}
    +
    \sum_{j=0}^{\infty}
    2^{\frac{j\theta}{p}}
    \|\overline{\Delta}_{2^{-j+1}}f\|_{L_p(\mathbb R^n,\gamma)}.
\end{split}
\end{equation}
Here we used the weighted averaging inequality for the first term and the estimate above for the differences. Summing over the finite set \(S\), we get
\begin{equation}
\label{eq.tr_dir_n2}
    \sum_{x_0\in S}
    |\operatorname{Tr}f(x_0)|
    \lesssim
    \|f\|_{L_p(\mathbb R^n,\gamma)}
    +
    \sum_{j=0}^{\infty}
    2^{\frac{j\theta}{p}}
    \|\overline{\Delta}_{2^{-j}}f\|_{L_p(\mathbb R^n,\gamma)}
    \lesssim
    \|f\|_{B^{\frac{\theta}{p}}_{p,1}(\mathbb R^n,\gamma)}.
\end{equation}
\end{proof}
We recall that \(x_0\in S\) is a generalized weighted Lebesgue point of a function \(\phi\) if $\phi-\phi(x_0)\in L_p(Q^E_{\rho_0}(x_0),{\bar{\gamma}})$. It remains to show that, for every
\(f\in B^{\frac{\theta}{p}}_{p,1}(\mathbb R^n,\gamma)\), the trace
\(\operatorname{Tr}f\) has generalized weighted Lebesgue points at each
\(x_0\in S\). To this end, we introduce some auxiliary notation. 
\par
Let \(k_0\in\mathbb N_0\) be the unique integer such that $\rho_0\in(2^{-k_0-1},2^{-k_0}]$.
For each \(x_0\in S\) and each \(k\ge k_0\), we set
\begin{equation}
    B_k(x_0)
    :=
    Q^E_{\rho_0}(x_0)\setminus Q^E_{2^{-k}}(x_0),
    \qquad
    A_k(x_0)
    :=
    Q^E_{2^{-k}}(x_0)\setminus Q^E_{2^{-k-1}}(x_0).
\end{equation}
Furthermore, for \(k\ge0\), we define
\begin{equation}
    \gamma_k^\theta(x_0)
    :=
    \frac{\bm{\gamma}(Q_{2^{-k}}(x_0))}{2^{-k\theta}}.
\end{equation}
\par
We start with several auxiliary properties.
\begin{Lm}
    \label{Lm.an_h}
    The following assertions hold.
    \begin{enumerate}
        \item For each \(x_0\in S\) and each \(k\ge k_0\),
        \begin{equation}
        \label{eq.meas_an}
            \bm{\bar{\gamma}}(A_k(x_0))
            \lesssim
            \gamma_k^\theta(x_0).
        \end{equation}
        \item For each \(x_0\in S\), the sequence
        \(\{\gamma_k^\theta(x_0)\}_{k\ge0}\) satisfies the discrete Hardy condition. Namely, if \(p>1\), then
        \begin{equation}
        \label{eq.hard_cond_disc}
            \sup_{k\in\mathbb N_0}
            \left(
            \sum_{j=0}^k\gamma_j^\theta(x_0)
            \right)^{\frac1p}
            \left(
            \sum_{j=k}^{\infty}
            \left(\gamma_j^\theta(x_0)\right)^{-\frac{1}{p-1}}
            \right)^{\frac{p-1}{p}}
            <\infty,
        \end{equation}
        while, if \(p=1\), then
        \begin{equation}
        \label{eq.hard_cond_discp1}
            \sup_{k\in\mathbb N_0}
            \left(
            \sum_{j=0}^k\gamma_j^\theta(x_0)
            \right)
            \left(\gamma_k^\theta(x_0)\right)^{-1}
            <\infty.
        \end{equation}
    \end{enumerate}
\end{Lm}
\begin{proof}
We first prove \eqref{eq.meas_an}. Fix \(x_0\in S\) and \(k\ge k_0\).
The annulus $A_k(x_0)$ can be covered by a finite number \(N=N(n)\) of relative cubes $\{Q^E_{\kappa 2^{-k}}(x_i)\}_{i=1}^{N}$,
where \(\kappa>0\) is chosen sufficiently small and \(x_i\in E\cap A_k(x_0)\). Since \(S\) is finite and \(x_0\) is isolated in \(S\), choosing \(\kappa\) sufficiently small gives $\kappa 2^{-k}\le\frac12\operatorname{dist}(x_i,S)$
for all such \(x_i\) and all \(k\ge k_0\). Hence the almost regularity condition
\eqref{eq.alm_reg} applies to each cube \(Q^E_{\kappa 2^{-k}}(x_i)\). Therefore,
using also the local doubling property of \(\bm{\gamma}\), we obtain
\begin{equation}
\begin{split}
    \bm{\bar{\gamma}}(A_k(x_0))
    \le
    \sum_{i=1}^N
    \bm{\bar{\gamma}}(Q^E_{\kappa 2^{-k}}(x_i))
    \lesssim
    \sum_{i=1}^N
    \frac{\bm{\gamma}(Q_{\kappa 2^{-k}}(x_i))}
    {2^{-k\theta}}
    \lesssim
    \frac{\bm{\gamma}(Q_{2^{-k}}(x_0))}
    {2^{-k\theta}}
    =
    \gamma_k^\theta(x_0).
\end{split}
\end{equation}
This proves \eqref{eq.meas_an}.

We now prove the discrete Hardy condition. Put
\begin{equation}
    \Psi_{x_0}(r)
    :=
    \frac{\bm{\gamma}(Q_r(x_0))}{r^\theta}.
\end{equation}
Since \(\gamma\in A_p^{\operatorname{loc}}(\mathbb R^n)\), the measure
\(\bm{\gamma}\) is locally doubling. Hence \(\Psi_{x_0}(r)\) is comparable to
\(\gamma_k^\theta(x_0)\) for \(r\in[2^{-k-1},2^{-k}]\). Consequently,
\begin{equation}
    \int_{2^{-k-1}}^{2^{-k}}
    \Psi_{x_0}(r)\,\frac{dr}{r}
    \approx
    \gamma_k^\theta(x_0),
\end{equation}
and, for \(p>1\),
\begin{equation}
    \int_{2^{-k-1}}^{2^{-k}}
    \Psi_{x_0}(r)^{-\frac{1}{p-1}}\,\frac{dr}{r}
    \approx
    \left(\gamma_k^\theta(x_0)\right)^{-\frac{1}{p-1}}.
\end{equation}
Thus the integral Hardy condition \eqref{eq.hard_in_cond} implies
\eqref{eq.hard_cond_disc}. The case \(p=1\) is obtained in the same way from
\eqref{eq.hard_in_condp1}. This completes the proof.
\end{proof}
\begin{Lm}
    \label{Lm.ren_dir}
    For every \(f\in B^{\frac{\theta}{p}}_{p,1}(\mathbb R^n,\gamma)\), the trace
    \(\operatorname{Tr}f\) has generalized weighted Lebesgue points at each
    \(x_0\in S\). Moreover, there is a constant \(C>0\) such that
    \begin{equation}
    \label{eq.sing_norm}
        \sum_{x_0\in S}
        \|\operatorname{Tr}f-\operatorname{Tr}f(x_0)\|_{L_p(Q^E_{\rho_0}(x_0),{\bar{\gamma}})}
        \le
        C
        \|f\|_{B^{\frac{\theta}{p}}_{p,1}(\mathbb R^n,\gamma)}.
    \end{equation}
\end{Lm}
\begin{proof}[Proof of Lemma~\ref{Lm.ren_dir}]
If \(S=\emptyset\), there is nothing to prove. Thus assume that \(S\neq\emptyset\). Fix \(f\in B^{\frac{\theta}{p}}_{p,1}(\mathbb R^n,\gamma)\) and
\(x_0\in S\).
For \(x\in Q^E_{\rho_0}(x_0)\setminus\{x_0\}\), let \(k(x)\in\mathbb N_0\) be the unique integer such that $|x-x_0|\in[2^{-k(x)-1},2^{-k(x)})$.
\par
\emph{Step 1.}
Let \(x\in Q^E_{\rho_0}(x_0)\) be a Lebesgue point of the precise representative of \(f\). This holds for \(\mathcal H^d\lfloor_E\)-almost every \(x\in E\) by Corollary~\ref{Ca.coreg_tr}. By Lemma~\ref{Lm.2dir}, \(x_0\) is also a Lebesgue point of \(f\). Hence Lemma~\ref{Lm.poinw_dif} and Remark~\ref{Rm.point_dif_leb} give
\begin{equation}
\label{eq.tr_dir_nnn1}
\begin{split}
    |\operatorname{Tr}f(x)-\operatorname{Tr}f(x_0)|
    \lesssim
    \sum_{j=k(x)-2}^{\infty}
    \biggl(
    \inf_{c\in\mathbb R}
    \fint\limits_{Q_{2^{-j}}(x)}
    |f(y)-c|\,dy
    +
    \inf_{c\in\mathbb R}
    \fint\limits_{Q_{2^{-j}}(x_0)}
    |f(y)-c|\,dy
    \biggr).
\end{split}
\end{equation}
Fix \(\kappa\in(0,2^{-5}]\). 
For every \(w\in Q_{\kappa2^{-j}}(x)\), the inclusion $Q_{2^{-j}}(x)\subset Q_{2^{-j+1}}(w)$ holds. Therefore,
\begin{equation}
\label{eq.tr_dir_nnn2}
\begin{split}
    \inf_{c\in\mathbb R}
    \fint\limits_{Q_{2^{-j}}(x)}
    |f(y)-c|\,dy
    \le
    \fint\limits_{Q_{2^{-j}}(x)}
    |f(y)-f(w)|\,dy
    \lesssim
    \overline{\Delta}_{2^{-j+1}}f(w).
\end{split}
\end{equation}
Averaging this estimate over \(w\in Q_{\kappa2^{-j}}(z)\) and applying
\eqref{eq.equiv_muck}, we get
\begin{equation}
    \inf_{c\in\mathbb R}
    \fint\limits_{Q_{2^{-j}}(x)}
    |f(y)-c|\,dy
    \lesssim  \left(
    \fint\limits_{Q_{\kappa2^{-j}}(x)}
    \bigl(\overline{\Delta}_{2^{-j+1}}f(w)\bigr)^p
    \,d\bm{\gamma}(w)
    \right)^{\frac1p} =:
    f_j(x).
\end{equation}
Thus, for \(\mathcal H^d\lfloor_E\)-almost every
\(x\in Q^E_{\rho_0}(x_0)\),
\begin{equation}
\label{eq.tr_dir_point_est}
\begin{split}
    |\operatorname{Tr}f(x)-\operatorname{Tr}f(x_0)|
    &\lesssim
    \sum_{j=k(x)-2}^{\infty} f_j(x)
    +
    \sum_{j=k(x)-2}^{\infty} f_j(x_0)
    \\
    &\le
    \sum_{j=k_0-2}^{\infty} f_j(x)\chi_{B_{j+3}(x_0)}(x)
    +
    \sum_{j=k(x)-2}^{\infty} f_j(x_0).
\end{split}
\end{equation}
\par
By Minkowski's inequality and the decomposition into annuli \(A_k(x_0)\), we obtain
\begin{equation}
\label{eq.tr_dir_n10}
\begin{split}
    \|\operatorname{Tr}f-\operatorname{Tr}f(x_0)\|_{L_p(Q_{\rho_0}^E(x_0),{\bar{\gamma}})}
    \lesssim
    \sum_{j=k_0-2}^{\infty}
    \|f_j\chi_{B_{j+3}(x_0)}\|_{L_p(Q_{\rho_0}^E(x_0),{\bar{\gamma}})}
    \\+
    \left(
    \sum_{k=k_0}^{\infty}
    \bm{\bar{\gamma}}(A_k(x_0))
    \left(
    \sum_{j=k-2}^{\infty} f_j(x_0)
    \right)^p
    \right)^{\frac1p}=:S_1+S_2.
\end{split}
\end{equation}

\emph{Step 2}. First, we estimate $S_1$. For each fixed \(j\ge k_0-2\), Fubini's theorem gives
\begin{equation}
\label{eq.tr_dir_nnn3}
\begin{split}
    &\|f_j\chi_{B_{j+3}(x_0)}\|_{L_p(Q_{\rho_0}^E(x_0),{\bar{\gamma}})}^p=
    \int\limits_{B_{j+3}(x_0)}
    \fint\limits_{Q_{\kappa2^{-j}}(x)}
    \bigl(\overline{\Delta}_{2^{-j+1}}f(y)\bigr)^p
    \,d\bm{\gamma}(y)\,d\bm{\bar{\gamma}}(x)
    \\
    &\quad=
    \int\limits_{\mathbb{R}^n}
    \bigl(\overline{\Delta}_{2^{-j+1}}f(y)\bigr)^p
    \int\limits_{Q^E_{\kappa2^{-j}}(y)\cap B_{j+3}(x_0)}
    \frac{d\bm{\bar{\gamma}}(x)}
    {\bm{\gamma}(Q_{\kappa2^{-j}}(x))}
    \,d\bm{\gamma}(y).
\end{split}
\end{equation}
Assume that the intersection \(Q^E_{\kappa2^{-j}}(y)\cap B_{j+3}(x_0)\neq\emptyset\) is nonempty, and choose $x^*\in Q^E_{\kappa2^{-j}}(y)\cap B_{j+3}(x_0)$.
Then $Q^E_{\kappa2^{-j}}(y)\cap B_{j+3}(x_0)
    \subset
    Q^E_{2\kappa2^{-j}}(x^*)$.
Moreover, by the definition of \(B_{j+3}(x_0)\) and by the choice
\(\kappa\le 2^{-5}\), the cube \(Q^E_{2\kappa2^{-j}}(x^*)\) stays away from \(S\) at the scale at which the almost regularity condition applies. Hence, using \eqref{eq.alm_reg} and the local doubling property of \(\bm{\gamma}\), we obtain
\begin{equation}
\label{eq.inner_s1}
\begin{split}
    \int\limits_{Q^E_{\kappa2^{-j}}(y)\cap B_{j+3}(x_0)}
    \frac{d\bm{\bar{\gamma}}(x)}
    {\bm{\gamma}(Q_{\kappa2^{-j}}(x))}
    \lesssim
    \frac{
    \bm{\bar{\gamma}}(Q^E_{2\kappa2^{-j}}(x^*))
    }
    {
    \bm{\gamma}(Q_{2\kappa2^{-j}}(x^*))
    }
    \lesssim
    2^{j\theta}.
\end{split}
\end{equation}
Substituting \eqref{eq.inner_s1} into \eqref{eq.tr_dir_nnn3}, we get
\begin{equation}
    \|f_j\chi_{B_{j+3}(x_0)}\|_{L_p(Q_{\rho_0}^E(x_0),{\bar{\gamma}})}
    \lesssim
    2^{\frac{j\theta}{p}}
    \|\overline{\Delta}_{2^{-j+1}}f\|_{L_p(\mathbb R^n,\gamma)}.
\end{equation}
Consequently,
\begin{equation}
\label{eq.tr_dir_nn1}
    S_1
    \lesssim
    \sum_{j=k_0-2}^{\infty}
    2^{\frac{j\theta}{p}}
    \|\overline{\Delta}_{2^{-j+1}}f\|_{L_p(\mathbb R^n,\gamma)}
    \lesssim
    \|f\|_{B^{\frac{\theta}{p}}_{p,1}(\mathbb R^n,\gamma)}.
\end{equation}

\emph{Step 3}. Second, we estimate $S_2$.
By \eqref{eq.meas_an}, $\bm{\bar{\gamma}}(A_k(x_0))
    \lesssim
    \gamma_k^\theta(x_0)$.
Moreover, by the local doubling property of \(\bm{\gamma}\),
\begin{equation}
    f_j(x_0)
    \lesssim
    \left(\gamma_j^\theta(x_0)\right)^{-\frac1p}
    2^{\frac{j\theta}{p}}
    \|\overline{\Delta}_{2^{-j+1}}f\|_{L_p(\mathbb R^n,\gamma)}.
\end{equation}
Therefore,
\begin{equation}
\label{eq.s2_before_hardy}
\begin{split}
    S_2
    &\lesssim
    \left(
    \sum_{k=k_0}^{\infty}
    \gamma_k^\theta(x_0)
    \left(
    \sum_{j=k-2}^{\infty}
    \left(\gamma_j^\theta(x_0)\right)^{-\frac1p}
    2^{\frac{j\theta}{p}}
    \|\overline{\Delta}_{2^{-j+1}}f\|_{L_p(\mathbb R^n,\gamma)}
    \right)^p
    \right)^{\frac1p}.
\end{split}
\end{equation}
By Lemma~\ref{Lm.an_h}, the sequence
\(\{\gamma_j^\theta(x_0)\}_{j\ge0}\) satisfies the discrete Hardy condition.
Hence the weighted discrete Hardy inequality gives
\begin{equation}
\label{eq.tr_dir_nn2}
\begin{split}
    S_2
    &\lesssim
    \left(
    \sum_{j=k_0-2}^{\infty}
    \left(
    2^{\frac{j\theta}{p}}
    \|\overline{\Delta}_{2^{-j+1}}f\|_{L_p(\mathbb R^n,\gamma)}
    \right)^p
    \right)^{\frac1p}
    \\
    &\lesssim
    \sum_{j=k_0-2}^{\infty}
    2^{\frac{j\theta}{p}}
    \|\overline{\Delta}_{2^{-j+1}}f\|_{L_p(\mathbb R^n,\gamma)}
    \lesssim
    \|f\|_{B^{\frac{\theta}{p}}_{p,1}(\mathbb R^n,\gamma)}.
\end{split}
\end{equation}
For \(p=1\), the same conclusion follows from the \(p=1\) version of the discrete Hardy inequality, using \eqref{eq.hard_cond_discp1}.

Combining \eqref{eq.tr_dir_n10}, \eqref{eq.tr_dir_nn1}, and
\eqref{eq.tr_dir_nn2}, we obtain $\operatorname{Tr}f-\operatorname{Tr}f(x_0)
    \in
    L_p(Q^E_{\rho_0}(x_0),{\bar{\gamma}})$
and
\begin{equation}
    \|\operatorname{Tr}f-\operatorname{Tr}f(x_0)\|_{L_p(Q^E_{\rho_0}(x_0),{\bar{\gamma}})}
    \lesssim
    \|f\|_{B^{\frac{\theta}{p}}_{p,1}(\mathbb R^n,\gamma)}.
\end{equation}
Summing over the finite set \(S\) gives \eqref{eq.sing_norm}. The proof is complete.
\end{proof}
\begin{Prop}
\label{Prop.direct_trace}
    For every \(f\in B^{\frac{\theta}{p}}_{p,1}(\mathbb R^n,\gamma)\), one has $\operatorname{Tr}f
        \in
        \mathfrak L_p(E,{\bar{\gamma}},S)$.
    Moreover, there is a constant \(C>0\), independent of \(f\), such that
    \begin{equation}
        \|\operatorname{Tr}f\|_{\mathfrak L_p(E,{\bar{\gamma}},S)}
        \le
        C
        \|f\|_{B^{\frac{\theta}{p}}_{p,1}(\mathbb R^n,\gamma)}.
    \end{equation}
\end{Prop}

\begin{proof}
    By Lemma~\ref{Lm.0dir}, $\operatorname{Tr}f
        \in
        L_p(E\setminus Q_{\rho_0}^E(S),{\bar{\gamma}})$
    and
    \begin{equation}
        \|\operatorname{Tr}f\|_{L_p(E\setminus Q_{\rho_0}^E(S),{\bar{\gamma}})}
        \lesssim
        \|f\|_{B^{\frac{\theta}{p}}_{p,1}(\mathbb R^n,\gamma)}.
    \end{equation}
    By Lemma~\ref{Lm.2dir},
    \begin{equation}
        \sum_{x_0\in S}|\operatorname{Tr}f(x_0)|
        \lesssim
        \|f\|_{B^{\frac{\theta}{p}}_{p,1}(\mathbb R^n,\gamma)}.
    \end{equation}
    Finally, Lemma~\ref{Lm.ren_dir} gives
    \begin{equation}
        \sum_{x_0\in S}
        \|\operatorname{Tr}f-\operatorname{Tr}f(x_0)\|_{L_p(Q^E_{\rho_0}(x_0),{\bar{\gamma}})}
        \lesssim
        \|f\|_{B^{\frac{\theta}{p}}_{p,1}(\mathbb R^n,\gamma)}.
    \end{equation}
    Combining these estimates and using the definition of
    \(\mathfrak L_p(E,{\bar{\gamma}},S)\), we obtain the desired bound.
\end{proof}
\section{The inverse trace theorem}
In this section, we prove the inverse part of the trace theorem. We construct a nonlinear extension operator and prove its boundedness. We mention that our construction is inspired by the recent paper \cite{tyulenev2026}.
\par
Throughout this section, we fix the following data:
\begin{itemize}
    \item a number \(d\in(0,n)\) and an Ahlfors--David \(d\)-regular set \(E\subset\mathbb R^n\);
    \item an integrability parameter \(p\in[1,\infty)\);
    \item a codimension parameter \(\theta\in(0,p)\);
    \item a weight \(\gamma\in A_p^{\operatorname{loc}}(\mathbb R^n)\) such that \(E\) is Ahlfors--David codimension-\(\theta\) almost regular with respect to \(\gamma\).
\end{itemize}
We start with the construction of a special family of smooth functions.
\subsection{Extension family}
\begin{Def}
    For every \(k\in\mathbb Z\) and every \(m\in\mathbb Z^n\), let \(Q_{k,m}\) denote the closed dyadic cube
    \begin{equation}
        Q_{k,m}
        :=
        \prod_{i=1}^n
        \left[
        \frac{m_i}{2^k},
        \frac{m_i+1}{2^k}
        \right].
    \end{equation}
    For each \(k\in\mathbb Z\), we denote by \(\mathcal D_k\) the collection $\mathcal D_k:=\{Q_{k,m}:m\in\mathbb Z^n\}$
    of dyadic cubes with sidelength \(2^{-k}\). Given \(c>0\), we set
    \begin{equation}
    \label{eq.ext_f_def}
        \mathcal D_k(E,c)
        :=
        \{Q\in\mathcal D_k:cQ\cap E\ne\emptyset\}.
    \end{equation}
\end{Def}
Until the end of this section, we fix two auxiliary parameters $\kappa\in\left(1,\frac32\right)$ and $\eta\in(0,\kappa-1)$. Since \(E\) and \(\kappa\) are fixed, we shall write \(\mathscr D_k:=\mathcal{D}_k(E,\kappa)\).
Let \(\psi\in C_0^\infty(\mathbb R^n)\) be such that
\begin{enumerate}
    \item \(\chi_{(1-\eta)Q_{0,0}}(x)\le \psi(x)\le \chi_{(1+\eta)Q_{0,0}}(x)\);
    \item \(\sum_{m\in\mathbb Z^n}\psi(x-m)=1\) for all \(x\in\mathbb R^n\).
\end{enumerate}
For each dyadic cube \(Q=Q_{k,m}\), we set $\psi_Q(x):=\psi(2^kx-m)$.
And for each \(k\in\mathbb Z\), we define
\begin{equation}
    g_k(x)
    :=
    \sum_{Q\in \mathscr D_k}
    \psi_Q(x).
\end{equation}
We also put
\begin{equation}
    \rho_-:=\frac{\kappa-1-\eta}{2},
    \qquad
    \rho_+:=\frac{\kappa+1+\eta}{2}.
\end{equation}
\begin{Lm}
    \label{Lm.lay}
    For every \(k\in\mathbb Z\),
    \begin{enumerate}
    \item for every $Q\in \mathcal{D}_k$ \begin{equation}
\label{eq.grad_est}
    |\nabla\psi_Q(x)|
    \le
    C2^k, \qquad \text{for all } x \in \mathbb{R}^n;
\end{equation}
\item \begin{equation}
    |\nabla g_k(x)|
    \le
    C2^k, \qquad \text{for all } x \in \mathbb{R}^n;
\end{equation}
    \item \begin{equation}
        \chi_{Q_{\rho_-2^{-k}}(E)}(x)
        \le
        g_k(x)
        \le
        \chi_{Q_{\rho_+2^{-k}}(E)}(x), \qquad \text{for all } x \in \mathbb{R}^n.
    \end{equation}
    \end{enumerate}
\end{Lm}
\begin{proof}
    The first two assertions are clear from the definitions. Therefore, we only need to prove the third assertion. To this end, notice that, by the definition of \(\psi\), for each \(k\in\mathbb Z\) the family
    \(\{\psi_Q\}_{Q\in\mathcal D_k}\) forms a partition of unity. Hence, for every \(x\in\mathbb R^n\),
    \begin{equation}
        g_k(x)
        =
        \sum_{Q\in\mathcal D_k}\psi_Q(x)
        -
        \sum_{Q\notin\mathscr D_k}\psi_Q(x)
        =
        1-
        \sum_{Q\notin\mathscr D_k}\psi_Q(x).
    \end{equation}
    Let \(Q\notin \mathscr D_k\). Then \(\kappa Q\cap E=\emptyset\), and therefore $\operatorname{dist}(Q,E)\ge\frac{\kappa-1}{2}2^{-k}$.
    Since \(\operatorname{supp}\psi_Q\subset (1+\eta)Q\), we have
    \begin{equation}
        \operatorname{dist}((1+\eta)Q,E)
        \ge
        \frac{\kappa-1-\eta}{2}2^{-k}
        =
        \rho_-2^{-k}.
    \end{equation}
    Thus, if \(\operatorname{dist}(x,E)<\rho_-2^{-k}\), then
    \(x\notin\operatorname{supp}\psi_Q\) for every
    \(Q\notin\mathscr D_k\). Consequently, $g_k(x)=1$
    for all \(x\in Q_{\rho_-2^{-k}}(E)\). This proves the first inequality.
    \par
    To prove the second inequality, it is enough to observe that
    \begin{equation}
        \operatorname{supp}g_k
        \subset
        \bigcup_{Q\in\mathscr D_k}(1+\eta)Q.
    \end{equation}
    If \(Q\in\mathscr D_k\), then \(\kappa Q\cap E\ne\emptyset\). Hence every point of \((1+\eta)Q\) has distance at most $\frac{\kappa+1+\eta}{2}2^{-k} = \rho_+2^{-k}$
    from \(E\). Therefore,
    \begin{equation}
        \operatorname{supp}g_k
        \subset
        Q_{\rho_+2^{-k}}(E),
    \end{equation}
    which proves the second inequality.
\end{proof}
Next, we divide the family \(\mathscr D_k\) into two subfamilies: regular and singular cubes. To this end, we fix another parameter $\lambda\in(\kappa,2)$. For each \(Q\in\mathcal D_k(E,\kappa)\), we define the associated patch on \(E\) by $\widehat Q:=\lambda Q\cap E$.
We emphasize that \(\widehat Q\) need not be a relative cube in \(E\). However, for \(c>0\), we shall use the notation $c\widehat Q:=c\lambda Q\cap E$.
Finally, we choose $\sigma>10(\kappa+\lambda+1)$.
For each \(k\in\mathbb Z\), we split \(\mathscr D_k\) by setting
\begin{equation}
\label{eq.reg_sing_def}
    \begin{split}
    \mathscr D_k^r
    &:=
    \{Q\in\mathscr D_k:
    \operatorname{dist}(Q,S)\ge \sigma2^{-k}\},
    \\
   \mathscr D_k^s
    &:=
    \mathscr D_k\setminus\mathscr D_k^r,
    \end{split}
\end{equation}
where \(S=RS_{p,\theta}(\gamma)\cap E\).
\begin{Lm}
    \label{Lm.ext_f_pr}
    The following assertions hold.
    \begin{enumerate}
        \item For each \(k\in\mathbb Z\) and each \(\Lambda\ge1\), the collection
        \begin{equation}
            \{\Lambda\widehat Q\}_{Q\in\mathscr D_k}
        \end{equation}
        has covering multiplicity at most \(([\Lambda\lambda]+2)^n\).
        \item For every \(k\in\mathbb N_0\) and every
        \(Q\in\mathscr D_k^r\),
        \begin{equation}
        \label{eq.alm_reg_dyad}
            \frac{\bm{\gamma}(Q)}
            {\bm{\bar{\gamma}}(\widehat Q)}
            \approx
            2^{k\theta}.
        \end{equation}
    \end{enumerate}
\end{Lm}
\begin{proof}
    The first assertion follows directly from the fact that the family $\{\Lambda\lambda Q:Q\in\mathcal D_k\}$
    has covering multiplicity at most \(([\Lambda\lambda]+2)^n\).
    \par
    We prove the second assertion. Let \(Q\in\mathscr D_k^r\). Choose
    \(x_Q\in\kappa Q\cap E\). Since \(Q\) belongs to the regular family and \(x_Q\in\kappa Q\), the choice of \(\sigma\) gives $\operatorname{dist}(x_Q,S) \ge c_\sigma 2^{-k}$,
    where \(c_\sigma> \lambda+\kappa+1\). In particular, all cubes below are sufficiently far from \(S\), so that the almost regularity condition applies.
    Since \(x_Q\in\kappa Q\), we have
    \begin{equation}
        Q^E_{\frac{\lambda-\kappa}{2}2^{-k}}(x_Q)
        \subset
        \widehat Q
        \subset
        Q^E_{\frac{\lambda+\kappa}{2}2^{-k}}(x_Q).
    \end{equation}
    By the almost regularity condition \eqref{eq.alm_reg} and the local doubling properties of \(\bm{\gamma}\) and \(\bm{\bar{\gamma}}\) away from \(S\), we obtain
    \begin{equation}
        \bm{\bar{\gamma}}(\widehat Q)
        \approx
        \bm{\bar{\gamma}}
        \left(
        Q^E_{\frac{\lambda+\kappa}{2}2^{-k}}(x_Q)
        \right)
        \approx
        \frac{
        \bm{\gamma}
        \left(
        Q_{\frac{\lambda+\kappa}{2}2^{-k}}(x_Q)
        \right)
        }
        {2^{-k\theta}}.
    \end{equation}
    Therefore, applying the local doubling property of \(\bm{\gamma}\), we obtain \eqref{eq.alm_reg_dyad}. The proof is complete.
\end{proof}
\subsection{Extension operator}
Let us recall that in the previous subsection we fixed the following parameters:
\begin{enumerate}
    \item the parameter \(\kappa\in(1,\frac32)\), which controls the tube of dyadic cubes touching \(E\);
    \item the parameter \(\eta\in(0,\kappa-1)\), which controls the supports of the smooth partition of unity;
    \item the parameters $\rho_{\pm}:=\frac{\kappa \pm (1+\eta)}{2}$,
    which describe the support of \(g_k\);
    \item the parameter \(\lambda\in(\kappa,2)\), which controls the dilation in the definition of \(\widehat Q\);
    \item the parameter \(\sigma>10(\kappa+\lambda+1)\), which controls the separation from the set of rapid singularities in the definition of regular cubes.
\end{enumerate}
For brevity, we put
\begin{equation}
    S:=RS_{p,\theta}(\gamma)\cap E,
    \qquad
    \rho_0:=\rho_S,
    \qquad
    c_\eta:=1+\eta.
\end{equation}
We also introduce several new parameters.
\par
First, choose a separation scale \(k_{\operatorname{sep}}\in\mathbb N\) such that
\begin{equation}
    2(\sigma+c_\eta+1)2^{-k_{\operatorname{sep}}}<\rho_0.
\end{equation}
Second, define the cutoff shift \(m_{\operatorname{cut}}\in\mathbb N\) by
\begin{equation}
    m_{\operatorname{cut}}
    :=
    \left[
    \log_2\frac{\rho_+}{\rho_-}
    \right]+1.
\end{equation}
The crucial role of this parameter is described by the following consequence of Lemma~\ref{Lm.lay}(3):
\begin{equation}
\label{eq.cut_off_shift}
    \operatorname{supp}g_{k+m_{\operatorname{cut}}}
    \subset
    \{x\in\mathbb R^n:g_k(x)=1\}
\end{equation}
for every \(k\in\mathbb Z\).
Finally, we set
\begin{equation}
    \xi:=2c_\eta+\lambda,
    \qquad
    \zeta_r:=\sigma-2c_\eta-\lambda,
    \qquad
    \zeta_s:=\sigma+2c_\eta+\lambda.
\end{equation}
These constants will be used to simplify notation in the combinatorial estimates below. Notice that, by the choice of \(\sigma\), we have $\zeta_r>0$ and $\frac{\xi}{\zeta_r}\le \frac14$.
\par
Throughout this section, all parameters introduced above are fixed. Therefore, any constant depending only on these parameters will be regarded as a structural constant. In particular, the constants denoted by \(C\), as well as the implicit constants in the notation \(\lesssim\) and \(\approx\), may depend on these parameters, but are independent of the function \(\phi\).
\par
We begin the construction of the extension operator with a simple observation. For every \(k\ge k_{\operatorname{sep}}\) and every
\(Q\in\mathscr D_k^s\), there is a unique point \(x_0\in S\) such that $\operatorname{dist}(Q,S)=\operatorname{dist}(Q,x_0)$.
Indeed, if two distinct points of \(S\) had the same property, their distance would be at most \((2\sigma+1)2^{-k}\), which is strictly smaller than \(\rho_0\), contradicting the definition of \(\rho_0\). We denote this point by \(x_0(Q)\).
\par
For every \(\phi\in\mathfrak L_p(E,{\bar{\gamma}},S)\) and every
\(Q\in\mathscr D_k\), \(k\ge k_{\operatorname{sep}}\), we set
\begin{equation}
    a_Q(\phi)
    :=
    \begin{cases}
    \displaystyle
    \fint\limits_{\widehat Q}\phi(x)\,d\bm{\bar{\gamma}}(x),
    & \text{if } Q\in\mathscr D_k^r,\\[3mm]
    \phi(x_0(Q)),
    & \text{if } Q\in\mathscr D_k^s.
    \end{cases}
\end{equation}
We now define \(\operatorname{Ext}\phi\) in several steps.
\par
\emph{Step 1.}
Since \(\frac{\xi}{\zeta_r}\le \frac14\), Proposition~\ref{Prop.cruc_prop} and the absolute continuity of the integral allow us to choose, for each
\(\phi\in\mathfrak L_p(E,{\bar{\gamma}},S)\), a strictly increasing sequence
    $\{k_l\}_{l=0}^{\infty}\subset\mathbb N$,
    $k_0=k_{\operatorname{sep}}$,
such that $k_{l+1}\ge k_l+m_{\operatorname{cut}}$ and, for every \(l\ge1\),
\begin{equation}
\label{eq.ext_def1}
    \left\|\delta_l\phi
    \right\|_{L_1(\mathcal{O}_l,{\bar{\gamma}})}
    \le
    2^{-lp}
    \|\phi\|_{\mathfrak L_p(E,{\bar{\gamma}},S)}^p
\end{equation}
and
\begin{equation}
\label{eq.ext_def1_tail}
    \sum_{x_0\in S}
    \|\phi-\phi(x_0)\|_{L_p(\mathcal{I}_l(x_0),{\bar{\gamma}})}^p
    \le
    2^{-lp}
    \|\phi\|_{\mathfrak L_p(E,{\bar{\gamma}},S)}^p,
\end{equation}
where, for brevity, we set
\begin{equation}
    \delta_l\phi:=
    \delta_{\xi2^{-k_l}}^{E,p}\phi, \qquad \mathcal{O}_l:=E\setminus Q^E_{\zeta_r2^{-k_l}}(S), \qquad \mathcal{I}_l(x_0):=Q^E_{\zeta_s2^{-k_l}}(x_0).
\end{equation}
The dependence of the sequence \(\{k_l\}\) on \(\phi\) is the source of the nonlinearity of the extension operator.
\par
\emph{Step 2.}
For each \(l\ge0\), we set
\begin{equation}
\label{eq.ext_def2}
    \mathcal E_l\phi(x)
    :=
    \sum_{Q\in\mathscr D_{k_l}}
    a_Q(\phi)\psi_Q(x),
    \qquad
    \mathcal E_{-1}\phi(x):=0.
\end{equation}
\par
\emph{Step 3.}
We define an extension of \(\phi\) by
\begin{equation}
\label{eq.def_ext}
    \operatorname{Ext}\phi(x)
    :=
    \sum_{l=0}^{\infty}
    \left(
    \mathcal E_l\phi(x)-\mathcal E_{l-1}\phi(x)
    \right)
    \widetilde g_l(x),
    \qquad
    \widetilde g_l:=g_{k_l+m_{\operatorname{cut}}}.
\end{equation}
For \(x\in\mathbb R^n\setminus E\), the series in \eqref{eq.def_ext} contains only finitely many nonzero terms. On \(E\), we may define \(\operatorname{Ext}\phi\) arbitrarily, since \(d<n\) and hence \(\mathcal L^n(E)=0\). For brevity, we use the notation
\begin{equation}
\widetilde{\mathcal D}_l(\phi)
    :=
   \mathscr D_{k_l}, \qquad
    \widetilde{\mathcal D}_l^r(\phi)
    :=
     \mathscr D_{k_l}^r,
    \qquad
    \widetilde{\mathcal D}_l^s(\phi)
    :=
    \mathscr D_{k_l}^s.
\end{equation}
\par
First, we estimate the norm of each term in \eqref{eq.def_ext} separately. We remark that the terms with \(l=0\) and \(l>0\) have to be treated differently because of the definition of the extension.
\begin{Lm}
\label{Lm.zero_term}
    There is a constant \(C>0\) such that, for every
    \(\phi\in\mathfrak L_p(E,{\bar{\gamma}},S)\),
    \begin{equation}
        \label{eq.zero_term_norm}
        \|\mathcal E_0\phi\cdot\widetilde g_0\|_{L_p(\mathbb R^n,\gamma)}
        \le
        C\|\phi\|_{\mathfrak L_p(E,{\bar{\gamma}},S)}
    \end{equation}
    and
    \begin{equation}
        \label{eq.zero_term_grad_norm}
        \left\|
        \left|\nabla(\mathcal E_0\phi\cdot\widetilde g_0)\right|
        \right\|_{L_p(\mathbb R^n,\gamma)}
        \le
        C\|\phi\|_{\mathfrak L_p(E,{\bar{\gamma}},S)}.
    \end{equation}
\end{Lm}
\begin{proof}
Recall that \(k_0=k_{\operatorname{sep}}\) in the construction of the sequence
\(\{k_l\}_{l=0}^{\infty}\). Thus $\widetilde{\mathcal D}_0(\phi)= \mathscr D_{k_{\operatorname{sep}}}$. Since the family \(\{c_\eta Q:Q\in\mathcal D_{k_{\operatorname{sep}}}\}\) has uniformly bounded multiplicity, we have
\begin{equation}
\label{eq.ext_lp1}
    \|\mathcal E_0\phi\cdot\widetilde g_0\|_{L_p(\mathbb R^n,\gamma)}^p
    \lesssim
    \sum_{Q\in\widetilde{\mathcal D}_0(\phi)}
    |a_Q(\phi)|^p\,\bm{\gamma}(c_\eta Q).
\end{equation}
\par
We first estimate the contribution of regular cubes. If
\(Q\in\widetilde{\mathcal D}_0^r(\phi)\), then, by Jensen's inequality,
\begin{equation}
\label{eq.ext_lp2}
    |a_Q(\phi)|^p
    \le
    \fint\limits_{\widehat Q}
    |\phi(x)|^p\,d\bm{\bar{\gamma}}(x).
\end{equation}
Using the local doubling property of \(\bm{\gamma}\) and Lemma~\ref{Lm.ext_f_pr}, we obtain
\begin{equation}
\label{eq.ext_lp3}
\begin{split}
    \sum_{Q\in\widetilde{\mathcal D}_0^r(\phi)}
    |a_Q(\phi)|^p\,\bm{\gamma}(c_\eta Q)
    &\lesssim
    \sum_{Q\in\widetilde{\mathcal D}_0^r(\phi)}
    \frac{\bm{\gamma}(Q)}
    {\bm{\bar{\gamma}}(\widehat Q)}
    \int\limits_{\widehat Q}
    |\phi(x)|^p\,d\bm{\bar{\gamma}}(x)
    \lesssim
    \int\limits_{G_{k_{\operatorname{sep}}}}
    |\phi(x)|^p\,d\bm{\bar{\gamma}}(x),
\end{split}
\end{equation}
where
\begin{equation}
\label{eq.ext_lp4}
    G_k
    :=
    \bigcup_{Q\in\mathscr D_k^r}
    \widehat Q.
\end{equation}
The set \(G_{k_{\operatorname{sep}}}\) stays a positive distance away from \(S\). Therefore, by the definition of
\(\mathfrak L_p(E,{\bar{\gamma}},S)\) and by Remark~\ref{Rm.ren_sp}, we have
\begin{equation}
\label{eq.ext_lp4b}
    \|\phi\|_{L_p(G_{k_{\operatorname{sep}}},{\bar{\gamma}})}
    \lesssim
    \|\phi\|_{\mathfrak L_p(E,{\bar{\gamma}},S)}.
\end{equation}

We now estimate the contribution of singular cubes. If
\(Q\in\widetilde{\mathcal D}_0^s(\phi)\), then $a_Q(\phi)=\phi(x_0(Q))$.
Moreover,
\begin{equation}
    \bigcup_{Q\in\widetilde{\mathcal D}_0^s(\phi)} c_\eta Q
    \subset
    Q_{C2^{-k_{\operatorname{sep}}}}(S)
\end{equation}
for some constant \(C>0\) depending only on the fixed parameters. Since
\(k_{\operatorname{sep}}\) is fixed and \(S\) is finite, the local integrability of
\(\gamma\) gives
\begin{equation}
\label{eq.ext_lp5}
\begin{split}
    \sum_{Q\in\widetilde{\mathcal D}_0^s(\phi)}
    |a_Q(\phi)|^p\,&\bm{\gamma}(c_\eta Q)
    \lesssim
    \sum_{x_0\in S}
    |\phi(x_0)|^p
    \bm{\gamma}\bigl(Q_{C2^{-k_{\operatorname{sep}}}}(x_0)\bigr)
    \lesssim
    \|\phi\|_{\mathfrak{L}_p(E,{\bar{\gamma}},S)}^p.
\end{split}
\end{equation}
Combining \eqref{eq.ext_lp1}, \eqref{eq.ext_lp3}, \eqref{eq.ext_lp4b}, and
\eqref{eq.ext_lp5}, we obtain
\begin{equation}
\label{eq.ext_lp6}
    \|\mathcal E_0\phi\cdot\widetilde g_0\|_{L_p(\mathbb R^n,\gamma)}
    \lesssim
    \|\phi\|_{\mathfrak L_p(E,{\bar{\gamma}},S)}.
\end{equation}
\par
It remains to prove the gradient estimate. By Lemma~\ref{Lm.lay}(1) and (2) applied with $k=k_{\operatorname{sep}}$,
\begin{equation}
\label{eq.ext_grad_0}
    \left\|
    \left|\nabla(\mathcal E_0\phi\cdot\widetilde g_0)\right|
    \right\|_{L_p(\mathbb R^n,\gamma)}^p
    \lesssim
    \sum_{Q\in\widetilde{\mathcal D}_0(\phi)}
    |a_Q(\phi)|^p\,\bm{\gamma}(c_\eta Q).
\end{equation}
Therefore, the same estimates as above applied to the right-hand side of \eqref{eq.ext_grad_0} give \eqref{eq.zero_term_grad_norm}.
The proof is complete.
\end{proof}
Next, we estimate the terms with \(l\ge1\). To this end, given
\(\phi\in\mathfrak L_p(E,{\bar{\gamma}},S)\), we define special families of interacting cubes. For \(l\ge1\), set
\begin{equation}
\label{eq.ext_lp8}
\begin{split}
\mathcal F_l^{r,r}
&:=
\{(Q_1,Q_2)\in
\widetilde{\mathcal D}_l^r(\phi)\times
\widetilde{\mathcal D}_{l-1}^r(\phi):
c_\eta Q_1\cap c_\eta Q_2\ne\emptyset\},
\\
\mathcal F_l^{s,r}
&:=
\{(Q_1,Q_2)\in
\widetilde{\mathcal D}_l^s(\phi)\times
\widetilde{\mathcal D}_{l-1}^r(\phi):
c_\eta Q_1\cap c_\eta Q_2\ne\emptyset\},
\\
\mathcal F_l^{r,s}
&:=
\{(Q_1,Q_2)\in
\widetilde{\mathcal D}_l^r(\phi)\times
\widetilde{\mathcal D}_{l-1}^s(\phi):
c_\eta Q_1\cap c_\eta Q_2\ne\emptyset\},
\\
\mathcal F_l^{s,s}
&:=
\{(Q_1,Q_2)\in
\widetilde{\mathcal D}_l^s(\phi)\times
\widetilde{\mathcal D}_{l-1}^s(\phi):
c_\eta Q_1\cap c_\eta Q_2\ne\emptyset\}.
\end{split}
\end{equation}
These families describe different interacting pairs of regular and singular cubes from two consecutive selected layers.
\begin{Lm}
\label{Lm.higher_terms}
    There is a constant \(C>0\) such that, for every
    \(\phi\in\mathfrak L_p(E,{\bar{\gamma}},S)\) and every \(l\ge1\),
    \begin{equation}
        \label{eq.higher_term_norm}
        \left\|
        \left(\mathcal E_l\phi-\mathcal E_{l-1}\phi\right)\widetilde g_l
        \right\|_{L_p(\mathbb R^n,\gamma)}
        \le
        C
        2^{-k_l\frac{\theta}{p}}
        2^{-l}
        \|\phi\|_{\mathfrak L_p(E,{\bar{\gamma}},S)}
    \end{equation}
    and
    \begin{equation}
        \label{eq.higher_term_grad_norm}
        \left\|
        \left|
        \nabla\left(
        \left(\mathcal E_l\phi-\mathcal E_{l-1}\phi\right)\widetilde g_l
        \right)
        \right|
        \right\|_{L_p(\mathbb R^n,\gamma)}
        \le
        C
        2^{k_l\left(1-\frac{\theta}{p}\right)}
        2^{-l}
        \|\phi\|_{\mathfrak L_p(E,{\bar{\gamma}},S)}.
    \end{equation}
\end{Lm}
\begin{proof}
\emph{Step 1}.
Fix \(\phi\in\mathfrak L_p(E,{\bar{\gamma}},S)\) and \(l\ge1\). By the choice of
\(m_{\operatorname{cut}}\)  (see \eqref{eq.cut_off_shift}) we have $\operatorname{supp}\widetilde g_l \subset \{x\in\mathbb R^n:g_{k_l}(x)=1\}$. Moreover, since \(k_l\ge k_{l-1}+m_{cut}\), the same inclusion also gives $\operatorname{supp}\widetilde g_l
    \subset
    \{x\in\mathbb R^n:g_{k_{l-1}}(x)=1\}$.
Hence, on the support of \(\widetilde g_l\),
\begin{equation}
    \sum_{Q_1\in\widetilde{\mathcal D}_l(\phi)}\psi_{Q_1}(x)=1,
    \qquad
    \sum_{Q_2\in\widetilde{\mathcal D}_{l-1}(\phi)}\psi_{Q_2}(x)=1.
\end{equation}
Therefore, for all \(x\in\mathbb R^n\),
\begin{equation}
\label{eq.ext_lp7}
\begin{split}
    &\left(\mathcal E_l\phi(x)-\mathcal E_{l-1}\phi(x)\right)\widetilde g_l(x)
    =
    \sum_{\substack{Q_1\in\widetilde{\mathcal D}_l(\phi)\\
                    Q_2\in\widetilde{\mathcal D}_{l-1}(\phi)}}
    \left(a_{Q_1}(\phi)-a_{Q_2}(\phi)\right)
    \psi_{Q_1}(x)\psi_{Q_2}(x)\widetilde g_l(x).
\end{split}
\end{equation}
Only pairs with \(c_\eta Q_1\cap c_\eta Q_2\ne\emptyset\) contribute to this sum. Since the families
\(\{c_\eta Q:Q\in\mathcal D_k\}\) have uniformly bounded multiplicity, we obtain
\begin{equation}
\label{eq.ext_lp7n}
\begin{split}
    &\left\|
    \left(\mathcal E_l\phi-\mathcal E_{l-1}\phi\right)\widetilde g_l
    \right\|_{L_p(\mathbb R^n,\gamma)}^p\lesssim
    \sum_{\substack{Q_1\in\widetilde{\mathcal D}_l(\phi)\\
                    Q_2\in\widetilde{\mathcal D}_{l-1}(\phi)}}
    |a_{Q_1}(\phi)-a_{Q_2}(\phi)|^p
    \bm{\gamma}(c_\eta Q_1\cap c_\eta Q_2).
\end{split}
\end{equation}
We split the last sum according to the four families
\(\mathcal F_l^{r,r}\), \(\mathcal F_l^{r,s}\), \(\mathcal F_l^{s,r}\), and
\(\mathcal F_l^{s,s}\).
\par 
\emph{Step 2.}
Let \((Q_1,Q_2)\in\mathcal F_l^{r,r}\). By Jensen's inequality,
\begin{equation}
\label{eq.ext_lp9}
    |a_{Q_1}(\phi)-a_{Q_2}(\phi)|^p
    \le
    \fint\limits_{\widehat Q_1}
    \fint\limits_{\widehat Q_2}
    |\phi(x)-\phi(y)|^p
    \,d\bm{\bar{\gamma}}(y)\,d\bm{\bar{\gamma}}(x).
\end{equation}
Since \(Q_1\in\widetilde{\mathcal D}_l^r(\phi)\), Lemma~\ref{Lm.ext_f_pr}(2) and the local doubling property of \(\bm{\gamma}\) give
\begin{equation}
\label{eq.ext_lp9b}
    \frac{\bm{\gamma}(c_\eta Q_1\cap c_\eta Q_2)}
    {\bm{\bar{\gamma}}(\widehat Q_1)}
    \lesssim
    \frac{\bm{\gamma}(Q_1)}
    {\bm{\bar{\gamma}}(\widehat Q_1)}
    \lesssim
    2^{-k_l\theta}.
\end{equation}
Therefore,
\begin{equation}
\label{eq.ext_lp10}
\begin{split}
&\sum_{(Q_1,Q_2)\in\mathcal F_l^{r,r}}
|a_{Q_1}(\phi)-a_{Q_2}(\phi)|^p
\bm{\gamma}(c_\eta Q_1\cap c_\eta Q_2)
\\
&\quad\lesssim
2^{-k_l\theta}
\sum_{(Q_1,Q_2)\in\mathcal F_l^{r,r}}
\int\limits_{\widehat Q_1}
\fint\limits_{\widehat Q_2}
|\phi(x)-\phi(y)|^p
\,d\bm{\bar{\gamma}}(y)\,d\bm{\bar{\gamma}}(x).
\end{split}
\end{equation}
\par
For each fixed \(Q_2\in\widetilde{\mathcal D}_{l-1}^r(\phi)\), the interacting patches satisfy
\begin{equation}
\label{eq.adj_c}
    \bigcup_{\substack{Q_1\in\widetilde{\mathcal D}_l(\phi):\\
    c_\eta Q_1\cap c_\eta Q_2\ne\emptyset}}
    \widehat Q_1
    \subset
    \xi\widehat Q_2.
\end{equation}
Moreover, for every \(x\in\xi\widehat Q_2\), we have $\widehat Q_2
    \subset
    Q_{\xi2^{-k_{l-1}}}^E(x)$.
Since \(Q_2\) is regular, the local doubling property of \(\bm{\bar{\gamma}}\) away from \(S\) gives
    $\bm{\bar{\gamma}}(\widehat Q_2)
    \approx
    \bm{\bar{\gamma}}(Q_{\xi2^{-k_{l-1}}}^E(x))$.
Consequently, using \eqref{eq.adj_c} and the bounded multiplicity of the family of patches, we obtain
\begin{equation}
\label{eq.rr_local}
\begin{split}
&\sum_{\substack{Q_1\in\widetilde{\mathcal D}_l(\phi):\\
c_\eta Q_1\cap c_\eta Q_2\ne\emptyset}}
\int\limits_{\widehat Q_1}
\fint\limits_{\widehat Q_2}
|\phi(x)-\phi(y)|^p
\,d\bm{\bar{\gamma}}(y)\,d\bm{\bar{\gamma}}(x)\lesssim
\int\limits_{\xi\widehat Q_2}
\delta_{l-1}\phi(x)
\,d\bm{\bar{\gamma}}(x).
\end{split}
\end{equation}
Summing over \(Q_2\in\widetilde{\mathcal D}_{l-1}^r(\phi)\) and using the bounded multiplicity once more, we get
\begin{equation}
\label{eq.rr_sum}
\begin{split}
&\sum_{(Q_1,Q_2)\in\mathcal F_l^{r,r}}
|a_{Q_1}(\phi)-a_{Q_2}(\phi)|^p
\bm{\gamma}(c_\eta Q_1\cap c_\eta Q_2)
\lesssim
2^{-k_l\theta}
\left\|
\delta_{l-1}\phi
\right\|_{L_1(U_{l-1},\bar{\gamma})},
\end{split}
\end{equation}
where
\begin{equation}
    U_{l-1}
    :=
    \bigcup_{Q_2\in\widetilde{\mathcal D}^r_{l-1}(\phi)}
    \xi\widehat Q_2.
\end{equation}
By the definitions of regular cubes and of \(\mathcal{O}_l\), we have
    $U_{l-1}
    \subset
    \mathcal{O}_{l-1}$.
Hence, by the choice of the sequence \(\{k_l\}\) (see \eqref{eq.ext_def1}),
\begin{equation}
\label{eq.ext_lp_rr}
\begin{split}
&\sum_{(Q_1,Q_2)\in\mathcal F_l^{r,r}}
|a_{Q_1}(\phi)-a_{Q_2}(\phi)|^p
\bm{\gamma}(c_\eta Q_1\cap c_\eta Q_2)\lesssim
2^{-k_l\theta}2^{-lp}
\|\phi\|_{\mathfrak L_p(E,\bm{\bar{\gamma}},S)}^p.
\end{split}
\end{equation}
\par
\emph{Step 3.}
Let \((Q_1,Q_2)\in\mathcal F_l^{r,s}\). Then
\(a_{Q_2}(\phi)=\phi(x_0(Q_2))\), and by Jensen's inequality,
\begin{equation}
\label{eq.ext_lp11}
    |a_{Q_1}(\phi)-a_{Q_2}(\phi)|^p
    \le
    \fint\limits_{\widehat Q_1}
    |\phi(x)-\phi(x_0(Q_2))|^p\,d\bm{\bar{\gamma}}(x).
\end{equation}
Using Lemma~\ref{Lm.ext_f_pr}(2) and the local doubling property of
\(\bm{\gamma}\), we get
\begin{equation}
\label{eq.ext_lp12}
\begin{split}
&\sum_{(Q_1,Q_2)\in\mathcal F_l^{r,s}}
|a_{Q_1}(\phi)-a_{Q_2}(\phi)|^p
\bm{\gamma}(c_\eta Q_1\cap c_\eta Q_2)
\\
&\qquad\lesssim
2^{-k_l\theta}
\sum_{Q_2\in\widetilde{\mathcal D}_{l-1}^s(\phi)}
\sum_{\substack{Q_1\in\widetilde{\mathcal D}_l^r(\phi):\\
c_\eta Q_1\cap c_\eta Q_2\ne\emptyset}}
\int\limits_{\widehat Q_1}
|\phi(x)-\phi(x_0(Q_2))|^p\,d\bm{\bar{\gamma}}(x).
\end{split}
\end{equation}
For fixed \(x_0\in S\), the interacting regular patches $\widehat Q_1$ corresponding to singular cubes \(Q_2\) with \(x_0(Q_2)=x_0\) are contained in $\mathcal{I}_{l-1}(x_0)$.
Therefore, by bounded overlap of patches,
\begin{equation}
\label{eq.ext_lp_rs}
\begin{split}
&\sum_{(Q_1,Q_2)\in\mathcal F_l^{r,s}}
|a_{Q_1}(\phi)-a_{Q_2}(\phi)|^p
\bm{\gamma}(c_\eta Q_1\cap c_\eta Q_2)
\lesssim
2^{-k_l\theta}
\sum_{x_0\in S}
\|\phi-\phi(x_0)\|_{L_p(\mathcal{I}_{l-1}(x_0),{\bar{\gamma}})}^p.
\end{split}
\end{equation}
By the additional choice condition \eqref{eq.ext_def1_tail}, this gives
\begin{equation}
\label{eq.ext_lp_rs_final}
\begin{split}
&\sum_{(Q_1,Q_2)\in\mathcal F_l^{r,s}}
|a_{Q_1}(\phi)-a_{Q_2}(\phi)|^p
\bm{\gamma}(c_\eta Q_1\cap c_\eta Q_2)
\lesssim
2^{-k_l\theta}2^{-lp}
\|\phi\|_{\mathfrak L_p(E,{\bar{\gamma}},S)}^p.
\end{split}
\end{equation}
\par
\emph{Step 4.}
We first prove that $\mathcal F_l^{s,r}=\emptyset$.
Assume, to the contrary, that
\((Q_1,Q_2)\in\mathcal F_l^{s,r}\). Then \(Q_1\) is singular at level \(k_l\), while \(Q_2\) is regular at level \(k_{l-1}\), and
    $c_\eta Q_1\cap c_\eta Q_2\ne\emptyset$.
Since \(Q_1\) is singular,
    $\operatorname{dist}(Q_1,S)<\sigma2^{-k_l}$.
The interaction condition gives
    $\operatorname{dist}(Q_1,Q_2)\le 2c_\eta2^{-k_{l-1}}$.
Therefore,
\begin{equation}
    \operatorname{dist}(Q_2,S)
    \le
    \operatorname{dist}(Q_2,Q_1)
    +
    \operatorname{diam}Q_1
    +
    \operatorname{dist}(Q_1,S)
    \le
    2c_\eta2^{-k_{l-1}}+(\sigma+1)2^{-k_l}.    
\end{equation}
Since \(k_l\ge k_{l-1}+m_{\operatorname{cut}}\) and since \(\sigma\) was chosen sufficiently large compared with the fixed geometric constants, the right-hand side is strictly smaller than $\sigma2^{-k_{l-1}}$.
This contradicts the assumption \(Q_2\in\mathscr D^r_{k_{l-1}}\). Hence $\mathcal F_l^{s,r}=\emptyset$.
\par
Now let \((Q_1,Q_2)\in\mathcal F_l^{s,s}\). We claim that $x_0(Q_1)=x_0(Q_2)$.
Indeed, suppose that \(x_0(Q_1)\ne x_0(Q_2)\). Since distinct points of \(S\) have distance at least \(2\rho_0\), we have $|x_0(Q_1)-x_0(Q_2)|\ge2\rho_0$.
On the other hand, since \(Q_1\) and \(Q_2\) are singular and interact, we get
\begin{equation}
\begin{split}
    |x_0(Q_1)-x_0(Q_2)|
    \le
    \operatorname{dist}(x_0(Q_1),Q_1)
    +
    \operatorname{diam}Q_1
    +
    \operatorname{dist}(Q_1,Q_2)
    \\\quad+
    \operatorname{diam}Q_2
    +
    \operatorname{dist}(Q_2,x_0(Q_2))
    \le
    2(\sigma+c_\eta+1)2^{-k_{l-1}}.
\end{split}
\end{equation} 
Since \(k_{l-1}\ge k_{\operatorname{sep}}\), the definition of \(k_{\operatorname{sep}}\) gives
    $2(\sigma+c_\eta+1)2^{-k_{l-1}}<\rho_0$. This is a contradiction. Therefore, $x_0(Q_1)=x_0(Q_2)$.
Consequently,
    $a_{Q_1}(\phi)=\phi(x_0(Q_1))=\phi(x_0(Q_2))=a_{Q_2}(\phi)$,
and all singular--singular terms vanish.
\par
Combining \eqref{eq.ext_lp7n}, \eqref{eq.ext_lp_rr}, the regular--singular estimate
\eqref{eq.ext_lp_rs_final}, and the fact that the singular--regular and
singular--singular contributions vanish, we obtain \eqref{eq.higher_term_norm}.
\par
\emph{Step 5.}
We differentiate the representation \eqref{eq.ext_lp7}. Using Lemma~\ref{Lm.lay}(1) and (2) and the bounded multiplicity of interacting cubes, we obtain
\begin{equation}
\label{eq.grad_higher_aux}
\begin{split}
&\left\|
\left|\nabla\left(
(\mathcal E_l\phi-\mathcal E_{l-1}\phi)\widetilde g_l
\right)\right|
\right\|_{L_p(\mathbb R^n,\gamma)}^p
\\
&\quad\lesssim
2^{pk_l}
\sum_{\substack{Q_1\in\widetilde{\mathcal D}_l(\phi)\\
Q_2\in\widetilde{\mathcal D}_{l-1}(\phi)}}
|a_{Q_1}(\phi)-a_{Q_2}(\phi)|^p
\bm{\gamma}(c_\eta Q_1\cap c_\eta Q_2).
\end{split}
\end{equation}
The sum on the right-hand side was already estimated in the previous steps, which gives \eqref{eq.higher_term_grad_norm}. The proof is complete.
\end{proof}
Now we are ready to establish the boundedness of the extension operator.
\begin{Prop}
    \label{Prop.ext_bound}
    There is a constant \(C>0\) such that, for every
    \(\phi\in\mathfrak L_p(E,{\bar{\gamma}},S)\),
    \begin{equation}
    \label{eq.ext_bp_stat}
        \|\operatorname{Ext}\phi\|_{B^{\frac{\theta}{p}}_{p,1}(\mathbb R^n,\gamma)}
        \le
        C
        \|\phi\|_{\mathfrak L_p(E,{\bar{\gamma}},S)}.
    \end{equation}
\end{Prop}
\begin{proof}
Fix \(\phi\in\mathfrak L_p(E,{\bar{\gamma}},S)\). For brevity, set $f_l
    :=
    \left(\mathcal E_l\phi-\mathcal E_{l-1}\phi\right)\widetilde g_l$,
    $ l\ge0$.
Then $\operatorname{Ext}\phi
    =
    \sum\limits_{l=0}^{\infty} f_l$.
By the triangle inequality in \(B^{\frac{\theta}{p}}_{p,1}(\mathbb R^n,\gamma)\), it is enough to estimate $\sum\limits_{l=0}^{\infty}
    \|f_l\|_{B^{\frac{\theta}{p}}_{p,1}(\mathbb R^n,\gamma)}$.
\par
For each \(l\ge0\), Remark~\ref{Rm.dif_bes}, applied with
\(\delta=2^{-k_l}\), gives
\begin{equation}
\label{eq.ext_bound_aux}
    \|f_l\|_{B^{\frac{\theta}{p}}_{p,1}(\mathbb R^n,\gamma)}
    \lesssim
    2^{-k_l\left(1-\frac{\theta}{p}\right)}
    \||\nabla f_l|\|_{L_p(\mathbb R^n,\gamma)}
    +
    2^{k_l\frac{\theta}{p}}
    \|f_l\|_{L_p(\mathbb R^n,\gamma)}.
\end{equation}
For \(l=0\), Lemma~\ref{Lm.zero_term} yields
\begin{equation}
    \|f_0\|_{B^{\frac{\theta}{p}}_{p,1}(\mathbb R^n,\gamma)}
    \lesssim
    \|\phi\|_{\mathfrak L_p(E,{\bar{\gamma}},S)},
\end{equation}
because \(k_0=k_{\operatorname{sep}}\) is fixed. For \(l\ge1\), Lemma~\ref{Lm.higher_terms} and \eqref{eq.ext_bound_aux} give
\begin{equation}
    \|f_l\|_{B^{\frac{\theta}{p}}_{p,1}(\mathbb R^n,\gamma)}
    \lesssim
    2^{-l}
    \|\phi\|_{\mathfrak L_p(E,{\bar{\gamma}},S)}.
\end{equation}
Therefore,
\begin{equation}
    \sum_{l=0}^{\infty}
    \|f_l\|_{B^{\frac{\theta}{p}}_{p,1}(\mathbb R^n,\gamma)}
    \lesssim
    \|\phi\|_{\mathfrak L_p(E,{\bar{\gamma}},S)}.
\end{equation}
This proves \eqref{eq.ext_bp_stat}.
\end{proof}
Finally, we show that the extension operator \(\operatorname{Ext}\) is a right inverse to the trace operator. We first prove an auxiliary estimate.
\begin{Lm}
    \label{Lm.ext_inv_est}
    Given \(\phi\in\mathfrak L_p(E,{\bar{\gamma}},S)\), for \(l\in\mathbb N_0\), set $r_l:=\rho_-2^{-(k_l+m_{\operatorname{cut}})}$.
    Let \(x\in E\setminus S\). Then there exists \(l(x)\in\mathbb N_0\) such that, for every \(l\ge L\ge l(x)\), every cube
    \(Q\in\widetilde{\mathcal D}_l(\phi)\) satisfying
       $ c_\eta Q\cap Q_{r_L}(x)\ne\emptyset$
    is regular. Moreover, there exist constants \(C>0\) and \(\beta>0\), independent of
    \(\phi,l,L,x\), and \(a\), such that, for every \(a\in\mathbb R\),
    \begin{equation}
    \label{eq.ext_inv_est}
        \|(\mathcal E_l\phi-a)\widetilde g_l\|_{L_p(Q_{r_L}(x),\gamma)}
        \le
        C2^{-k_l\frac{\theta}{p}}
        \|\phi-a\|_{L_p(Q^E_{\beta r_L}(x),{\bar{\gamma}})}.
    \end{equation}
\end{Lm}
\begin{proof}
Since \(x\notin S\) and \(S\) is finite, we have $\operatorname{dist}(x,S)>0$.
Choose \(l(x)\) so large that, for all \(L\ge l(x)\), $C_0 2^{-k_L}
    <\frac12\operatorname{dist}(x,S)$,
where \(C_0>0\) is a fixed geometric constant larger than all constants appearing below. Then, if \(l\ge L\ge l(x)\), \(Q\in\widetilde{\mathcal D}_l(\phi)\), and
\(c_\eta Q\cap Q_{r_L}(x)\ne\emptyset\), we have $\operatorname{dist}(Q,S)
    \ge
    \frac12\operatorname{dist}(x,S)
    >
    \sigma2^{-k_l}$. Hence \(Q\in\widetilde{\mathcal D}_l^r(\phi)\).
\par
We now prove \eqref{eq.ext_inv_est}. By \eqref{eq.cut_off_shift}, the condition
\(\widetilde g_l(y)\ne0\) implies \(g_{k_l}(y)=1\). Hence
\begin{equation}
    (\mathcal E_l\phi(y)-a)\widetilde g_l(y)
    =
    \sum_{Q\in\widetilde{\mathcal D}_l(\phi)}
    (a_Q(\phi)-a)\psi_Q(y)\widetilde g_l(y).
\end{equation}
Using the bounded multiplicity of the family \(\{c_\eta Q:Q\in\mathcal D_{k_l}\}\), we obtain
\begin{equation}
\label{eq.ext_inv_aux1}
\begin{split}
    \|(\mathcal E_l\phi-a)\widetilde g_l\|_{L_p(Q_{r_L}(x),\gamma)}^p
    \lesssim
    \sum_{\substack{Q\in\widetilde{\mathcal D}_l(\phi):\\
    c_\eta Q\cap Q_{r_L}(x)\ne\emptyset}}
    |a_Q(\phi)-a|^p{\gamma}(c_\eta Q).
\end{split}
\end{equation}
By the first part of the proof, all cubes in the last sum are regular. Therefore,
using Jensen's inequality, the local doubling property of $\gamma$, and Lemma~\ref{Lm.ext_f_pr}(2), we obtain
\begin{equation}
    |a_Q(\phi)-a|^p\bm{\gamma}(c_\eta Q)
    \lesssim
    2^{-k_l\theta}
    \int_{\widehat Q}
    |\phi(y)-a|^p\,d\bm{\bar{\gamma}}(y).
\end{equation}
Moreover, by elementary geometry, there exists \(\beta>0\), depending only on the fixed parameters, such that
\begin{equation}
    \bigcup_{\substack{Q\in\widetilde{\mathcal D}_l(\phi):\\
    c_\eta Q\cap Q_{r_L}(x)\ne\emptyset}}
    \widehat Q
    \subset
    Q^E_{\beta r_L}(x).
\end{equation}
Substituting these estimates into \eqref{eq.ext_inv_aux1} and using and the bounded multiplicity of the patches \(\widehat Q\), we obtain \eqref{eq.ext_inv_est}. The proof is complete.
\end{proof}
\begin{Prop}
    \label{Prop.ext_inv}
    For each \(\phi\in\mathfrak L_p(E,{\bar{\gamma}},S)\),
    \begin{equation}
        \operatorname{Tr}(\operatorname{Ext}\phi)=\phi
    \end{equation}
    for \(\mathcal H^d\lfloor_E\)-almost every point of \(E\).
\end{Prop}
\begin{proof}
Fix \(\phi\in\mathfrak L_p(E,{\bar{\gamma}},S)\), and put $f:=\operatorname{Ext}\phi$.
By Proposition~\ref{Prop.ext_bound}, we have $f\in B^{\frac{\theta}{p}}_{p,1}(\mathbb R^n,\gamma)$.
Hence, by Theorem~\ref{Th.tr_exist_stat}, the trace \(\operatorname{Tr}f\) exists.
It is enough to prove that, for \(\mathcal H^d\lfloor_E\)-almost every
\(x\in E\setminus S\),
\begin{equation}
\label{eq.tr_rev_inv_aux1}
    \lim_{L\rightarrow\infty}\fint\limits_{Q_{r_L}(x)}
    |f(y)-\phi(x)|^p\,d\bm{\gamma}(y)=0.
\end{equation}
Indeed, by the weighted averaging inequality \eqref{eq.equiv_muck},
\eqref{eq.tr_rev_inv_aux1} implies
\begin{equation}
    \lim_{L\rightarrow\infty}\fint\limits_{Q_{r_L}(x)}
    |f(y)-\phi(x)|\,dy=0.
\end{equation}
On the other hand, since \(\operatorname{Tr}f\) exists,
\begin{equation}
    \lim_{L\rightarrow\infty}\fint\limits_{Q_{r_L}(x)}
    |f(y)-\operatorname{Tr}f(x)|\,dy=0.
\end{equation}
for \(\mathcal H^d\lfloor_E\)-almost every \(x\in E\). Therefore, $\operatorname{Tr}f(x)=\phi(x)$ for \(\mathcal H^d\lfloor_E\)-almost every \(x\in E\).
\par
It remains to prove \eqref{eq.tr_rev_inv_aux1}. Let \(x\in E\setminus S\) be a
Lebesgue point of \(\phi\) with respect to the measure \(\bm{\bar{\gamma}}\). This holds for \(\mathcal H^d\lfloor_E\)-almost every \(x\in E\setminus S\), because
\(\bar{\gamma}(x)>0\) for \(\mathcal H^d\lfloor_E\)-almost every \(x\in E\). Choose \(l(x)\) as in Lemma~\ref{Lm.ext_inv_est}, and let \(L\ge l(x)\). Since $r_L=\rho_-2^{-(k_L+m_{\operatorname{cut}})}$,
Lemma~\ref{Lm.lay}(3) gives $\widetilde g_L(y)=1$ for all $y\in Q_{r_L}(x)$.
Moreover, since $k_{l+1}\ge k_l+m_{cut}$, we have $\widetilde g_{l+1}\le \widetilde g_l$ for all $l\ge0$.
\par
For \(\mathcal L^n\)-almost every \(y\in Q_{r_L}(x)\), the defining series for
\(f(y)\) is locally finite. Using \(\widetilde g_L(y)=1\), we write
\begin{equation}
\label{eq.ext_inv1}
\begin{split}
    f(y)
    &=
    \mathcal E_L\phi(y)
    +
    \sum_{l=L+1}^{\infty}
    \mathcal E_l\phi(y)\widetilde g_l(y)
    -
    \sum_{l=L}^{\infty}
    \mathcal E_l\phi(y)\widetilde g_{l+1}(y).
\end{split}
\end{equation}
Since
\begin{equation}
    \sum_{l=L+1}^{\infty}\widetilde g_l(y)
    -
    \sum_{l=L}^{\infty}\widetilde g_{l+1}(y)
    =
    0,
\end{equation}
we obtain, for such \(y\),
\begin{equation}
\label{eq.ext_inv2}
\begin{split}
    |f(y)-\phi(x)|
    &\le
    |\mathcal E_L\phi(y)-\phi(x)|
    +
    \sum_{l=L+1}^{\infty}
    |\mathcal E_l\phi(y)-\phi(x)|\widetilde g_l(y)
    \\
    &\quad+
    \sum_{l=L}^{\infty}
    |\mathcal E_l\phi(y)-\phi(x)|\widetilde g_{l+1}(y).
\end{split}
\end{equation}
Applying Lemma~\ref{Lm.ext_inv_est} with \(a=\phi(x)\), and using
\(\widetilde g_{l+1}\le\widetilde g_l\), we get
\begin{equation}
\label{eq.ext_inv3}
    \|f-\phi(x)\|_{L_p(Q_{r_L}(x),\gamma)}
    \lesssim
    2^{-k_L\frac{\theta}{p}}
    \|\phi-\phi(x)\|_{L_p(Q^E_{\beta r_L}(x),{\bar{\gamma}})}.
\end{equation}
Here we used that the sequence \(k_l\) is strictly increasing, and hence
\begin{equation}
    \sum_{l=L}^{\infty}2^{-k_l\frac{\theta}{p}}
    \lesssim
    2^{-k_L\frac{\theta}{p}}.
\end{equation}

Dividing \eqref{eq.ext_inv3} by \(\bm{\gamma}(Q_{r_L}(x))^{\frac1p}\), and using the almost regularity condition together with the local doubling properties of
\(\bm{\gamma}\) and \(\bm{\bar{\gamma}}\) away from \(S\), we obtain
\begin{equation}
\label{eq.ext_inv4}
\begin{split}
    \fint\limits_{Q_{r_L}(x)}
    |f(y)-\phi(x)|^p\,d\bm{\gamma}(y)
    \lesssim
    \fint\limits_{Q^E_{\beta r_L}(x)}
    |\phi(y)-\phi(x)|^p\,d\bm{\bar{\gamma}}(y).
\end{split}
\end{equation}
Since \(x\) is a \(\bm{\bar{\gamma}}\)-Lebesgue point of \(\phi\), the right-hand side tends to \(0\) as \(L\to\infty\). This proves
\eqref{eq.tr_rev_inv_aux1}, and hence $\operatorname{Tr}(\operatorname{Ext}\phi)=\phi$
for \(\mathcal H^d\lfloor_E\)-almost every point of \(E\).
\end{proof}
\begin{Remark}
    Since both \(\operatorname{Tr}f\) and \(\phi\) belong to \(\mathfrak{L}_p(E,{\bar{\gamma}},S)\), and since they coincide $\mathcal{H}^d\lfloor_E$-almost everywhere on \(E\setminus S\), their generalized weighted values at points of \(S\) also coincide. Indeed, if \(x_0\in S\), then
\begin{equation}
    (\operatorname{Tr}f-\operatorname{Tr}f(x_0))
    -
    (\phi-\phi(x_0))
\in L_p(Q_{\rho_0}^E(x_0),{\bar{\gamma}}),
\end{equation}
while this function is equal to the constant
\(\phi(x_0)-\operatorname{Tr}f(x_0)\) $\mathcal{H}^d\lfloor_E$-almost everywhere on \(Q_{\rho_0}^E(x_0)\setminus\{x_0\}\). Since
\(\bm{\bar\gamma}(Q_r^E(x_0))=\infty\) for every \(r>0\), this constant must be zero.
\end{Remark}

\section{Examples}
In this section, we present several examples of weights satisfying the assumptions of the main theorem and illustrate the corresponding trace spaces.
\begin{Example}
    Let \(p\in[1,\infty)\), and let \(E\subset\mathbb R^n\) be an Ahlfors --David \(d\)-regular set, where \(d\in(0,n)\). Fix \(x_0\in E\), and define \( \gamma(x):=|x-x_0|^\alpha\), \(x\in\mathbb R^n, \) where \[ -n<\alpha<n(p-1) \quad\text{if }p>1, \qquad -n<\alpha\le0 \quad\text{if }p=1. \] Then \(\gamma\in A_p^{\operatorname{loc}}(\mathbb R^n)\).
    On \(E\), we set $\bar{\gamma}(x):=|x-x_0|^\alpha$.
    For \(x\in E\setminus\{x_0\}\) and $0<r\le \min\left\{\frac12|x-x_0|,1\right\}$,
    we have
    \begin{equation}
    \label{eq.ex1}
        \bm{\gamma}(Q_r(x))
        \approx
        r^n|x-x_0|^\alpha,
        \qquad
        \bm{\bar{\gamma}}(Q_r^E(x))
        \approx
        r^d|x-x_0|^\alpha.
    \end{equation}
    Put $\theta:=n-d$, then \eqref{eq.ex1} gives $\bm{\gamma}(Q_r(x))
        \approx
        r^\theta\bm{\bar{\gamma}}(Q_r^E(x))$
    away from \(x_0\).
\par
    At the point \(x_0\), we have
    \begin{equation}
        \frac{\bm{\gamma}(Q_r(x_0))}{r^\theta}
        \approx
        r^{\alpha+d}.
    \end{equation}
    Therefore, if \(\alpha<-d\), then \(x_0\) is a rapid singular point. If \(\alpha>-d\), then \(\bar{\gamma}\in L_1^{\operatorname{loc}}(E)\), and there is no rapid singularity. Thus, for \(\alpha\neq-d\)
    \begin{equation}
        RS_{p, \theta}(\gamma)
        =
        \begin{cases}
            \{x_0\}, & \alpha<-d,\\
            \emptyset, & \alpha>-d.
        \end{cases}
    \end{equation}
    If, in addition, $\theta < p$, then, for \(\alpha<-d\),
    \begin{equation}
        \operatorname{Tr}
        \bigl(
        B^{\frac{\theta}{p}}_{p,1}(\mathbb R^n,\gamma)
        \bigr)
        =
        \mathfrak L_p(E,{\bar{\gamma}},\{x_0\}),
    \end{equation}
    whereas, for \(\alpha>-d\),
    \begin{equation}
        \operatorname{Tr}
        \bigl(
        B^{\frac{\theta}{p}}_{p,1}(\mathbb R^n,\gamma)
        \bigr)
        =
        \mathfrak L_p(E,{\bar{\gamma}},\emptyset)
        =
        L_p(E,{\bar{\gamma}}).
    \end{equation}
\par
    In particular, if \(E=\mathbb R^{n-1}\), \(x_0=0\), \(p>1\), and $-(n-1)<\alpha<n(p-1)$,
    then \(\alpha>-d\), where \(d=n-1\). Thus we recover the result of Haroske and Schmeisser \cite{Har_Sch}:
    \begin{equation}
        \operatorname{Tr}
        \bigl(
        B^{\frac1p}_{p,1}(\mathbb R^n,|x|^\alpha)
        \bigr)
        =
        L_p(\mathbb R^{n-1},|x'|^\alpha).
    \end{equation}
    On the other hand, the range $ -n<\alpha<-(n-1)$
    is also covered by the present result and corresponds to the rapid singular case.
\end{Example}

\begin{Example}
    Let \(d\in(0,n)\), let \(E\subset\mathbb R^n\) be an Ahlfors--David \(d\)-regular set, and let \(\alpha\in\mathbb R\). Define
    \begin{equation}
        \gamma(x)
        :=
        \begin{cases}
            \operatorname{dist}(x,E)^\alpha,
            & \operatorname{dist}(x,E)\le1,\\
            1,
            & \operatorname{dist}(x,E)>1.
        \end{cases}
    \end{equation}
    Assume that
    \[ -(n-d)<\alpha<(n-d)(p-1) \quad\text{if }p>1, \qquad -(n-d)<\alpha\le0 \quad\text{if }p=1. \]
    Then \(\gamma\in A_p^{\operatorname{loc}}(\mathbb R^n)\). For \(x\in E\) and \(0<r\le1\), we have
    \begin{equation}
        \bm{\gamma}(Q_r(x))
        \approx
        r^{n+\alpha}.
    \end{equation}
    Put $\theta:=n-d+\alpha$,
    then
    \begin{equation}
        \frac{\bm{\gamma}(Q_r(x))}{r^\theta}
        \approx
        r^d
        \approx
        \mathcal H^d(Q_r^E(x)).
    \end{equation}
    Hence \(E\) is Ahlfors--David codimension-\(\theta\) regular with respect to \(\gamma\), with boundary weight \(\bar{\gamma}\equiv1\). Therefore, if, in addition, $\theta \in (0, p)$, then the main theorem gives
    \begin{equation}
        \operatorname{Tr}
        \bigl(
        B^{\frac{n-d+\alpha}{p}}_{p,1}(\mathbb R^n,\gamma)
        \bigr)
        =
        L_p(E).
    \end{equation}
    Thus, in the endpoint case \(q=1\), we recover part of the result of Piotrowska~\cite{Iwon}. We note that, due to the special form of the distance weight, Piotrowska obtained corresponding trace results under weaker restrictions on the parameters.
\end{Example}
\subsection*{Acknowledgments}
\addtocontents{toc}{\protect\setcounter{tocdepth}{0}}
The author is deeply grateful to his Scientific Advisor Alexander I. Tyulenev for proposing the problem, and for his guidance and attention throughout this work.
\par

\renewcommand{\bibname}{References}
\bibliographystyle{plain}
\bibliography{refs.bib}

@article{Gagl,
  author   = {Gagliardo, Emilio},
  title    = {Caratterizzazioni delle tracce sulla frontiera relative ad alcune classi di funzioni in {$n$} variabili},
  journal  = {Rend. Semin. Mat. Univ. Padova},
  fjournal = {Rendiconti del Seminario Matematico della Universit{\`a} di Padova},
  year     = {1957},
  volume   = {27},
  pages    = {284--305},
  url      = {https://eudml.org/doc/106977},
  language = {Italian}
}

@article{Bur_art,
  author   = {Burenkov, V. I. and Gol'dman, M. L.},
  title    = {On the extension of functions from {$L_p$}},
  journal  = {Tr. Mat. Inst. Steklova},
  fjournal = {Trudy Matematicheskogo Instituta Imeni V. A. Steklova},
  issn     = {0371-9685},
  volume   = {150},
  pages    = {31--51},
  year     = {1979},
  language = {Russian}
}

@article{Har_Sch,
  author    = {Haroske, Dorothee D. and Schmeisser, Hans-J{\"u}rgen},
  title     = {On trace spaces of function spaces with a radial weight: the atomic approach},
  journal   = {Complex Variables and Elliptic Equations},
  volume    = {55},
  number    = {8-10},
  pages     = {875--896},
  year      = {2010},
  publisher = {Taylor \& Francis},
  doi       = {10.1080/17476930903276050},
  url       = {https://doi.org/10.1080/17476930903276050}
}

@article{Tyul_b,
  author   = {Tyulenev, A. I.},
  title    = {Some new function spaces of variable smoothness},
  journal  = {Sbornik: Mathematics},
  year     = {2015},
  volume   = {206},
  number   = {6},
  pages    = {849--891},
  doi      = {10.1070/SM2015v206n06ABEH004481},
  url      = {http://mi.mathnet.ru/eng/sm8399},
  mrnumber = {3438581},
  zbl      = {1364.46034}
}

@article{Tyul_s,
  author   = {Tyulenev, A. I.},
  title    = {Traces of weighted {Sobolev} spaces with {Muckenhoupt} weight. The case {$p=1$}},
  journal  = {Nonlinear Analysis},
  volume   = {128},
  pages    = {248--272},
  year     = {2015},
  issn     = {0362-546X},
  doi      = {10.1016/j.na.2015.08.001},
  url      = {https://www.sciencedirect.com/science/article/pii/S0362546X15002679}
}

@article{GKZ, author = {Gogatishvili, Amiran and Koskela, Pekka and Zhou, Yuan}, title = {Characterizations of {B}esov and {T}riebel--{L}izorkin spaces on metric measure spaces}, journal = {Forum Mathematicum}, volume = {25}, number = {4}, pages = {787--819}, year = {2013}, doi = {10.1515/form.2011.135}, eprint = {1106.2561}, archivePrefix = {arXiv}, primaryClass = {math.CA}, url = {https://arxiv.org/abs/1106.2561} }

@book{Stein,
  author    = {Stein, Elias M.},
  title     = {Harmonic Analysis: {Real}-Variable Methods, Orthogonality, and Oscillatory Integrals},
  series    = {Princeton Mathematical Series},
  volume    = {43},
  isbn      = {0-691-03216-5},
  year      = {1993},
  publisher = {Princeton University Press},
  address   = {Princeton, NJ},
  language  = {English},
  zbl       = {0821.42001}
}

@article{Gul,
  author   = {Gulisashvili, A. B.},
  title    = {Traces of differential functions on subsets of the {Euclidean} space},
  journal  = {J. Sov. Math.},
  fjournal = {Journal of Soviet Mathematics},
  issn     = {0090-4104},
  volume   = {42},
  number   = {2},
  pages    = {1573--1583},
  year     = {1988},
  doi      = {10.1007/BF01665043},
  language = {English}
}

@article{Sak_Sot,
  author  = {Saksman, Eero and Soto, Tom{\'a}s},
  title   = {Traces of {Besov}, {Triebel--Lizorkin} and {Sobolev} Spaces on {Metric} Spaces},
  journal = {Analysis and Geometry in Metric Spaces},
  volume  = {5},
  number  = {1},
  pages   = {98--115},
  year    = {2017},
  doi     = {10.1515/agms-2017-0006},
  url     = {https://doi.org/10.1515/agms-2017-0006}
}

@misc{Luk_e,
  author = {Mal{\'y}, Luk{\'a}{\v s}},
  title  = {Trace and extension theorems for {S}obolev-type functions in metric spaces},
  year   = {2017},
  note   = {\href{https://arxiv.org/abs/1704.06344}{arXiv:1704.06344}}
}

@book{Jon,
  author       = {Jonsson, Alf and Wallin, Hans},
  title        = {Function Spaces on Subsets of {$\mathbb R^n$}},
  series       = {Mathematical Reports},
  volume       = {2},
  number       = {1},
  year         = {1984},
  publisher    = {Harwood Academic Publishers},
  address      = {Chur},
  note         = {xiv+221 pp.}
}

@book{Trieb,
  author    = {Triebel, Hans},
  title     = {Theory of Function Spaces},
  series    = {Mathematik und ihre Anwendungen in Physik und Technik},
  volume    = {38},
  year      = {1983},
  publisher = {Akademische Verlagsgesellschaft Geest \& Portig K.-G.},
  address   = {Leipzig},
  language  = {English},
  zbl       = {0546.46028}
}

@misc{Har_tr,
  author        = {Besoy, Blanca F. and Haroske, Dorothee D. and Triebel, Hans},
  title         = {Traces of some weighted function spaces and related non-standard real interpolation of {Besov} spaces},
  year          = {2020},
  eprint        = {2009.03656},
  archivePrefix = {arXiv},
  primaryClass  = {math.FA},
  url           = {https://arxiv.org/abs/2009.03656}
}

@article{Fraz,
  author  = {Frazier, Michael},
  title   = {Decomposition and traces of weighted mixed-norm {Besov} spaces},
  journal = {La Matematica},
  year    = {2024},
  volume  = {3},
  number  = {4},
  pages   = {1320--1359},
  doi     = {10.1007/s44007-024-00138-6},
  url     = {https://doi.org/10.1007/s44007-024-00138-6},
  issn    = {2730-9657}
}

@book{Bur_book,
  author    = {Burenkov, Victor I.},
  title     = {Sobolev Spaces on Domains},
  series    = {Teubner-Texte zur Mathematik},
  volume    = {137},
  issn      = {0138-502X},
  isbn      = {3-8154-2068-7},
  year      = {1998},
  publisher = {B. G. Teubner},
  address   = {Stuttgart},
  language  = {English},
  zbl       = {0893.46024}
}

@book{law,
  author    = {Evans, Lawrence C. and Gariepy, Ronald F.},
  title     = {Measure Theory and Fine Properties of Functions},
  publisher = {CRC Press},
  year      = {1992},
  isbn      = {0-8493-7157-0},
  address   = {Boca Raton}
}

@book{kuf,
  author    = {Kufner, Alois and Maligranda, Lech and Persson, Lars-Erik},
  title     = {The {Hardy} Inequality: About Its History and Some Related Results},
  isbn      = {978-80-86843-15-5},
  year      = {2007},
  publisher = {Vydavatelsk{\'y} Servis},
  address   = {Pilsen},
  language  = {English}
}

@article{Bes_or,
  author   = {Besov, O. V.},
  title    = {Investigation of a family of function spaces in connection with theorems of imbedding and extension},
  journal  = {Transl., Ser. 2, Am. Math. Soc.},
  fjournal = {Translations. Series 2. American Mathematical Society},
  issn     = {0065-9290},
  volume   = {40},
  pages    = {85--126},
  year     = {1964},
  language = {English},
  doi      = {10.1090/trans2/040/03},
  zbl      = {0158.13901}
}

@article{tyul_s_n,
  author   = {Tyulenev, A. I.},
  title    = {Description of traces of functions in the {Sobolev} space with a {Muckenhoupt} weight},
  journal  = {Proc. Steklov Inst. Math.},
  fjournal = {Proceedings of the Steklov Institute of Mathematics},
  issn     = {0081-5438},
  volume   = {284},
  pages    = {280--295},
  year     = {2014},
  language = {English},
  doi      = {10.1134/S0081543814010209},
  zbl      = {1320.46034}
}

@article{Iwon,
  author   = {Piotrowska, Iwona},
  title    = {Traces on fractals of function spaces with {Muckenhoupt} weights},
  journal  = {Funct. Approx. Comment. Math.},
  fjournal = {Functiones et Approximatio Commentarii Mathematici},
  issn     = {0208-6573},
  volume   = {36},
  pages    = {95--117},
  year     = {2006},
  language = {English},
  doi      = {10.7169/facm/1229616444},
  zbl      = {1140.46015}
}

@article{Peetre,
  author   = {Peetre, Jaak},
  title    = {A counter-example connected with {Gagliardo}'s trace theorem},
  journal  = {Commentationes Mathematicae. Special Issue},
  fjournal = {Annales Societatis Mathematicae Polonae},
  issn     = {0373-8299},
  pages    = {277--282},
  year     = {1979},
  language = {English},
  zbl      = {0442.46026}
}

@article{koch_sn,
  author    = {Kazaniecki, Krystian and Wojciechowski, Micha{\l}},
  title     = {Trace operator on von {Koch}'s snowflake},
  journal   = {Potential Analysis},
  year      = {2024},
  volume    = {61},
  pages     = {659--684},
  doi       = {10.1007/s11118-024-10124-w},
  url       = {https://doi.org/10.1007/s11118-024-10124-w},
  publisher = {Springer},
  issn      = {0926-2601}
}

@article{Ihnat,
  author   = {Ihnatsyeva, Lizaveta and V{\"a}h{\"a}kangas, Antti V.},
  title    = {Characterization of traces of smooth functions on {Ahlfors} regular sets},
  journal  = {Journal of Functional Analysis},
  volume   = {265},
  number   = {9},
  pages    = {1870--1915},
  year     = {2013},
  issn     = {0022-1236},
  doi      = {10.1016/j.jfa.2013.07.006},
  url      = {https://www.sciencedirect.com/science/article/pii/S0022123613002656}
}

@article{mig,
  author    = {Marcos, Miguel Andr{\'e}s},
  title     = {A trace theorem for {Besov} functions in spaces of homogeneous type},
  journal   = {Publicacions Matem{\`a}tiques},
  volume    = {62},
  number    = {1},
  pages     = {185--211},
  year      = {2018},
  publisher = {Universitat Aut{\`o}noma de Barcelona, Departament de Matem{\`a}tiques},
  doi       = {10.5565/PUBLMAT6211810},
  url       = {https://doi.org/10.5565/PUBLMAT6211810}
}

@article{Rych,
  author   = {Rychkov, Vyacheslav S.},
  title    = {{Littlewood--Paley} theory and function spaces with {$A_p^{\operatorname{loc}}$} weights},
  journal  = {Mathematische Nachrichten},
  volume   = {224},
  number   = {1},
  pages    = {145--180},
  year     = {2001},
  doi      = {10.1002/1522-2616(200104)224:1<145::AID-MANA145>3.0.CO;2-2},
  url      = {https://onlinelibrary.wiley.com/doi/abs/10.1002/1522-2616%28200104%29224%3A1%3C145%3A%3AAID-MANA145%3E3.0.CO%3B2-2}
}

@misc{tyulenev2026, author = {Tyulenev, Alexander}, title = {Descriptions of traces of weighted {S}obolev spaces to {A}hlfors--{D}avid regular sets in the case {$p=1$}}, year = {2026}, note = { \href{https://arxiv.org/abs/2606.11924}{arXiv:2606.11924}} }

@article{Gol,
  author  = {Gol'dman, M. L.},
  title   = {On the extension of functions from {$L_p(\mathbb R^m)$} to a space of higher dimension},
  journal = {Mathematical Notes},
  year    = {1979},
  volume  = {25},
  number  = {4},
  pages   = {266--270},
  doi     = {10.1007/BF01688477},
  note    = {Russian original: Mat. Zametki 25 (1979), no. 4, 513--520}
}

@article{shanm, author = {Lindquist, Jeff and Shanmugalingam, Nageswari}, title = {Traces and extensions of certain weighted {S}obolev spaces on {$\mathbb{R}^n$} and {B}esov functions on {A}hlfors regular compact subsets of {$\mathbb{R}^n$}}, journal = {Complex Analysis and its Synergies}, volume = {7}, number = {1}, pages = {Paper No. 7, 12 pp.}, year = {2021}, doi = {10.1007/s40627-021-00064-1}, eprint = {2007.10580}, archivePrefix = {arXiv}, primaryClass = {math.CV}, url = {https://arxiv.org/abs/2007.10580} }

@article{Costea,
  author    = {{\c S}erban Costea},
  title     = {{Besov} capacity and {Hausdorff} measures in metric measure spaces},
  journal   = {Publicacions Matem{\`a}tiques},
  volume    = {53},
  number    = {1},
  pages     = {141--178},
  year      = {2009}
}

@article{Nuutinen,
  author  = {Nuutinen, Juho},
  title   = {The {Besov} capacity in metric spaces},
  journal = {Annales Polonici Mathematici},
  volume  = {117},
  number  = {1},
  pages   = {59--78},
  year    = {2016},
  issn    = {1730-6272},
  doi     = {10.4064/ap3843-4-2016},
  url     = {https://doi.org/10.4064/ap3843-4-2016}
}

@article{Li,
  author  = {Li, Ziwei and Yang, Dachun and Yuan, Wen},
  title   = {Lebesgue points of {Besov} and {Triebel--Lizorkin} spaces with generalized smoothness},
  journal = {Mathematics},
  volume  = {9},
  number  = {21},
  pages   = {2724},
  year    = {2021},
  issn    = {2227-7390},
  doi     = {10.3390/math9212724},
  url     = {https://www.mdpi.com/2227-7390/9/21/2724}
}

@article{Netrusov,
  author   = {Netrusov, Yu. V.},
  title    = {Sets of singularities of functions in spaces of {Besov} and {Lizorkin--Triebel} type},
  journal  = {Trudy Mat. Inst. Steklov.},
  volume   = {187},
  year     = {1989},
  pages    = {162--177},
  note     = {English translation: Proc. Steklov Inst. Math. 187 (1990), 185--203},
  mrnumber = {1006450},
  zbl      = {0719.46018}
}

@article{HumLin,
  author  = {Hummel, Felix and Lindemulder, Nick},
  title   = {Elliptic and Parabolic Boundary Value Problems in Weighted Function Spaces},
  journal = {Potential Analysis},
  volume  = {57},
  pages   = {601--669},
  year    = {2022},
  doi     = {10.1007/s11118-021-09929-w}
}

@book{Sawano,
 author = {Sawano, Yoshihiro},
 title = {Theory of {Besov} spaces},
 fseries = {Developments in Mathematics},
 series = {Dev. Math.},
 issn = {1389-2177},
 volume = {56},
 isbn = {978-981-13-0835-2; 978-981-13-0836-9},
 year = {2018},
 publisher = {Singapore: Springer},
 language = {English},
 doi = {10.1007/978-981-13-0836-9},
 zbMATH = {6909721},
 Zbl = {1414.46004}
}

@article{Fabes,
author = {Eugene B. Fabes and Carlos E. Kenig and Raul P. Serapioni},
title = {The local regularity of solutions of degenerate elliptic equations},
journal = {Communications in Partial Differential Equations},
volume = {7},
number = {1},
pages = {77--116},
year = {1982},
publisher = {Taylor \& Francis},
doi = {10.1080/03605308208820218},
URL = {https://doi.org/10.1080/03605308208820218},
eprint = {https://doi.org/10.1080/03605308208820218}
}
\end{document}